\documentclass[reqno]{amsart}
\usepackage{url}
\usepackage{amssymb,amsthm,amsfonts,amstext,amsmath}
\usepackage{pdfsync}  
\usepackage{mathptmx}       
\usepackage{newcent}       
\usepackage{helvet}         
\usepackage{courier}        
\usepackage{graphicx}
\usepackage{enumerate}
\usepackage{hyperref}
\usepackage{tikz-cd}
\usepackage{multicol}
\usepackage{verbatim}
\usepackage{a4wide}
\numberwithin{equation}{section}
\usepackage{color}

\usepackage[normalem]{ulem} 

\usepackage[active]{srcltx}
\usepackage{pstricks}
\usepackage{float}
\usepackage{mathrsfs}
\usepackage[charter]{mathdesign}
\usepackage{tikz}
\usepackage{bbm}

\newtheorem{theorem}{Theorem}[section]
\newtheorem{lemma}[theorem]{Lemma}
\newtheorem{proposition}[theorem]{Proposition}
\newtheorem{corollary}[theorem]{Corollary}
\newtheorem{remark}[theorem]{Remark}
\newtheorem{definition}{Definition}
\newcommand{\mc}[1]{{\mathcal #1}}

\newcommand{\bb}[1]{{\mathbb #1}}

\newcommand{\<}{\langle}
\renewcommand{\>}{\rangle}
\newcommand{\p}{\partial}
\newcommand{\pfrac}[2]{\genfrac{}{}{}{1}{#1}{#2}}

\let\oldtocsection=\tocsection
\let\oldtocsubsection=\tocsubsection
\let\oldtocsubsubsection=\tocsubsubsection
\renewcommand{\tocsection}[2]{\hspace{0em}\oldtocsection{#1}{#2}}
\renewcommand{\tocsubsection}[2]{\hspace{1em}\oldtocsubsection{#1}{#2}}
\renewcommand{\tocsubsubsection}[2]{\hspace{2em}\oldtocsubsubsection{#1}{#2}}
\DeclareRobustCommand{\SkipTocEntry}[5]{}

\newcommand\ve{\varepsilon}

\keywords{Porous medium model, Hydrodynamic limit, Porous medium equation, Boundary conditions.}

\begin{document}

\title[Hydrodynamics of  porous medium model with slow reservoirs]{Hydrodynamics of  porous medium model with slow reservoirs}

\author{L. Bonorino}
\address{UFRGS, Instituto de Matem\'atica e Estat\'istica, Campus do Vale, Av. Bento Gon\c calves, 9500. CEP 91509-900, Porto Alegre, Brasil}
\curraddr{}
\email{bonorino@mat.ufrgs.br}
\thanks{}

\author{R. de Paula}
\address{Center for Mathematical Analysis,  Geometry and Dynamical Systems, Instituto Superior T\'ecnico, Universidade de Lisboa, Av. Rovisco Pais, 1049-001 Lisboa, Portugal.}
\curraddr{}
\email{renato.paula@tecnico.ulisboa.pt}
\thanks{}

\author{P. Gon\c calves}
\address{Center for Mathematical Analysis,  Geometry and Dynamical Systems, Instituto Superior T\'ecnico, Universidade de Lisboa, Av. Rovisco Pais, 1049-001 Lisboa, Portugal.}
\curraddr{}
\email{pgoncalves@tecnico.ulisboa.pt}
\thanks{}

\author{A. Neumann}
\address{UFRGS, Instituto de Matem\'atica e Estat\'istica, Campus do Vale, Av. Bento Gon\c calves, 9500. CEP 91509-900, Porto Alegre, Brasil}
\curraddr{}
\email{aneumann@mat.ufrgs.br}
\thanks{}


\maketitle
\begin{abstract}
We analyze the hydrodynamic behavior of the porous medium model (PMM) in a discrete space $\{0,\ldots, n\}$, where the sites $0$ and $n$ stand for reservoirs. Our strategy relies on the entropy method of Guo, Papanicolau and Varadhan \cite{entropy}. However, this method cannot be straightforwardly applied, since there are configurations that do not evolve according to the dynamics (blocked configurations). In order to avoid this problem, we slightly perturbed the dynamics in such a way that the macroscopic behavior of the system keeps following the porous medium equation (PME), but with boundary conditions which depend on the reservoirs strength's.
\end{abstract}

\section{Introduction}
\noindent

In recent years, there has been an intensive research activity around the derivation of partial differential equations (PDEs)  with boundary conditions from interacting particle systems \cite{kl}. 
This derivation is known as \emph{hydrodynamic limit}, which consists in proving, rigorously, that the conserved quantities of a random microscopic dynamics are described by the solution of some PDE. Therefore, this PDE coins  the name \emph{hydrodynamic equation}. The aforementioned procedure, consists, probabilistically speaking, in a Law of Large Numbers for the empirical measure associated to the conserved quantities of the system. More recently, there has been quite a lot of attention devoted to the analysis of microscopic systems with local perturbations, and one of the puzzling questions is to see whether these perturbations have an impact at the macroscopic behavior of the system. Usually, these perturbations, being local, do not destroy the nature of the PDE, but instead they bring up additional boundary conditions to the PDE, see for instance \cite{patricia1} and references therein.\\

 In light of these questions, in this paper we present the derivation of the porous medium equation (PME) with boundary conditions from a microscopic dynamics which is placed in contact with reservoirs.  
 Up to our knowledge, this is the first derivation of a nonlinear degenerate PDE with boundary conditions which can be obtained as the hydrodynamic limit of an underlying microscopic random dynamics. More specifically, we obtain three different types of boundary conditions (Dirichlet, Robin, and Neumann) depending on the intensity of the rate at the reservoir's dynamics. 
 We remark however that the first microscopic derivation of the PME was obtained in \cite{seppa1} and \cite{seppa2}, in which the authors considered a model with continuous occupational variables. The first microscopic derivation considering discrete occupational variables was obtained in \cite{patricia}. 
There, the authors considered the porous medium model (PMM) evolving in the discrete d-dimensional torus $\mathbb{T}_{n}^{d}$ without the presence of reservoirs and therefore, the PME did not have any type of boundary conditions.
 The article \cite{patricia} motivated us  to work with discrete occupational variables  in order to derive the PME, that is, to consider as the random  microscopic  dynamics, an ad-hoc  version of the PMM analyzed there. With the aim to derive  boundary conditions in the PME, we combined the microscopic dynamics of \cite{patricia} with the boundary dynamics of \cite{bmns}. In the latter article, the dynamics at the bulk was given by the simple symmetric exclusion process (SSEP), then the authors obtained the heat equation with different types of boundary conditions, namely Dirichlet, Robin, and Neumann. \\
 
Now we describe precisely what is the random dynamics that we analyze in this article: the \textit{PMM} \textit{with slow reservoirs}. First, we fix the discrete space where the particle will be moving around, that is, the space $\Sigma_n=\{1,\ldots, n-1\}$, which we call \emph{bulk}. For $x\in \Sigma_n$, the occupation variable $\eta(x)$ takes values in $\{0,1\}$ and $\eta(x)=0$ (resp. $\eta(x)=1$) stands for empty (resp.  occupied) site. The configuration of particles, that we denote by $\eta$ is, therefore, an element of $\{0,1\}^{\Sigma_n}$. The PMM is an exclusion process (since only one particle per site is allowed) with dynamical constraints, i.e., a particle at site $x$ can jump to $x+1$, if and only if there is at least one particle at sites $x-1$ or $x+2$. The jump rate is given by the sum of the number of particles at sites $x-1$ and $x+2$. Due to the constraint of the model's rates, and since one can have configurations in which the distance between two successive particles is larger than two, the model exhibits the so-called \textit{blocked configurations}, that is, configurations that do not evolve under the dynamics. To avoid them, we superpose the PMM dynamics with the dynamics of the SSEP  on the bulk in such a way that the macroscopic hydrodynamic behavior of this perturbed dynamics still evolves according to the PME. This means that when scaling the time diffusively, we tune the SSEP dynamics in such a way that its impact  is not seen at the macroscopic level. At this point this is exactly the same dynamics considered in \cite{patricia} but restricted to the bulk. At the boundary, we used the same dynamics introduced in \cite{bmns}, that is, a Glauber dynamics at sites $1$ and $n-1$, which plays the role of \emph{reservoirs}.  These reservoirs will also be scaled by a parameter which can be taken to infinity, and the highest its value, the slowest is the boundary dynamics. More specifically, the dynamics of the reservoirs can be described as follows.  Particles can be inserted into the system at the site $1$ (resp. $n-1$) with rate $m\alpha n^{-\theta}$ (resp. $m \beta n^{-\theta}$), and can be removed from the system through the site $1$ (resp. $n-1$) at rate $m(1-\alpha) n^{-\theta}$ (resp. $m(1-\beta) n^{-\theta}$). The factor $n^{-\theta}$ is the one scaling the boundary dynamics and the higher the value of $\theta$ the slower is the boundary dynamics, see Figure \ref{fig.1} for an illustration.  Throughout the text we  use the  parameters $\alpha,\beta \in (0,1)$, $m>0$ and $\theta \geq 0$.  \\

The PMM just described, belongs to the class of kinetically constrained lattice gases, which are interacting particle systems used to model the liquid/glass transition, see, for example,  \cite{toninelli, ritort} for a review on the subject. This class of models was introduced in the physics literature in \cite{kobanderson}, and they are usually classified as  \textit{cooperative} or \textit{non-cooperative}. In this classification, the PMM is a non-cooperative model, since its dynamical constraints are defined in  such a way that it is possible to construct a finite group of particles (the \textit{mobile cluster}), which can be moved to any position of the  discrete space where particles evolve, by using strictly positive exchange rates; and  any exchange is allowed when the mobile cluster is brought to the vicinity of the jumping particle. The non-cooperativity of the PMM and the fact that  we can perturb its dynamics with the  SSEP dynamics, are crucial properties of the model that will be extensively used in the proofs of our arguments. More precisely, when proving the hydrodynamic limit, in order to recognize the solution as a  weak solution to the PME, we will have to derive some replacement lemmas, which are stated and  proved in Section \ref{sec:RL}. In their proofs we will have to analyze the irreducibility of the model in the sense that we will have to send a particle from a site $x$ to some site $y$ at a distance of order $O(n)$.  In spite of having available  the SSEP dynamics, one could think that this could be accomplished easily. Nevertheless the problem cannot be overcome just by using  the SSEP jumps since they will be scale in a time less than the diffusive one and for this reason, particles cannot travel to sites at a distance of order $O(n)$. To push the argument further, we could try to use  the PMM jumps, but to do that we need the jumping particle to have particles in its vicinity and many times that does not happen. The trick is then to fix a finite size window around the jumping particle, create a mobile cluster in that window and once the mobile cluster is created  we can just use the PMM jumps to move the particles. After sending the particle to where we wanted we destroy the mobile cluster and we put the particles back to their  initial position. We remark that the jumps that are used to create and destroy the mobile cluster on the finite size window are the SSEP jumps, in all the rest of the path, we use the PMM jumps. The reader can see Figure \ref{figure-path} and  the proof of Lemma \eqref{teocontas} for a complete description of this argument.\\

As mentioned above, the main contribution of this article is to derive for the first time the hydrodynamic limit for the PMM with slow reservoirs. Then, we finally present the hydrodynamic equation for that model. The solution of the hydrodynamic equation is called  \emph{hydrodynamic profile}. Our hydrodynamic profiles are weak solutions of the PME with different boundary conditions depending on the range of the parameter $\theta$. For $0\leq \theta < 1$, we obtain the PME with Dirichlet boundary conditions, which is given by,
\begin{equation}\label{dirichlet}
\begin{cases} 
&\partial_{t}\rho_{t}(u)=\Delta\, ({\rho} _{t}(u))^2, \quad (t,u) \in (0,T]\times(0,1), \\
&{ \rho} _{t}(0)=\alpha, \quad  { \rho}_{t}(1)=\beta, \quad t\in(0,T].
\end{cases} 
\end{equation}
For $\theta = 1$, the boundary dynamics is slowed enough so the boundary conditions of Dirichlet type are replaced by a type of Robin boundary conditions,

\begin{equation}\label{robin}
\begin{cases}
&\partial_{t}\rho_{t}(u)= \Delta\, ({\rho} _{t}(u))^2, \quad (t,u) \in (0,T]\times(0,1), \\
&\partial_{u}(\rho_{t}(0))^2=\kappa(\rho_{t}(0) -\alpha), \quad t\in(0,T],\\
&\partial_{u}( \rho_{t}(1))^2=\kappa(\beta -\rho_{t}(1))\,, \quad t\in(0,T],
\end{cases}
\end{equation}
where $\kappa \in[0,\infty)$. Finally, for $\theta > 1$, the boundary is sufficiently slowed so that the Robin boundary conditions are replaced by Neumann boundary conditions (taking $\kappa = 0$ in \eqref{robin}) which dictate that, macroscopically, there is no flux of particles from the boundary reservoirs.\\
 
In order to better understand the hydrodynamic behavior of our model, we start by observing that the PME, $\partial_t \rho = \Delta(\rho^M)$, $M>1$, is a nonlinear evolution equation of parabolic type. This equation has received a lot of attention in the last decades due to the mathematical difficulties of building a theory for nonlinear versions of the heat equation. One can rewrite the equation in divergence form as 
\begin{equation}\label{pme}
\partial_t \rho = \nabla(D(\rho) \nabla \rho),
\end{equation}
where $\rho = \rho(t,u)$ is a scalar function and $D(\rho)=M\rho^{M-1}$ is the diffusion coefficient. The space variable $u$ takes values in some bounded or unbounded domain of $\mathbb{R}^d$ and the variable $t$ satisfies $t\geq 0$. 
As mentioned above, the  PME is also a degenerate parabolic equation, since the diffusion coefficient vanishes when $\rho$ goes to zero. Because of that, the regularity results for its solutions is weaker than the solutions of classical parabolic equations and the techniques for the study of PME are much more refined. Matters as existence and uniqueness of classical and weak solutions are also affected by the degeneracy of this equation.  From the physical point of view, one of the main differences between the PME and the heat equation is the so-called  finite speed of propagation, that is, its solutions can be compactly supported at each fixed time. This property implies the appearance of a free boundary that separates the support of the solution from the empty region. Across this boundary, the solution loses regularity. See \cite{vazquez_book} and references therein for a more detailed study of this equation.\\

The name PME was motivated by the work  \cite{Muskat}, in which the equation (with $M=2$) was used to model the density of a gas flowing through a porous medium. There are a lot of physical applications of the PME with several values of $M$, most of them being used to describe processes involving diffusion or heat transfer. In \cite{zel}, the equation was used to study the heat radiation in plasmas, and in \cite{gurney, gurtin}, the authors used the PME to describe migratory diffusion of biological populations. \\

We consider the one-dimensional boundary-value problem to the PME in a spatial domain $[0,1] \subset \mathbb{R}$ given in \eqref{pme} with $M = 2$. The spatial domain $[0,1]$ is the macroscopic space that corresponds to the discretized space $\Sigma_n$ defined above.  We remark that it is possible to extend our results to higher values of $M$ simply by taking  the jump rates of the process in accordance to that. For example, when $M=2$  in order to have a jump we required to have,  at least, one   particle close to the jumping particle, but if $M=3$ is taken,  we then need to require  two particles instead of one, see \eqref{eq:rates_m3} for the precise expression of the jump rates in this case. For simplicity of the presentation, all the arguments are given for the case $M=2$ but they extend easily to other values of $M$.\\

Now we explain the difficulties that we face when trying to derive the hydrodynamic limit for this model. The proof goes by showing tightness and characterizing uniquely the limit point. We remark that in the characterization of limit points, one important property of this model is that it is a \emph{gradient system}. This means that the instantaneous current of the system can be written as a discrete gradient of some local function of the dynamics, see \eqref{eq:current}.  In our case this function is a two degree function, that is,  it is a function given by sums of terms of the form $\eta(x)\eta(y)$  for $|y-x|\leq 2$. Due to this fact, one needs a replacement lemma in the whole bulk which allows to write this function in terms of an average of particles around a box of size $O(n)$. Since we are in the presence of reservoirs the proof of \cite{patricia} does not apply in this setting and we have to redo the whole argument. The idea consists in removing the boundary points from the bulk which do not allow these replacements; show that this removal is negligible in the limit and on the remaining  points we do a step-by-step replacement in the following fashion: at first step fix one of the variables $\eta(x)$ and do the replacement of $\eta(x+1)$ by the particle average to the right of $x+1$ on a box of size $O(n)$. Then, fix this average and repeat the previous replacement but now for the variable $\eta(x)$  and a box of size $O(n)$ to the left of $x$; this left-right argument is crucial so that the two boxes do not overlap and variables do not correlate.  When doing all these replacements one has to use the arguments described above, in which we need to create a mobile cluster capable of making particles move. Due to the reservoir's action, we also have to control the terms that arise at the boundary and we need to derive a couple of replacements to deal with these extra terms. \\

For the uniqueness of the limiting point we also had to derive the uniqueness of the weak solution of the PME with the different types of boundary conditions. The Dirichlet case could be easily proved but the Robin case deserved a special attention.  Since we did not find in the literature the exact statement of uniqueness we needed, we had to adapt the arguments in \cite{filo} to our particular  setting and for completeness we presented here the whole proof. Indeed, we obtain uniqueness for a Robin boundary condition for a function $u^2$ (in the place of a function $\beta^{-1}(u)$ in the notation of the article \cite{filo}) that is not Lipschitz, which is an important hypothesis for the proof given in \cite{filo}.\\
 
There are a couple of questions that  still have no answer  and are left for a future work. We highlight one which is concerned with the \emph{hydrostatic limit}. In our result on the  hydrodynamic limit we need to impose the starting measure to be associated to a profile, see \eqref{assoc_mea}. We note that when the boundary rates $\alpha$ and $\beta$ coincide with $\rho$, the Bernoulli product measure with constant parameter $\rho$ is a reversible measure for this model and, in particular, it is invariant. Nevertheless, when $\alpha\neq\beta$, this measure is no longer invariant and we have no information on the invariant measure of the system. The matrix method of Derrida \cite{derrida} cannot be straightforwardly applied to this model due to the complicated action of the bulk dynamics. One way to prove that the invariant measure of the model is associated with a profile, namely the stationary profile of the respective hydrodynamic equation (see Remark \ref{stationary-solution}) is to prove that its space correlations decay to $0$ when $n\to+\infty$. For this model it is not easy to obtain information of the correlations since the equations satisfied by the correlation function are not closed and again this is a consequence of the complicated action of the bulk dynamics. Another interesting problem is to derive the hydrodynamic limit for the PMM without the perturbation with the SSEP jumps. The difficulty we will face is the lack of mobility of the system: the creation of the mobile cluster now is not possible. These are problems that we will attack in the near future. \\

Here follows  an outline of the article. In Section \ref{s2}, we  state our results. In Subsection \ref{sec:model}, we  introduce some notations and we define  precisely  the PMM. In Subsection \ref{sec: hydro eq}, we present the notion of weak solution of the hydrodynamic equations, and we also present their stationary solutions. In Subsection \ref{sec:HL}, we state our main result. In Section \ref{sec:tightness}, we prove tightness for the sequence of probability measures of interest. In Section \ref{characterization}, we characterize the limit points. In Section \ref{sec:RL}, we provide estimates on Dirichlet forms and we present the proofs of all the replacement lemmas that are needed along the proof's arguments. Section \ref{sec: energy}, deals with energy estimates, and we finish the paper  with Section \ref{uniqueness} by presenting a proof of the uniqueness of weak solutions of each hydrodynamic equation.
  
\section{Statement of results}\label{s2}
\subsection{The model}\label{sec:model}
Let $n \geq 1$ a scaling parameter, and fix the following real numbers: $\theta \geq 0$, $m > 0$, $a\in (1,2)$ and $\alpha,\beta \in (0,1)$. Let $\Sigma_n$ be the discrete space $\{1,\ldots,n-1\}$  which we call \textit{bulk}. The dynamics of the PMM with a superposed SSEP dynamics and a Glauber dynamics can be described as follows: we associate a Poisson clock at each bond $\{x,x+1 \}$, with $x=1,\ldots,n-2$ and with a parameter depending on the exclusion rule and on the constraints of the process. At the left boundary (resp. right boundary) we artificially add Poisson clocks at the bonds $\{0,1\}$ (resp. $\{n-1,n\}$) and $\{ 1,0\}$ (resp. $\{n,n-1\}$) with a parameter that depends on the rate to get in or out the system at sites $1$ or $n-1$.  All these rates will be defined later on. Let $\Omega_n := \{ 0,1 \}^{\Sigma_n}$. We call the elements $\eta: \Sigma_n \rightarrow \{0,1\}$  \textit{configurations}. We say that the site $x$ is empty if $\eta(x)=0$, and that the site $x$ is occupied if $\eta(x)=1$. 
We can entirely characterize the continuous time Markov process $\{ \eta_t \}_{t \geq 0}$, with state space $\Omega_n$, by its infinitesimal generator $L_n$ given by $$L_{n} = L_{P} + n^{a-2}L_{S} + L_{B},$$ where $L_{P}$, $L_{S}$ and $L_{B}$ are the generators of the PMM, SSEP and the boundary dynamics, respectively. The generators act on functions $f: \Omega_n \rightarrow \mathbb{R}$ as 
\begin{equation*}
\begin{split}
( L_{P}f)(\eta) =& \,\, \sum_{x=1}^{n-2} c_{x,x+1}(\eta)\big\{ a_{x,x+1}(\eta)+a_{x+1,x}(\eta)\big\}\nabla_{x,x+1}f(\eta),\\
(L_{S}f)(\eta) =& \sum_{x=1}^{n-2}\big\{ a_{x,x+1}(\eta)+a_{x+1,x}(\eta)\big\} \nabla_{x,x+1}f(\eta),\\
( L_{B}f)(\eta) =& \,\,\tfrac{m}{n^\theta} I_{1}^{\alpha}(\eta)\nabla_{1}f(\eta) + \,\,\tfrac{m}{n^\theta}I_{n-1}^{\beta}(\eta)\nabla_{n-1}f(\eta),
\end{split}
\end{equation*}
where $\nabla_{x,x+1}f(\eta)=f(\eta^{x,x+1})-f(\eta)$, $\nabla_{z}f(\eta)= f(\eta^{z})-f(\eta)$ and 
for $\eta \in \Omega_n$ and $x,y\in\{1,\dots,n-1\}$, the exchange configurations are given by 
\begin{equation*}
\eta^{x,y}(z) = 
\begin{cases}
\eta(z), \; z \ne x,y,\\
\eta(y), \; z=x,\\
\eta(x), \; z=y,
\end{cases}
\quad {and}\quad \eta^x(z)= 
\begin{cases}
\eta(z), \; z \ne x,\\
1-\eta(x), \; z=x.
\end{cases}
\end{equation*}
Above, the exchange rates are given by
\begin{equation}\label{rates porous}
c_{x,x+1}(\eta) = \eta(x-1)+\eta(x+2),
\end{equation} 
\begin{equation}\label{rates ssep}
a_{x,y}(\eta) = \eta(x)(1-\eta(y)), \,\, x\neq y,
\end{equation}
\begin{equation}\label{rates boundary}
I_{z}^{b}(\eta) = b(1-\eta(z))+(1-b)\eta(z),
\end{equation}
for $x,y\in \{1,\ldots,n-2\}$, $z\in\{1,n-1\}$ and $b\in\{\alpha,\beta\}$. 
Note that, throughout the text, we use the convention 
\begin{equation}\label{convention}
\eta(0)=\alpha,\; \;\;\;\eta(n)=\beta,
\end{equation}
where $\alpha,\beta \in (0,1)$. Figure \ref{fig.1} below shows the dynamics of the model.
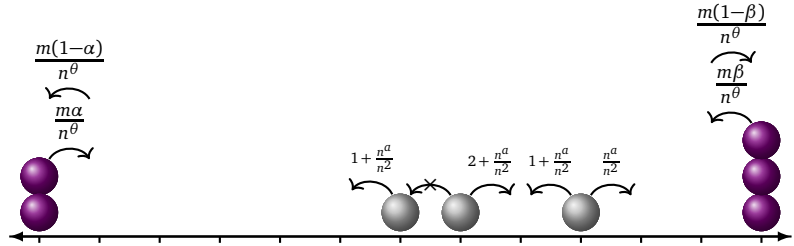
\begin{figure}[htb]
\begin{center}
\begin{tikzpicture}[thick, scale=0.8]
\draw[latex-] (-6.5,0) -- (6.5,0) ;
\draw[-latex] (-6.5,0) -- (6.5,0) ;
\foreach \x in  {-6,-5,-4,-3,-2,-1,0,1,2,3,4,5,6}
\draw[shift={(\x,0)},color=black] (0pt,0pt) -- (0pt,-3pt) node[below] 
{};

\node[ball color=violet, shape=circle, minimum size=0.5cm] at (-6.,1.0) {};
\node[ball color=violet, shape=circle, minimum size=0.5cm] at (-6.,0.4) {};
\node[ball color=violet, shape=circle, minimum size=0.5cm] (R) at (6.,0.4) {};
\node[ball color=violet, shape=circle, minimum size=0.5cm] (S) at (6.,1.0) {};
\node[ball color=violet, shape=circle, minimum size=0.5cm] (T) at (6.,1.6) {};

\node[shape=circle,minimum size=0.5cm] (Q) at (-6.,2.0) {};
\node[shape=circle,minimum size=0.5cm] (M) at (-6.,1.0) {};
\node[shape=circle,minimum size=0.5cm] (P) at (-5.,2.0) {};
\node[shape=circle,minimum size=0.5cm] (N) at (-5.,1.0) {};
\node[shape=circle,minimum size=0.5cm] (EE) at (-3.,0.5) {};
\node[shape=circle,minimum size=0.5cm] (E) at (-1.,0.5) {};
\node[shape=circle,minimum size=0.5cm] (K) at (0,0.4) {};
\node[shape=circle,minimum size=0.5cm] (G) at (2,0.4) {};
\node[shape=circle,minimum size=0.5cm] (LL) at (4.,0.4) {};
\node[shape=circle,minimum size=0.5cm] (U) at (5,1.6) {};
\node[shape=circle,minimum size=0.5cm] (W) at (5,2.6) {};
\node[shape=circle,minimum size=0.5cm] (V) at (6,2.6) {};

\node[ball color=black!30!, shape=circle, minimum size=0.5cm] (C) at (0,0.4) {};
\node[ball color=black!30!, shape=circle, minimum size=0.5cm] (A) at (1,0.4) {};
\node[ball color=black!30!, shape=circle, minimum size=0.5cm] (L) at (3.,0.4) {};

\path [->] (T) edge[bend right =60] node[above] {$\frac{m\beta}{n^\theta}$} (U);           
\path [->] (W) edge[bend left=60] node[above] {$\frac{m(1-\beta)}{n^\theta}$} (V);        
\path [->] (P) edge[bend right=60] node[above]  {${\frac{m(1-\alpha)}{n^\theta}}$} (Q);
\path [->] (M) edge[bend left=60] node[above] {$\frac{m\alpha}{n^\theta}$} (N);
\path [->] (A) edge[bend right=60,draw=black] node[] {$\times$} (K);
\path [->] (K) edge[bend right=60,draw=black] node[above] {\tiny{$1+\tfrac{n^{a}}{n^2}$}} (E);
\path [->] (A) edge[bend left=60] node[above] {\tiny{$2+\tfrac{n^{a}}{n^2}$}} (G);
\path [->] (L) edge[bend right=60] node[above] {\tiny{$1+\tfrac{n^{a}}{n^2}$}} (G);
\path [->] (L) edge[bend left=60] node[above] {\tiny{$\tfrac{n^{a}}{n^2}$}} (LL);

\end{tikzpicture}
\bigskip
\caption{The porous medium model with slow reservoirs.}\label{fig.1}
\end{center}
\end{figure}

\begin{remark}\label{remark-rates}
	We stress that \eqref{rates porous} is related to the diffusion coefficient of \eqref{pme} when $M=2$. Considering general values of $M$ in \eqref{pme}, we have to consider different values in \eqref{rates porous}. For example, when $M=3$, the diffusion coefficient of \eqref{pme} is given by $D(\rho) = 3\rho^{2}$, and the exchange rate in \eqref{rates porous} is given by \begin{equation}\label{eq:rates_m3}
	c_{x,x+1}(\eta) = \eta(x-2)\eta(x-1) + \eta(x-1)\eta(x+2) + \eta(x+2)\eta(x+3).\end{equation}
\end{remark}
\begin{remark}
The arguments for $M>2$ are exactly the same as the ones presented in this paper for the case $M=2$. For simplicity of the presentation, we stick to this choice of the rates.
\end{remark}
\begin{remark}\label{inv-measures}
Note that the dynamics is degenerate, gradient, and does not conserve the total number of particles. Note also that since the process is superposed with the SSEP dynamics, it is an irreducible Markov process on a finite state space, therefore only one invariant measure exists. In the equilibrium state, that is, when $\alpha=\beta$,  the interested reader can verify that the invariant measure of the process is the Bernoulli product measure, with a constant parameter, let us say, $\rho = \alpha = \beta $. For the non-equilibrium state, that is, when $\alpha \neq \beta$, we need to put more effort to obtain the invariant measure of this process.
\end{remark}

\begin{remark}From now on let $\{\eta_{tn^2}\}_{t\geq 0}$ denote the Markov process speeded up in the diffusive time scale $tn^2$ and driven by the infinitesimal generator $n^2L_n$. 
\end{remark}
\subsection{Hydrodynamic equations}\label{sec: hydro eq}
We first introduce some notations and definitions to state the hydrodynamic limit. Fix an interval $\mc I \subset \mathbb R$ and $m,n \in \mathbb{Z}$. We denote by:  
\begin{itemize}
	\item $C^{m,n}([0, T] \times \mc I)$, the set of real-valued functions defined on $[0,T] \times \mc I $ that are $m$ times  differentiable on the first variable and $n$ times differentiable on the second variable (with continuous derivatives);
	\item $C^{m}_{c}([0,1])$, the set of all $m$ continuously differentiable real-valued functions defined on $[0,1]$ with compact support;
	\item $C^{m,n}_0 ([0,T] \times [0,1])$, the set of real-valued functions $G \in C^{m,n}([0, T] \times[0, 1])$ such that $G_{s}(0)=G_{s}(1)=0$, for all $s\in[0,T]$;
	\item $\langle\cdot, \cdot\rangle$, the inner product in $L^2([0,1])$  with corresponding norm $\| \cdot \|_{2}$.
\end{itemize} 

\begin{definition}
\label{Def. Sobolev space}
Let $\mathcal{H}^{1}$ be the set of all locally summable functions $\zeta: [0,1] \to \mathbb{R}$ such that there exists a function $\partial_{u} \zeta \in L^{2}([0,1])$ satisfying
$$\langle \partial_u G, \zeta  \rangle =\langle G, \partial_{u} \zeta \rangle,$$
for all $G \in C^{\infty}(0,1)$ with compact support. For $\zeta \in \mathcal{H}^1$, we define the norm 
\begin{equation}\label{Sobolev norm 1}
\|\zeta \|_{\mathcal{H}^1} := \left(\|\zeta \|^{2}_{L^2[0,1]} + \|\partial_{u} \zeta \|^{2}_{L^2[0,1]} \right)^{1/2}.
\end{equation}  
Let $L^2(0,T; \mathcal{H}^1)$ be the space of all measurable functions $\xi: [0,T]\to \mathcal{H}^1$ such that
\begin{equation}\label{sobolev norm 2}
\|\xi\|^{2}_{L^2(0,T;\mathcal{H}^1)} := \int_{0}^{T}\|\xi_t\|^{2}_{\mathcal{H}^1} < \infty.
\end{equation}
\end{definition}

\noindent For more details about the definitions above the interested reader can see \cite{evans}. After both definitions and notations outlined above, we may move forth to define the notion of weak solution of the hydrodynamic equations that we will use along this article.

\begin{definition}
\label{Def. Dirichlet}
Let $\alpha,\beta \in(0,1)$ and $g:[0,1]\rightarrow [0,1]$ a measurable function. We say that $\rho:[0,T]\times[0,1] \to [0,1]$ is a weak solution of the PME with  Dirichlet boundary conditions
 \begin{equation}\label{eq:Dirichlet}
 \begin{cases}
 &\partial_{t}\rho_{t}(u)=\Delta\, ({\rho} _{t}(u))^2, \quad (t,u) \in (0,T]\times(0,1),\\
 &{ \rho} _{t}(0)=\alpha, \quad { \rho}_{t}(1)=\beta,\quad t \in (0,T], \\
 &{ \rho}_{0}(\cdot)= g(\cdot),
 \end{cases}
 \end{equation}
if the following conditions hold:
\begin{enumerate}
\item $\rho^2 \in L^{2}(0,T; \mathcal{H}^{1})$;
\item $\rho$ satisfies the integral equation:
\begin{equation}\label{eq:Dirichlet integral}
\begin{split}
F_{Dir}(G,t,\rho,g):=\,&\langle \rho_{t} , G_{t}\rangle  -\langle g ,  G_{0}\rangle - \int_0^t\langle \rho_{s}, (\partial_s G_{s} +\rho_s \Delta G_{s} ) \rangle \, ds\\&\quad\quad \quad\quad\quad + \int_0^t\big\{\beta^2\partial_uG_s(1) -\alpha^2\partial_uG_s(0)   \big\}\,ds =0,
\end{split}   
\end{equation}
for all $t\in [0,T]$ and any function $G \in C_0^{1,2} ([0,T]\times[0,1])$;
\item $\rho_{t}(0)=\alpha$ and $\rho_{t}(1)=\beta$ for all $t\in(0,T]$.

\end{enumerate}
\end{definition}

\begin{definition}
\label{Def. Robin}
Let $\kappa \geq 0$, $\alpha, \beta\in(0,1)$ and $g:[0,1]\rightarrow [0,1]$ a measurable function. We say that  $\rho:[0,T]\times[0,1] \to [0,1]$ is a weak solution of the PME with Robin boundary conditions 
 \begin{equation}\label{eq:Robin}
 \begin{cases}
 &\partial_{t}\rho_{t}(u)= \Delta\, ({\rho_t}(u))^2, \quad (t,u) \in (0,T]\times(0,1),\\
 &\partial_{u}(\rho_{t}(0))^2=\kappa(\rho_{t}(0) -\alpha),\quad t \in (0,T], \\
&\partial_{u}( \rho_{t}(1))^2=\kappa(\beta -\rho_{t}(1)),\quad t \in (0,T], \\
 &{ \rho}_{0}(\cdot)= g(\cdot),
 \end{cases}
 \end{equation}
if the following conditions hold: 
\begin{enumerate}
\item $\rho^2 \in L^{2}(0,T; \mathcal{H}^{1})$;
\item $\rho$ satisfies the integral equation:
\begin{equation}\label{eq:Robin integral}
\begin{split}
F_{Rob}(G,t,\rho,g):=\,&\langle \rho_{t},  G_{t}\rangle  -\langle g,  G_{0}\rangle- \int_0^t\langle \rho_{s},( \partial_s G_{s}+\rho_s \Delta G_s ) \rangle   \, ds\\
&+\int^{t}_{0}  \big\{ (\rho_{s}(1))^2 \partial_u G_{s}(1)-(\rho_{s}(0))^2  \partial_u G_{s}(0) \big\} \, ds\\
&-  \kappa\int^{t}_{0} \big\{ G_{s}(0)(\alpha -\rho_{s}(0)) + G_{s}(1)(\beta -\rho_{s}(1)) \big\} \, ds=0,
\end{split}   
\end{equation}
for all $t\in [0,T]$ and any function $G \in C^{1,2} ([0,T]\times[0,1])$. 
\end{enumerate}
\end{definition}
\begin{remark} \label{neumann_cond_rem}
For $\kappa =0$ we obtain in \eqref{eq:Robin} Neumann boundary conditions. 
\end{remark}
\begin{lemma}\label{lem:uniquess}
The weak solutions of (\ref{eq:Dirichlet}) and (\ref{eq:Robin}) are unique.
\end{lemma}
The proof of last lemma can be found in Section \ref{uniqueness}. 

\begin{remark}\label{stationary-solution} 
In order to get more information about the invariant measures of the process in the non-equilibrium state, it is good to know the stationary solution of each hydrodynamic equation. Thus, a simple computation shows that the stationary solution of \eqref{eq:Dirichlet} is given on $u\in(0,1)$ by $$\bar{\rho}(u)=\sqrt{(\beta^{2}-\alpha^{2})u+\alpha^{2}},$$ 
and the stationary solution of \eqref{eq:Robin} is given on $u\in(0,1)$ by 
\begin{equation}\label{statio1}
\bar{\rho}(u)=\sqrt{au+b},
\end{equation}
where 
\begin{equation}\label{statio2}
a=\kappa(\sqrt b-\alpha) \;\; \text{and} \;\; b= \left( \frac{\kappa\alpha + (\alpha + \beta)^2}{2(\alpha + \beta) + \kappa} \right)^2.
\end{equation}
The stationary solution for the Neumann case is simply a constant. {But, in fact, we observe that, looking back at  the stationary solution that we just computed, when we take $\kappa=0$, the stationary solution is given on $u\in(0,1)$ by $\bar{\rho}(u)=\frac{\alpha + \beta}{2}$.}
\end{remark}

\subsection{Hydrodynamic Limit}
\label{sec:HL}
For any configuration  $\eta \in \Omega_n$, we define the empirical measure $\pi^n(\eta,du)$ on $[0,1]$ by 
\begin{equation*}
\pi^n(\eta, du) =\dfrac{1}{n}\sum _{x\in \Sigma_n}\eta(x)\delta_{\frac{x}{n}}\left( du\right),
\end{equation*}
where $\delta_{a}$ is a Dirac mass on $a \in [0,1]$. We also define $\pi^n_{t}(\eta, du):=\pi^n(\eta_{tn^2}, du)$.
For a test function $G:[0,1]\rightarrow \mathbb{R}$, we denote by $\langle \pi_{t}^{n},G \rangle$ the integral of $G$ with respect to the measure $\pi_{t}^{n}$, which is equal to
\begin{equation*}
\langle \pi_{t}^{n},G \rangle = \frac{1}{n} \sum_{x\in \Sigma_{n}} G\left(\tfrac{x}{n}\right)\eta_{tn^{2}}(x).
\end{equation*}
Fix $T>0$ and $\theta\geq 0$. Let ${\mc M}_+$ be the space of positive measures on $[0,1]$ with total mass bounded by $1$ equipped with the weak topology. Let $\mu_n$ be a measure on $\Omega_n$. We denote by $\mathbb P _{\mu _n}$ the probability measure in the Skorokhod space $\mathcal D([0,T], \Omega_n)$, that is, the space of c\`adl\`ag trajectories induced by the accelerated Markov process $\{\eta_{tn^2}\}_{t\ge 0}$ and the initial measure $\mu_n$. We denote by $\mathbb E _{\mu _n}$ the expectation with respect to $\mathbb P_{\mu _n}$.  Let $\lbrace\mathbb{Q}_n\rbrace_{n \in \mathbb{N}}$ be the  sequence of probability measures on $\mathcal D([0,T],\mathcal{M}_{+})$ induced by the Markov processes $\lbrace \pi_{t}^n\rbrace_{t\geq 0}$ and by $\mathbb{P}_{\mu_n}$.

Given a measurable function $g: [0,1]\rightarrow[0,1]$, we say that a sequence of probability measures $\lbrace\mu_n\rbrace_{n \in \mathbb{N} }$ on $\Omega_n$ is \textit{associated with $g(\cdot)$}, if for any continuous function $G:[0,1]\rightarrow \mathbb{R}$ and any $\delta > 0$ 
\begin{equation}\label{assoc_mea}
  \lim _{n\to +\infty } \mu _n\Bigg( \eta \in \Omega_n : \Bigg| \dfrac{1}{n}\sum_{x \in \Sigma_n }G\left(\tfrac{x}{n} \right)\eta(x) - \< G,g\> \, \Bigg|    > \delta \Bigg)= 0.
\end{equation}

\begin{theorem}\label{main_theorem}
Let $g:[0,1]\rightarrow[0,1]$ be a measurable function and let $\lbrace\mu _n\rbrace_{n \in \mathbb{N}}$ be a sequence of probability measures on $\Omega_n$ associated with $g(\cdot)$. Then, for any $t \in [0,T]$ and any $\delta>0$,
\begin{equation*}\label{limHidreform}
 \lim_{n \to +\infty}\mathbb{P}_{\mu_n}\Bigg( \eta_{\cdot} \in \mathcal{D}([0,T], \Omega_n): \Bigg| \dfrac{1}{n}\sum_{x\in \Sigma_n}G\left(\tfrac{x}{n}\right)\eta_{tn^2}(x) - \< G,\rho_t \> \,  \Bigg| > \delta \Bigg)=0,
\end{equation*}
where
\begin{itemize}
\item[$\bullet$] for $\theta<1$,  $\rho_{t}(\cdot)$ is a weak solution of \eqref{eq:Dirichlet};
\item[$\bullet$]  for $\theta =1$,  $\rho_{t}(\cdot)$ is a weak solution of (\ref{eq:Robin}) with $\kappa = m$;
\item[$\bullet$] for $\theta >1$,  $\rho_{t}(\cdot)$ is a weak solution of (\ref{eq:Robin}) with $\kappa=0$.
\end{itemize}
\end{theorem}
To prove Theorem \ref{main_theorem} we will use the classical entropy method of Guo, Papanicolau, and Varadhan \cite{entropy}. In Section \ref{sec:tightness}, we prove that the sequence of probability measures $\{ \mathbb{Q}_n \}_{n \in \mathbb{N}}$ is tight, i.e., that the sequence has limit points $\mathbb{Q}$. In Section \ref{characterization}, we prove that the density $\rho_t(u)$ is a weak solution of the corresponding hydrodynamic equation. In Section \ref{sec:RL}, we present some estimates for the Dirichlet forms that are necessary to prove the replacement lemmas, and we present the proofs of the replacement lemmas. Then, in Section \ref{sec: energy}, we prove the energy estimates, that is, $\rho^2 \in L^{2}(0,T; \mathcal{H}^{1})$. To conclude, in Section \ref{uniqueness}, we prove uniqueness of weak solutions for each hydrodynamic equation presented above, and due to this fact, we guarantee the uniqueness of the limit point $\mathbb{Q}$. 

\section{Tightness}
\label{sec:tightness}
In this section we prove that the sequence $\{\mathbb{Q}_n \}_{n \in \mathbb{N}}$, defined in Section \ref{s2}, is tight. Before start proving tightness, let us present some results we shall use within this section.

Fix a function $G\in C^{1,2}([0,T]\times[0,1])$. We know by Dynkin's formula, see Lemma A1.5.1 of \cite{kl},  that
\begin{equation}\label{dynkin}
M^{n}_{t}(G) = \langle \pi^{n}_{t},G_t \rangle - \langle \pi^{n}_{0},G_0 \rangle - \int_{0}^{t} (\partial_{s} + n^{2}L_{P} + n^{a}L_{S}+n^{2}L_{B} ) \langle \pi^{n}_{s},G_s \rangle \,ds
\end{equation}
is a martingale with respect to the natural filtration $\{\mathcal{F}_{t}\}_{t \geq 0}$, where $\mathcal{F}_{t}= \{\sigma(\eta_{s}): s \leq t \}$. Assume, for argument's sake, that $G$ is time independent. For $\eta\in \Omega_n$ and $x\in \Sigma_n$, we denote by $j_{x,x+1}(\eta)$ the instantaneous current associated to the bond $\{ x,x+1\}$, which is given by 
\begin{equation}\label{eq:current}
 \begin{cases}
 &j_{0,1}(\eta) = \frac{m}{n^{\theta}}(\alpha - \eta(1)),\\
 &j_{x,x+1}(\eta) = \tau_{x}h(\eta)-\tau_{x+1}h(\eta), \;\; \text{for} \;\; x\in \{1,\ldots,n-2 \}, \\
 &j_{n-1,n}(\eta) = \frac{m}{n^{\theta}}(\eta(n-1)-\beta),
\end{cases}
\end{equation}
where 
\begin{eqnarray*}
\tau_xh(\eta) = \eta(x-1)\eta(x) + \eta(x)\eta(x+1) - \eta(x-1)\eta(x+1)+n^{a-2}\eta(x).
\end{eqnarray*}
Using the computations above, we have that $n^2L_n\< \pi_s^{n},G \>$ is given by
\begin{equation}\label{action_generator}
\begin{split}
 \frac{1}{n}\sum_{x=1}^{n-1}\Delta_n G\left(\tfrac xn\right)\tau_{x}h(\eta_{sn^{2}})
+&  \nabla^{+}_nG(0)\tau_1h(\eta_{sn^2}) - \nabla^{-}_nG(1) \tau_{n-1}h(\eta_{sn^2})\\
 +&nG\left(\tfrac 1n\right) \tfrac{m}{n^{\theta}}\big(\alpha - \eta_{sn^2}(1)\big) + nG\left(\tfrac{n-1}{n}\right)\tfrac{m}{n^{\theta}}\big(\beta - \eta_{sn^2}(n-1)\big) ,
\end{split}
\end{equation}
where for $x\in\Sigma_n$, the discrete Laplacian is given by
\begin{equation*}
\Delta_{n}G\left(\tfrac{x}{n} \right) = n^{2}\Big( G\left(\tfrac{x-1}{n} \right)- 2G\left(\tfrac{x}{n} \right)+ G\left(\tfrac{x+1}{n} \right)\Big),
\end{equation*}
and the discrete derivatives are given by
\begin{equation*}
\nabla_{n}^{+}G\left(\tfrac{x}{n} \right) = n\Big(G\left(\tfrac{x+1}{n}\right) - G\left(\tfrac{x}{n} \right)\Big)\quad  \textrm{and }\quad\nabla_{n}^{-}G\left(\tfrac{x}{n} \right) = n\Big(G\left(\tfrac{x}{n} \right) - G\left(\tfrac{x-1}{n}\right)\Big).
\end{equation*}
Since the function $G$ is time independent and using the convention \eqref{convention}, the martingale in \eqref{dynkin} is equal to
\begin{equation}\label{Complete Dynkin}
\begin{split}
 & \langle \pi^{n}_{t},G \rangle - \langle \pi^{n}_{0},G \rangle  - \int_{0}^{t} \frac{1}{n}\sum_{x=1}^{n-1}\Delta_{n}G\left(\tfrac xn\right)\tau_{x}h(\eta_{sn^{2}}) \, ds  \\
&-\int_{0}^{t} \nabla^{+}_nG(0)\tau_1h(\eta_{sn^2}) \,ds + \int_{0}^{t} \nabla^{-}_nG(1) \tau_{n-1}h(\eta_{sn^2}) \,ds \\
&-m\frac{n}{n^\theta}\int_{0}^{t}\Big\{ G\left(\tfrac 1n\right)\big( \alpha - \eta_{sn^{2}}(1) \big) + G\left(\tfrac {n-1}{n}\right)\big( \beta - \eta_{sn^{2}}(n-1) \big) \Big\} \,ds.
\end{split}
\end{equation}
\begin{remark}\label{estimates}
By the mean value theorem and since $|\eta_{sn^2}(x)|\leq 1$, we have that $$\big|\Delta_{n}G\left(\tfrac{x}{n}\right)\big| \leq 2 \| G'' \|_{\infty}, \;\; |\nabla^{+}_{n}G(0)| \leq 2 \| G' \|_{\infty}, \;\; \text{and} \;\; |\nabla^{-}_{n}G(1)| \leq 2 \| G' \|_{\infty},$$ for all $s\geq 0$ and $x\in\Sigma_n$. 
\end{remark}
\begin{remark}
Note that when $n \rightarrow +\infty$ the terms that come from the SSEP jumps vanish, so that, throughout the paper we ignore them and we look only at the remaining terms.
\end{remark}

\begin{proposition}
The sequence of measures $\{ \mathbb{Q}_{n} \}_{n \in \mathbb{N}}$ is tight with respect to the Skorokhod topology of $\mathcal{D}([0,T],\mathcal{M}_{+})$.
\end{proposition}
\begin{proof}
From Proposition 4.1.6 of \cite{kl}, it is enough to show tightness of the real-valued process $\{ \langle \pi_{t}^{n},G \rangle \}_{0\leq t \leq T}$ for a time independent function $G\in C([0,1])$. We claim that for each $\varepsilon > 0$,
\begin{equation}\label{tightness}
\begin{split}
\lim_{\gamma \to 0} \varlimsup_{n \to +\infty} \sup_{\tau\in \mathcal{T}_{T},\sigma \leq \gamma}\mathbb{P}_{\mu_{n}} \bigg( \big| \langle \pi_{\tau + \sigma}^{n},G \rangle - \langle \pi_{\tau}^{n},G \rangle \big| >\varepsilon \bigg)=0,
\end{split}
\end{equation}
where $\mathcal{T}_{T}$ is the set of stopping times bounded by $T$. By Proposition 4.1.7 of \cite{kl}, it is enough to show that \eqref{tightness} holds for functions $G$ in a dense subset of $C([0,1])$, with respect to the uniform topology of $C([0,1])$. From \eqref{dynkin}, Markov's and Chebyshev's inequalities, the probability in \eqref{tightness} can be bounded from above by
\begin{equation*}
\begin{split}
&\mathbb{P}_{\mu_{n}}\Bigg( \big|M_{\tau + \sigma}^{n}(G)-M_{\tau}^{n}(G)\big| > \frac{\varepsilon}{2} \Bigg) 
+ \mathbb{P}_{\mu_{n}}\Bigg( \Bigg| \int_{\tau}^{\tau+\sigma}n^{2}L_{n}\langle \pi_{r}^{n},G \rangle \, dr \Bigg| > \frac{\varepsilon}{2} \Bigg) \\
\leq& \frac{4}{\varepsilon^2} \mathbb{E}_{\mu_{n}}\Bigg( \big| M_{\tau+\sigma}^{n}(G)-M_{\tau}^{n}(G) \big|^2 \Bigg) + \frac{2}{\varepsilon} \mathbb{E}_{\mu_{n}} \Bigg( \Bigg| \int_{\tau}^{\tau + \sigma} n^{2}L_{n}\langle \pi_{r}^{n},G \rangle \, dr \Bigg| \Bigg).
\end{split}
\end{equation*}
So, if we prove that
\begin{equation}\label{condi-1}
\lim_{\gamma \to 0}\varlimsup_{n \to +\infty} \sup_{\tau\in \mathcal{T}_{T},\sigma \leq \gamma} \mathbb{E}_{\mu_{n}} \Bigg(  \big( M_{\tau + \sigma}^{n}(G)-M_{\tau}^{n}(G) \big)^2 \Bigg) = 0,
\end{equation}
and
\begin{equation}\label{condi-2}
\lim_{\gamma \to 0}\varlimsup_{n \to +\infty} \sup_{\tau\in \mathcal{T}_{T},\sigma \leq \gamma} \mathbb{E}_{\mu_{n}}\Bigg( \Bigg| \int_{\tau}^{\tau + \sigma}n^{2}L_{n}\langle \pi_{r}^{n},G \rangle \, dr \Bigg| \Bigg) = 0,
\end{equation}
the claim follows. We have divided the proof of \eqref{condi-1} and \eqref{condi-2} into two cases: $\theta \geq 1$ and $\theta \in [0,1)$.\\
\textbf{Case $\theta \geq 1:$} We begin by analyzing \eqref{condi-1}. Let $G \in C^2([0,1])$, where $C^2([0,1])$ is a dense subset of $C([0,1])$ with respect to the uniform topology. Define 
\begin{equation*}
F_{s}^{n}(G) := n^{2} \big( L_{n}\langle \pi_{s}^{n},G \rangle^{2} -2\langle \pi_{s}^{n},G \rangle L_{n}\langle \pi_{s}^{n},G \rangle \big).
\end{equation*}
Note that
\begin{equation*}
\mathbb{E}_{\mu_{n}}\Bigg( \big(M_{\tau + \sigma}^{n}(G) - M_{\tau}^{n}(G) \big)^2 \Bigg) = \mathbb{E}_{\mu_{n}}\Bigg( \int_{\tau}^{\tau+ \sigma} F_s^n(G) \,ds \Bigg), 
\end{equation*}
since $\big(M_{\tau + \sigma}^{n}(G)-M_{\tau}^{n}(G)\big)^{2} - \int_{\tau}^{\tau + \sigma} F_{s}^{n}(G) \,ds$ is a mean zero martingale. 
Note that, \eqref{condi-1} holds if we show that $\int_{\tau}^{\tau + \sigma} F_{s}^{n}(G)\,ds$ converges to zero uniformly in $t\in [0,T]$, when $n\to +\infty$. From Remark \ref{estimates}, a simple computation shows that $F_{s}^{n}(G)$ is bounded from above by a constant, times
\begin{equation}\label{quad1}
\begin{split}
\frac{1}{n} \|(G')^{2} \|_{\infty}+C(\alpha,\beta)\frac{m}{n^{\theta}} \| G^{2}\|_{\infty} +n^{a-3}\| (G')^2 \|_{\infty},
\end{split}
\end{equation}
where $C(\alpha,\beta)$ is a real constant depending on $\alpha$ and $\beta$. Taking $n\to +\infty$ in the previous display, the result follows.

It remains to prove \eqref{condi-2}. Recall \eqref{action_generator}. From Remark \ref{estimates}, we can bound the bulk term from above by
\begin{equation}\label{bound-bulk}
\left|  \Delta_n G\left(\tfrac xn \right)\tau_{x}h(\eta_{tn^2})\right| \, \leq \, 2 \| G'' \|_{\infty},
\end{equation}
and the boundary terms by
\begin{equation}\label{bound-boundary}
\begin{split}
\nabla^{+}_nG(0)\tau_1h(\eta_{sn^2}) + nG\left(\tfrac 1n\right) \tfrac{m}{n^{\theta}}\big(\alpha - \eta_{sn^2}(1)\big) \, & \, \leq \| G' \|_{\infty} + n^{1-\theta}m\| G \|_{\infty}, \\
-\nabla^{-}_nG_s(1) \tau_{n-1}h(\eta_{sn^2}) + nG\left(\tfrac{n-1}{n}\right)\tfrac{m}{n^{\theta}}\big(\beta - \eta_{sn^2}(n-1)\big) \, & \, \leq \|G' \|_{\infty} + n^{1-\theta}m\|G\|_{\infty}. 
\end{split}
\end{equation}
So, since $\theta \geq 1$, by \eqref{action_generator}, \eqref{bound-bulk}, and \eqref{bound-boundary}, we have that
\begin{equation*}
\begin{split}
&\lim_{\gamma \to 0} \varlimsup_{n\to +\infty} \sup_{\tau\in \mathcal{T}_{T},\sigma \leq \gamma} \mathbb{E}_{\mu_{n}}\Bigg( \Bigg| \int_{\tau}^{\tau + \sigma}n^{2}L_{n}\langle \pi_{r}^{n},G \rangle \, dr \Bigg| \Bigg)=0.
\end{split}
\end{equation*}

This proves \eqref{condi-2}. Note that the proof of \eqref{condi-1} works for any $\theta > 0$, but does not work for $\theta =0$ since the second term in \eqref{quad1} does not vanish when we take $n \to +\infty$. We treat this case below.\\
\textbf{Case $\theta \in [0,1)$:} Note that if we try to apply the strategy used above, we will have problems trying to control the expression $\int_{\tau}^{\tau + \sigma}n^{2}L_{B}\langle \pi_{s}^{n},G \rangle \,ds$. This happens because for these values of $\theta$, the terms that come from the boundary go to infinity with $n$. Due to this fact, since these terms also depend on the value of $G\left(\tfrac{1}{n}\right)$ and $G\left(\tfrac{n-1}{n}\right)$, we can get rid of them by asking the test function $G$ to have compact support in $(0,1)$. With this assumption, we can show that \eqref{condi-1} and \eqref{condi-2} are still valid when $G \in C_{c}^{2}(0,1)$ only by using the computations done for $\theta \geq 1$. To finish the proof, we need to show that \eqref{condi-1} and \eqref{condi-2} hold for $G\in C(0,1)$. The idea then is to approximate $G\in C(0,1)$ in $L_{1}$ by functions in $C_{c}^{2}(0,1)$. We leave the interested reader to look for the proof of this in, for example, Section $4$ of \cite{bmns}.  
\end{proof}
\section{Characterization of limit points}
\label{characterization}
We begin by fixing some notations used along the text. Fix $x\in \Sigma_{n}$, $\ell \in \mathbb{N}$, $\varepsilon >0, \delta>0$ and recall that $a\in(1,2)$. In what follows $\ve n$ denotes $\lfloor \ve n\rfloor$. Let 
\begin{equation}\label{available replac}
\Sigma^{\varepsilon }_{n}=\{1+\varepsilon n, \ldots, n-1-\varepsilon n\},
\end{equation} 
and 
\begin{equation}\label{boxes_1}\overleftarrow{\Lambda}_{x}^{\ell}:=\{ x-\ell+1, \ldots, x\}\quad \Big(\textrm{resp.}\quad \overrightarrow{\Lambda}_{x}^{\ell} := \{ x, \ldots, x+\ell-1\}\Big)\end{equation}  be the box of size $\ell$ to the left (resp. right) of the site $x$. We denote by 
\begin{equation}\label{boxes}
\overleftarrow{\eta}^{\ell}(x) = \frac{1}{\ell} \ \sum_{y \in \overleftarrow{\Lambda}_{x}^{\ell}}\eta(y)\quad \textrm{and}\quad \overrightarrow{\eta}^{\ell}(x) = \frac{1}{\ell}\sum_{y \in \overrightarrow{\Lambda}_{x}^{\ell}}
\eta(y)
\end{equation}  
the empirical densities in the boxes $\overleftarrow{\Lambda}_{x}^{\ell}$ and   $\overrightarrow{\Lambda}_{x}^{\ell}$, respectively.

From Section \ref{sec:tightness} we know that limit points $\bb Q$ of the sequence $\{\mathbb Q_n\}_{n\in \mathbb{N}}$ exist. We now observe that, as a consequence of the exclusion rule,  they are  concentrated on trajectories of measures, that are absolutely continuous with respect to the Lebesgue measure, see \cite{kl} for more details. Moreover, we claim that the density $\rho_{t}(u)$ is a weak solution of the corresponding hydrodynamic equation. This is proved in the next proposition. 
\begin{proposition}
Let $\mathbb{Q}$ be a limit point of the sequence $\{ \mathbb{Q}_{n}\}_{n\in \mathbb{N}}$. Then 
\begin{equation*}
\begin{split}
\mathbb{Q}\Big( \pi_{\cdot}\in \mathcal{D}([0,T],\mathcal{M}_{+}): F_{\theta}(G,t,\rho,g)= 0, \forall t\in[0,T], \forall G\in C_\theta \Big)= 1.
\end{split}
\end{equation*}
Above $F_\theta=F_{Dir}$ and $C_\theta=C_0^{1,2}([0,T]\times[0,1])$ for $\theta<1$; and $F_\theta=F_{Rob}$  and $C_\theta=C^{1,2}([0,T]\times[0,1])$ for $\theta\geq 1$.
\end{proposition}

\begin{proof}
The proof ends as long as we show that for any $\delta > 0$ and $ G\in C_\theta$  
\begin{equation}\label{CLP1}
\bb Q\Bigg(\pi _{\cdot}\in \mc D([0,T], \mathcal{M_{+}}): \sup_{0\le t \le T} \big| F_{\theta}(G,t,\rho,g) \big| >\delta \Bigg)=0,
\end{equation}
for each regime of $\theta$.  We start with the case $\theta\geq1$.
Recall from item $(2)$ of Definition \ref{Def. Robin} the definition of $F_{Rob}$. We note that the set inside last probability  is not an open set in the Skorokhod space. To avoid this problem, we fix $\ve>0$ and we consider two approximations of the identity, for fixed $u\in{[0,1]}$, which are given on $v\in[0,1]$ by  
$$\overleftarrow{\iota}^u_\ve(v)=\frac{1}{\ve}\textrm{1}_{(u-\ve,u]}(v)\quad \textrm{and}\quad\overrightarrow\iota^u_\ve(v)=\frac{1}{\ve}\textrm{1}_{[u,u+\ve)}(v).$$ We use the notation
\begin{equation}\label{RC61}\< \pi_s, \overleftarrow\iota^u_\ve\>= \frac{1}{\ve}\int_{u-\ve}^{u}\rho_{s}(v) \,dv \quad \textrm{and}\quad \< \pi_s, \overrightarrow\iota^u_\ve\>= \frac{1}{\ve}\int_{u}^{u+\ve}\rho_{s}(v) \,dv.\end{equation}
By summing and subtracting proper terms, we bound the probability in \eqref{CLP1} from above by the sum of
\begin{equation}\label{RC6}
\begin{split}
&\bb Q \Bigg( \sup_{0\le t \le T} \Bigg| \langle \rho_{t},  G_{t}\rangle  -\langle \rho_0,  G_{0}\rangle- \int_0^t\langle \rho_{s}, \partial_s G_{s} \rangle \, ds 
+\int_0^t\int_\ve^{1-\ve}\< \pi_s, \overrightarrow\iota^u_\ve\>\< \pi_s, \overleftarrow\iota^u_\ve\>\Delta G_s(u) \,du \, ds\\
 & +\int^{t}_{0} \Big \{\< \pi_s, \overleftarrow\iota^1_\ve\>\< \pi_s, \overleftarrow\iota^{1-\ve}_\ve\>\partial_u G_{s}(1)-\< \pi_s, \overrightarrow\iota^0_\ve\> \< \pi_s, \overrightarrow\iota^\ve_\ve\> \partial_u G_{s}(0) \Big\} \, ds \\
& -m\int^{t}_{0} \Big\{ G_{s}(0)\big( \alpha-\< \pi_s, \overrightarrow\iota^0_\ve\> \big) + G_{s}(1)\big( \beta -\< \pi_s, \overleftarrow\iota^1_\ve\> \big) \Big\} \,ds  \Bigg |>\tfrac{\delta}{7} \Bigg),
\end{split}
\end{equation}
\begin{equation}
\label{RC71}
\bb Q \Big( \big| \<\rho_{0}-g, G_{0}\> \big| >\tfrac{\delta}{7} \Big),
\end{equation}
\begin{equation}\label{RC9}
\begin{split}
&\bb Q \Bigg( \sup_{0\le t \le T} \Bigg| \int_0^t\Bigg\{\<\rho_s^2,\Delta G_s\>-\int_\ve^{1-\ve}\< \pi_s, \overrightarrow\iota^u_\ve\>\< \pi_s, \overleftarrow\iota^u_\ve\>\Delta G_s(u)\,du\Bigg\}\,ds  \Bigg| >\tfrac{\delta}{7} \Bigg),
\end{split}
\end{equation}
\begin{equation}
\label{RC7}
\bb Q \Bigg(  \sup_{0\le t \le T}  \Bigg| \,m\int^{t}_{0}  G_{s}(0)\Big\{\< \pi_s, \overrightarrow\iota^0_\ve\>-\rho_{s}(0)\Big\} \, ds  \Bigg| >\tfrac{\delta}{7} \Bigg),
\end{equation}
\begin{equation}
\label{RC8}
\bb Q\Bigg(  \sup_{0\le t \le T}  \Bigg| \int^{t}_{0} \Big\{(\rho_{s}(0))^{2}-\< \pi_s, \overrightarrow\iota^0_\ve\> \< \pi_s, \overrightarrow\iota^\ve_\ve\>\Big\}  \partial_u G_{s}(0) \, ds   \Bigg| >\tfrac{\delta}{7}\Bigg),
\end{equation}
plus two terms similar to the last ones but with respect to the right boundary. The term (\ref{RC71}) is equal to zero since $\mathbb Q$ is a limit point of $\{\mathbb Q_n\}_{n\in\mathbb N}$ and $\mathbb Q_n$ is induced by $\mu_n$, which satisfies \eqref{assoc_mea}.
To treat \eqref{RC9} we use  \eqref{RC61}, the fact that $0\leq \rho \leq 1$ and Lebesgue's Differentiation Theorem, to observe that, for almost $u\in[0,1]$,
\begin{equation}
\label{approximation}
\lim_{\ve\to 0}\Big|\rho_s(u)-\<\pi_s,\overrightarrow\iota_\ve^u\>\Big|= 0\quad\mbox{and}\quad \lim_{\ve\to 0}\Big|\rho_s(u)-\<\pi_s,\overleftarrow\iota_\ve^u\>\Big|=0.
\end{equation}
From the previous result and
\begin{equation}\label{4.12}
\begin{split}
&\Big|\int_0^t\<\rho^2_s, \Delta G_s\>\,ds-\int_0^t\int_\ve^{1-\ve}\< \rho_s, \overrightarrow\iota^u_\ve\>\< \rho_s, \overleftarrow\iota^u_\ve\>\;\Delta G_s(u) \,\,du \, ds\Big\vert\\&\leq \ve C(G,T)+
\int_0^t\int_\ve^{1-\ve}\Big\{|\rho_s(u)-\< \rho_s, \overrightarrow\iota^u_\ve\>|+|\rho_s(u)-\< \rho_s, \overleftarrow\iota^u_\ve\>|\Big\}\;|\Delta G_s(u)| \,\,du \, ds,
\end{split}
\end{equation} 
\eqref{RC9} goes to zero, when $\ve\to 0$. Now, for \eqref{RC7} and \eqref{RC8} we need the limits in  \eqref{approximation} to be true for all $u\in[0,1]$ (or at least at the points  of the boundary of $[0,1]$) and this is the statement of Lemma \ref{media_fronteira}. The probability in  \eqref{RC7} goes to zero, when $\ve\to 0$, by a simple application of Lemma \ref{media_fronteira}. To show that \eqref{RC8} goes to zero, when $\ve\to 0$, we  do a computation similar to one in \eqref{4.12}:
\begin{equation*}
\begin{split}
&\Big|\int_0^t\Big\{(\rho_{s}(0))^{2}-\< \pi_s, \overrightarrow\iota^0_\ve\> \< \pi_s, \overrightarrow\iota^\ve_\ve\>\Big\}  \partial_u G_{s}(0) \, ds  \Big\vert\\&\leq
\int_0^t\Big\{|\rho_s(0)-\< \rho_s, \overrightarrow\iota^0_\ve\>|+|\rho_s(0)-\rho_s(\ve)|+|\rho_s(\ve)-\< \rho_s, \overrightarrow\iota^\ve_\ve\>|\Big\}\;| \partial_u G_{s}(0)| \,\ \,ds,
\end{split}
\end{equation*} 
and use again Lemma  \ref{media_fronteira}.


%

Therefore, it remains only to look at (\ref{RC6}). Note that we still cannot use Portmanteau's Theorem, since the functions $\overleftarrow\iota^u_\ve$ and $\overrightarrow\iota^u_\ve$ are not continuous. Nevertheless, we can approximate each one of these functions by continuous functions, in such a way that the error vanishes as $\ve \to 0$. Then, since the set inside the probability in (\ref{RC6}) is an open set with respect to the Skorokhod topology, we can use Portmanteau's Theorem and  bound (\ref{RC6}) from above by
 
\begin{equation}\label{RC10}
\begin{split}
&\liminf_{n\to +\infty}\,\bb Q_n\Bigg(  \sup_{0\le t \le T} \Bigg|\langle \pi_{t},  G_{t}\rangle  -\langle \pi_0,  G_{0}\rangle- \int_0^t\langle \pi_{s}, \partial_s G_{s} \rangle   \, ds
\;\;-\int_0^t\int_\ve^{1-\ve}\< \pi_s, \overrightarrow\iota^u_\ve\>\< \pi_s, \overleftarrow\iota^u_\ve\>\Delta G_s(u)\,du\, ds\\
&\;\;+\int^{t}_{0} \Big \{\< \pi_s, \overleftarrow\iota^1_\ve\>\< \pi_s, \overleftarrow\iota^{1-\ve}_\ve\> \partial_u G_{s}(1)-\< \pi_s, \overrightarrow\iota^0_\ve\> \< \pi_s, \overrightarrow\iota^\ve_\ve\> \partial_u G_{s}(0) \Big\} \, ds \\
&\;\;\left.-m\int^{t}_{0} \Big\{ G_{s}(0)\big(\alpha-\< \pi_s, \overrightarrow\iota^0_\ve\>\big) + G_{s}(1)\big(\beta -\< \pi_s, \overleftarrow\iota^1_\ve\>\big) \Big\}\, ds \Bigg|>\tfrac{\delta}{7}\right).
\end{split}
\end{equation}
Summing and subtracting $\int_{0}^{t} n^{2}L_{n}\langle \pi_{s}^{n},G_{s}\rangle ds$ to the term inside the supremum in (\ref{RC10}), and recalling \eqref{Complete Dynkin}, we bound the probability in \eqref{RC10} from above by the sum of the next two terms
\begin{equation}\label{RC12}
\bb P_{\mu_n}\Bigg( \sup_{0\le t \le T} \big| M_{t}^{n}(G) \big| >\tfrac{\delta}{14}\Bigg),
\end{equation} 
and
\begin{equation}
\label{RC13}
\begin{split}
&\bb P_{\mu_n}\Bigg(  \sup_{0\le t \le T} \Bigg| \int_0^t n^{2}L_{n}\langle \pi_{s}^{n},G_{s}\rangle\,ds-\int_0^t\int_\ve^{1-\ve}\< \pi^{n}_s, \overrightarrow\iota^u_\ve\>\< \pi^{n}_s, \overleftarrow\iota^u_\ve\>\Delta G_s(u)\,du\, ds\\
&\;\; +\int^{t}_{0} \Big \{\< \pi^{n}_s, \overleftarrow\iota^1_\ve\>\< \pi^{n}_s, \overleftarrow\iota^{1-\ve}_\ve\> \partial_u G_{s}(1)-\< \pi^{n}_s, \overrightarrow\iota^0_\ve\> \< \pi^{n}_s, \overrightarrow\iota^\ve_\ve\>\partial_u G_{s}(0) \Big\} \, ds \\
&\;\; -m\int^{t}_{0} \Big\{ G_{s}(0) \big( \alpha-\< \pi^{n}_s, \overrightarrow\iota^0_\ve\> \big) + G_{s}(1) \big(\beta -\< \pi^{n}_s, \overleftarrow\iota^1_\ve\> \big) \Big\} \,  ds  \Bigg| >\tfrac{\delta}{14} \Bigg).
\end{split}
\end{equation} 
From Doob's inequality and \eqref{quad1}, the term (\ref{RC12}) vanishes as $n\to +\infty$. Finally, for $\bar{\delta}>0$, we can bound (\ref{RC13}) from above by the sum of the following terms
\begin{equation}
\label{RC14}
\begin{split}
&\,\bb P_{\mu_n}\Bigg(  \sup_{0\le t \le T} \Bigg|\int_0^t  \Bigg\{\frac{1}{n} \sum_{x\in\Sigma_n}\Delta_{n}G_s\left(\tfrac xn\right)\tau_{x}h(\eta_{sn^{2}})-\int_\ve^{1-\ve}\< \pi^n_s, \overrightarrow\iota^u_\ve\>\< \pi^n_s, \overleftarrow\iota^u_\ve\>\Delta G_s(u)\,du\Bigg\}\, ds \Bigg|>\tilde \delta\Bigg),
\end{split}
\end{equation}

\begin{equation}
\label{RC15}
\begin{split}
&\,\bb P_{\mu_n}\Bigg(  \sup_{0\le t \le T} \Bigg|\int^{t}_{0}\Big\{ \nabla^{+}_nG_s(0)\tau_1h(\eta_{sn^2})  -\< \pi^n_s, \overrightarrow\iota^0_\ve\> \< \pi^n_s, \overrightarrow\iota^\ve_\ve\>\partial_u G_{s}(0) \Big\} \, ds  \Bigg|>\tilde \delta \Bigg),
\end{split}
\end{equation}

\begin{equation}
\label{RC16}
\begin{split}
&\,\bb P_{\mu_n}\Bigg(  \sup_{0\le t \le T} \Bigg|\int^{t}_{0} \Big\{ G_{s}(0)(\alpha-\< \pi^n_s, \overrightarrow\iota^0_\ve\>) -G_s\left(\tfrac 1n\right)(\alpha - \eta_{sn^{2}}(1))\Big\} \, ds  \Bigg|>\tilde \delta \Bigg),
\end{split}
\end{equation}
plus two terms which are similar to the last ones, but which involve the right boundary. 
Now, we show that \eqref{RC16} vanishes when $n\to+\infty$ and then $\ve\to0$. By Taylor expansion on $G$, the terms which involve $\alpha$ vanish when $n\to+\infty$. Recall \eqref{boxes}. Observing that $\< \pi^n_s, \overrightarrow\iota^0_\ve\>=\overrightarrow{\eta}^{\ve n}_{sn^2}(1)$, from Lemma \ref{RLbound2}, \eqref{RC16} goes to zero as $n\to +\infty$ and $\ve\to 0$. Now, we treat \eqref{RC15}. Using Taylor expansion,  $\partial_u G_{s}(0) $ can be replaced by its discrete derivative $\nabla^{+}_nG_s(0)$. 
Since $$\< \pi^n_s, \overrightarrow\iota^0_\ve\>\< \pi^n_s, \overrightarrow\iota^\ve_\ve\>=\overrightarrow{\eta}^{\ve n}_{sn^2}(1)
\overrightarrow{\eta}^{\ve n}_{sn^2}(\ve n+1)+O\left(\tfrac{1}{\ve n}\right),$$
and $\tau_1h(\eta)=\eta(1)\eta(2)+\alpha (\eta(1)-\eta(2))$, we can use Theorem \ref{RLbound1} to replace the product $\eta(1)\eta(2)$ by $\overrightarrow{\eta}^{\ve n}_{sn^2}(1)\overrightarrow{\eta}^{\ve n}_{sn^2}(\ve n+1)$. And applying  Corollary \ref{cor_5.12}, the term with $\eta(1)-\eta(2)$ vanishes. Then, from these observations, \eqref{RC15} vanishes, as $n\to +\infty$ and $\ve\to 0$. Finally, we treat \eqref{RC14}. Recall \eqref{available replac}. Note that the sum in $\Sigma_n$ can be written as a sum over $\Sigma_n^\ve$ by paying a price of order $O(\ve)$. Now, note that the error from changing the integral in the space variable by its Riemann sum is of order $O(\tfrac 1n)$, and therefore we can bound \eqref{RC14} from above by 
\begin{equation}\label{prob17}
\begin{split}
&\bb P_{\mu_n}\Bigg(  \sup_{0\le t \le T} \Bigg|\int_0^t  \frac{1}{n}\sum_{x\in \Sigma_n^\ve}\Big\{\Delta_{n}G_s\left(\tfrac xn\right)\tau_{x}h(\eta_{sn^{2}})-\< \pi^n_s, \overrightarrow\iota^{x/n}_\ve\>\< \pi^n_s, \overleftarrow\iota^{ x/n }_\ve\>\Delta G_s\left(\tfrac xn\right)\Big\}\, ds \Bigg|>\tilde \delta \Bigg).
\end{split}
\end{equation}
By Taylor expansion on the test function $G$, we can replace its Laplacian by its discrete Laplacian, by paying a price of order $O(\tfrac 1n)$. Since for $x\in\Sigma_n$, $\tau_xh(\eta)=\eta(x-1)\eta(x)+\eta(x)\eta(x+1)-\eta(x-1)\eta(x+1)$ and 
$$\< \pi^n_s, \overleftarrow\iota^{x/n}_\ve\>\< \pi^n_s, \overrightarrow\iota^{x/n}_\ve\>=\overleftarrow{\eta}^{\ve n}_{sn^2}(x)
\overrightarrow{\eta}^{\ve n}_{sn^2}(x+1)+O\left(\tfrac{1}{\ve n}\right),$$
\eqref{prob17} can be bounded from above by the sum of the following terms
\begin{equation}\label{term_bulk}
\begin{split}
&\bb P_{\mu_n}\Bigg(  \sup_{0\le t \le T} \Bigg|\int_0^t  \frac{1}{n}\sum_{x\in \Sigma_n^\ve}\Big\{\Delta_{n}G_s\left(\tfrac xn\right)\big\{\eta_{sn^{2}}(x)\eta_{sn^2}(x+1)-\overleftarrow\eta_{sn^2}^{\ve n}(x)\overrightarrow\eta_{sn^2}^{\ve n}(x+1)\big\}\Big\}ds \Bigg|>\tilde \delta\Bigg),
\end{split}
\end{equation}
\begin{equation}\label{term bulk2}
\begin{split}
&\bb P_{\mu_n}\Bigg(  \sup_{0\le t \le T} \Bigg|\int_0^t  \frac{1}{n}\sum_{x\in \Sigma_n^\ve}\Delta_{n}G_s\left(\tfrac xn\right)\eta_{sn^{2}}(x-1)\big(\eta_{sn^2}(x)-\eta_{sn^2}(x+1)\big)ds \Bigg|>\tilde \delta\Bigg).
\end{split}
\end{equation}
From Theorem \ref{replace-bulk} and the application of Remark \ref{Remark lemma 6.2} twice, \eqref{term_bulk} and \eqref{term bulk2} vanish, respectively, as $n\to+\infty$ and $\ve\to 0$.  This ends the proof in the case $\theta=1$. We observe that the case $\theta>1$ is contained in the previous proof. 

Finally, we present the proof in the case $\theta\in [0,1)$. Recall the definition of $F_{Dir}$ from item $(2)$ of Definition \ref{Def. Dirichlet}. Following the same ideas presented in the case $\theta =1$, we can bound \eqref{CLP1} from above by the sum of
\begin{equation}\label{RC17}
\begin{split}
&\bb Q\Bigg(  \sup_{0\le t \le T} \Bigg|\langle \pi_{t},  G_{t}\rangle  -\langle \pi_0,  G_{0}\rangle +\int_0^t\int_\ve^{1-\ve}\< \pi_s, \overrightarrow\iota^u_\ve\>\< \pi_s, \overleftarrow\iota^u_\ve\>\Delta G_s(u)\,du\, ds \\
&\;\; - \int_0^t\langle \pi_{s}, \partial_s G_{s} \rangle   \, ds + \int^{t}_{0} \Big\{ \beta^2 \partial_u G_s(1)-\alpha^2 \partial_u G_s(0) \Big\} \, ds  \Bigg|> \tfrac{\delta}{3}\Bigg),
\end{split}
\end{equation}
\begin{equation}
\label{RC18}
\bb Q \Big(  |  \<\rho_{0}-g, G_{0}\>|>\tfrac{\delta}{3}\Big),
\end{equation}
\begin{equation}\label{RC19}
\begin{split}
&\bb Q\Bigg(  \sup_{0\le t \le T} \Bigg|\int_0^t\Bigg\{\<\rho_s^2,\Delta G_s\>-\int_\ve^{1-\ve}\< \pi_s, \overrightarrow\iota^u_\ve\>\< \pi_s, \overleftarrow\iota^u_\ve\>\Delta G_s(u)\,du\Bigg\}\,ds  \Bigg|>\tfrac{\delta}{3}\Bigg).
\end{split}
\end{equation}
Using the same arguments that we used above to treat \eqref{RC71} and \eqref{RC9}, we can see that \eqref{RC18} and \eqref{RC19} vanish. Therefore, it remains only to bound \eqref{RC17}. By the same arguments used in case $\theta = 1$, \eqref{RC17} is bounded from above by
\begin{equation}\label{RC20}
\begin{split}
&\liminf_{n \to +\infty}\, \bb Q_n\Bigg(  \sup_{0\le t \le T} \Bigg|\langle \pi_{t},  G_{t}\rangle  -\langle \pi_0,  G_{0}\rangle+\int_0^t\int_\ve^{1-\ve}\< \pi_s, \overrightarrow\iota^u_\ve\>\< \pi_s, \overleftarrow\iota^u_\ve\>\Delta G_s(u)\,du\, ds\\
&\;\; - \int_0^t\langle \pi_{s}, \partial_s G_{s} \rangle   \, ds+\int^{t}_{0} \Big\{ \beta^2 \partial_u G_s(1)-\alpha^2 \partial_u G_s(0) \Big\}\,  ds  \Bigg|>\tfrac{\delta}{3}\Bigg).
\end{split}
\end{equation}
Summing and subtracting $\int_{0}^{t} n^{2}L_{n}\langle \pi_{s}^{n},G_{s}\rangle ds$ to the term inside the supremum in (\ref{RC20}) and recalling \eqref{Complete Dynkin}, we can bound the probability in \eqref{RC20} from above by 
\begin{equation}\label{RC22}
\begin{split}
& \bb P_{\mu_{n}}\Bigg(  \sup_{0\le t \le T} \Bigg| \int_{0}^{t}n^2 L_n\< \pi_s^n,G_s \>ds +\int_0^t\int_\ve^{1-\ve}\< \pi^n_s, \overrightarrow\iota^u_\ve\>\< \pi^n_s, \overleftarrow\iota^u_\ve\>\Delta G_s(u)\,du\, ds\\
&\;\; +\int^{t}_{0} \Big\{ \beta^2 \partial_u G_s(1)-\alpha^2 \partial_u G_s(0) \Big\}\,  ds  \Bigg|>\tfrac{\delta}{6}\Bigg),
\end{split}
\end{equation}
plus $\bb P_{\mu_n}\left( \sup_{0\le t \le T} \left\vert M_{t}^{n}(G) \right\vert>\tfrac{\delta}{6}\right)$, which we showed above that vanishes when $n \to +\infty$ without using the fact that $\theta =1$.

From \eqref{RC22} and following again the steps of the case $\theta=1$, we need to bound the next terms
\begin{equation}\label{RC23}
\begin{split}
&\,\bb P_{\mu_n}\Bigg(  \sup_{0\le t \le T} \Bigg|\int_0^t  \Bigg\{\dfrac{1}{n}\sum_{x\in \Sigma_{n}^{\varepsilon}}\Delta_{n}G_s\left(\tfrac xn\right)\tau_{x}h(\eta_{sn^{2}})-\int_\ve^{1-\ve}\< \pi^n_s, \overrightarrow\iota^u_\ve\>\< \pi^n_s, \overleftarrow\iota^u_\ve\>\Delta G_s(u)\,du\Bigg\}\, ds \Bigg|> \tilde{\delta}\Bigg),
\end{split}
\end{equation}
\begin{equation}\label{RC24}
\begin{split}
&\,\bb P_{\mu_n}\Bigg(  \sup_{0\le t \le T} \Bigg| m\tfrac{n}{n^\theta} \int_0^t G_s\left(\tfrac{1}{n}\right)\big(\alpha-\eta_{sn^2}(1)\big)+G_s\left(\tfrac{n-1}{n}\right)\big(\beta - \eta_{sn^2}(n-1)\big) \, ds \Bigg|> \tilde{\delta} \Bigg),
\end{split}
\end{equation}
\begin{equation}\label{RC25}
\begin{split}
&\,\bb P_{\mu_n}\Bigg(  \sup_{0\le t \le T} \Bigg|\int_0^t   \alpha\nabla_n^+ G_s(0)\big(\eta_{sn^2}(1)-\eta_{sn^2}(2)\big)  \, ds \Bigg|> \tilde{\delta} \Bigg),
\end{split}
\end{equation}
\begin{equation}\label{RC26}
\begin{split}
&\,\bb P_{\mu_n}\Bigg(  \sup_{0\le t \le T} \Bigg|\int_0^t   \big(\eta_{sn^2}(1)\eta_{sn^2}(2)-\alpha^2\big)\partial_uG_s(0)\, ds \Bigg|> \tilde{\delta} \Bigg),
\end{split}
\end{equation}
\begin{equation}\label{RC27}
\begin{split}
&\,\bb P_{\mu_n}\Bigg(  \sup_{0\le t \le T} \Bigg|\int_0^t   \eta_{sn^2}(1)\eta_{sn^2}(2)\big( \nabla_n^{+}G_s(0) - \partial_uG_s(0) \big) \, ds \Bigg|> \tilde{\delta} \Bigg),
\end{split}
\end{equation}
plus three other terms similar to the last ones which come from the right boundary. Note that from the previous computations done for \eqref{term_bulk}, we have that \eqref{RC23} vanishes, as $n\to +\infty$. Not only \eqref{RC24} vanishes, (since from Lemma \ref{RL first} we can replace $\eta_{sn^2}(1)$ by $\alpha$ and $\eta_{sn^2}(n-1)$ by $\beta$), but also \eqref{RC25} vanishes for Corollary \ref{cor_5.12}. For \eqref{RC26}, we also replace $\eta_{sn^2}(2)$ by $\eta_{sn^2}(1)$, and we apply  Theorem \ref{RL first} twice to replace $\eta_{sn^2}(1)$ by $\alpha$. Finally, since $G\in C_0^{1,2}([0,T]\times[0,1])$, we have that $\nabla_n^+ G_s(0) \to \partial_u G_s(0)$ uniformly in $s$, which implies that \eqref{RC27} vanishes as $n \to +\infty$.
\end{proof}

\section{Replacement lemmas}\label{sec:RL}
This section is divided in four subsections as follows.
In Subsection \ref{replace1}, we state some estimates that will be used along the proofs of the replacement lemmas. We define Dirichlet forms, the \textit{carr\'e du champ} operator, and the Bernoulli product measure. Thereafter, we compare the Dirichlet forms and the \textit{carr\'e du champ} operator in order to state some of the estimates that will be used in the proofs of the replacement lemmas. In Subsections \ref{replace2} and \ref{replace3}, we present the proofs of the several replacement lemmas at the bulk and at the boundary, respectively. Finally, in Subsection 
\ref{replace4}, we prove item $3.$ in Definition \ref{Def. Dirichlet}.

\subsection{Dirichlet forms}\label{replace1}
Let $\mu$ be a probability measure on $\Omega_n$, and $f:\Omega_{n} \rightarrow \mathbb{R}$ a density with respect to $\mu$. The Dirichlet form of the process is defined as 
\begin{equation*}
\langle f, -L_n f\rangle_\mu= \langle f, -L_{P}f \rangle_{\mu} + n^{a-2}\langle f, -L_{S}f \rangle_{\mu} + \langle f, -L_{B}f \rangle_{\mu},
\end{equation*}
where $$\langle h, g\rangle_\mu=\sum_{\eta\in\Omega_n} h(\eta)g(\eta)\,\mu(\eta),$$
for all functions $ h,g:\Omega_{n} \rightarrow \mathbb{R}$. Moreover, recalling \eqref{rates ssep} and \eqref{rates boundary}, we define the \textit{carré du champ} operator, denoted by $D_n$ acting on functions $f: \Omega_n \rightarrow \mathbb{R}$, with respect to $\mu$ as
\begin{equation*} 
D_n(f, \mu)\, := \,D_{P}(f,\mu)+ n^{a-2}D_{S}(f,\mu)+D_{B}(f,\mu),
\end{equation*}
where  
\begin{equation*}
D_{P}(f,\mu)\;:=\;\sum_{x=1}^{n-2}\int p_{x,x+1}(\eta)(\nabla_{x,x+1}f(\eta))^{2} \, d\mu
\end{equation*}
and 
\begin{equation*}
D_{S}(f,\mu)\, :=\,\sum_{x=1}^{n-2} 
\int \big\{a_{x,x+1}(\eta)+a_{x+1,x}(\eta)\big\}(\nabla_{x,x+1}f(\eta))^{2} \, d\mu.
\end{equation*}
Note that the expression above can be rewritten as \begin{equation}\label{D_S}
D_{S}(f,\mu)\, =\,\sum_{x=1}^{n-2} 
\int (\nabla_{x,x+1}f(\eta))^{2} \, d\mu.
\end{equation}
Above $p_{x,x+1}(\eta):=c_{x,x+1}(\eta)\big\{a_{x,x+1}(\eta)+a_{x+1,x}(\eta)\big\}$, where the rates $c_{x,x+1}(\eta)$ and $a_{x,x+1}(\eta)$ are given in \eqref{rates porous} and \eqref{rates ssep} respectively, and
\begin{equation*}
D_{B}(f,\mu)\, := \, \tfrac{m}{n^\theta}\Big(F_{1}^{\alpha}(f,\mu)+F_{n-1}^{\beta}(f,\mu)\Big),
\end{equation*}
where $F_{1}^{\alpha}$ and $F_{n-1}^{\beta}$ are given by
\begin{equation}\label{I_terms}
F_{x}^{e}(f,\mu)= \int I_{x}^{e}(\eta)(\nabla_{x}f(\eta))^{2}\,d\mu,
\end{equation}
with $I_{x}^{e}$ given in \eqref{rates boundary} for $e\in\{\alpha,\beta\}$ and $x\in\{1,n-1\}$. For a measurable  profile $\rho:[0,1]\rightarrow [0,1]$, we define the Bernoulli product measure $\nu^{n}_{\rho(\cdot)}$ on $\Omega_{n}$ with marginals given by 
\begin{equation*}
\nu^{n}_{\rho(\cdot)} \{\eta: \, \eta(x)=1 \} = \rho\left( \tfrac{x}{n}\right). 
\end{equation*}

Let $f$ be a density with respect to $\nu^{n}_{\rho(\cdot)}$. The goal of this part of the section is to state the following estimate for the Dirichlet form $\< L_n \sqrt{f}, \sqrt{f} \>_{\nu^{n}_{\rho(\cdot)}}$, that is necessary in the proofs of the replacement lemmas.
\begin{lemma}\label{cond perfil}
Let $\rho:[0,1]\rightarrow [0,1]$ be a Lipschitz profile such that for all $u\in(0,1)$, 
\begin{equation}\label{measure_prof}
\alpha = \rho(0) \leq \rho(u) \leq \rho(1)=\beta,
\end{equation}
and which is locally constant at the boundary. Then, the Dirichlet form satisfies
\begin{equation}\label{eq:prices_DF}
\< L_n\sqrt{f}, \sqrt{f} \>_{\nu^n_{\rho(\cdot)}} \leq -\dfrac{1}{4}D_n(\sqrt{f},\nu_{\rho(\cdot)}^n) + O\left(\tfrac{1}{n}\right).
\end{equation}
In case   $\rho:[0,1]\rightarrow [0,1]$ is a constant profile, then
\begin{equation*}
\begin{split}
\langle L_n\sqrt{f},\sqrt{f} \rangle_{\nu_{\rho(\cdot)}^n}  &\leq  -\dfrac{1}{4}D_n(\sqrt{f},\nu_{\rho(\cdot)}^n) + O(\tfrac{1}{n^\theta}).
\end{split}
\end{equation*}
\end{lemma}
\begin{proof}
We note that since it is not difficult to prove the result, the interested reader can find the proof of it in Section 5 of \cite{dgn}.
\end{proof}

Now, we state all the replacement lemmas that were used in Section \ref{characterization}. We divide this part of the section into two subsections: one to prove the replacements lemmas concerning the bulk, and another to prove the replacements lemmas concerning the boundary.

\subsection{Replacement lemmas at the bulk}\label{replace2}
For the bulk, we basically have to prove that we can replace $\eta(x)$ by $\overleftarrow{\eta}^{\varepsilon n}(x)$ and $\eta(x+1)$ by $\overrightarrow{\eta}^{\varepsilon n}(x+1)$, as stated in Theorem \ref{replace-bulk}. {We remark that the sites $x\in \Sigma_n \setminus \Sigma_n^\varepsilon$, where $\Sigma_n^\varepsilon$ is defined in \eqref{available replac}, are the ones where we do not have space to replace $\eta(x)$ by $\overleftarrow{\eta}^{\varepsilon n}(x)$ (nor  $ \overrightarrow{\eta}^{\varepsilon n}(x)$),} and are those where we do not need to make the replacement. 

%
%
%
\begin{theorem}\label{replace-bulk}
Let $G^n_s: [0,1] \rightarrow \mathbb{R}$ be such that  $\|G_s^n\|_{\infty}\leq M<\infty$, for all $n\in \bb N$ and $s\in[0,T]$. For any $t\in [0,T]$, we have that
\begin{equation*}\label{replacement}
\lim_{\varepsilon \to 0}\varlimsup_{n \to +\infty} \mathbb{E}_{\mu_{n}}\Bigg( \Bigg| \int_{0}^{t}\dfrac{1}{n}\sum_{x \in \Sigma^{\varepsilon }_{n}}G_{s}^{n}\left(\tfrac{x}{n}\right)\big\{ \eta_{sn^2}(x)\eta_{sn^2}(x+1)-\overleftarrow{\eta}_{sn^2}^{\varepsilon n}(x)\overrightarrow{\eta}_{sn^2}^{\varepsilon n}(x+1) \big\} \, ds \Bigg| \Bigg) = 0.
\end{equation*}
\end{theorem}
Let us explain the idea behind the proof of this theorem. The proof is divided in three steps which are proved in the lemmas below. For $x\in \Sigma^\ve_n$ and $\delta>0$
\begin{itemize}
    \item[1)] replace $\eta(x)\eta(x+1)$ by $\overleftarrow{\eta}^{\ell}(x)\eta(x+1)$, for $\ell=n^{a-1-\delta}$; (Lemma \ref{teocontas})
      \vspace{0.1cm}
    
    \item[2)] replace $\overleftarrow{\eta}^{\ell}(x)\eta(x+1)$ by $\overleftarrow{\eta}^{\ell}(x)\overrightarrow{\eta}^{\varepsilon n}(x+1)$, for $\ell = n^{a-1-\delta}$; (Lemma \ref{teorcontas2}) 
    \vspace{0.1cm}
    \item[3)] replace $\overleftarrow{\eta}^{\ell}(x)\overrightarrow{\eta}^{\varepsilon n}(x+1)$ by $\overleftarrow{\eta}^{L}(x)\overrightarrow{\eta}^{\varepsilon n}(x+1)$, for $\ell = n^{a-1-\delta}$ and $L=\varepsilon n$. (Lemma \ref{engordandocaixa1})
\end{itemize}

\begin{lemma}\label{teocontas}
Let $G^n_s: [0,1] \rightarrow \mathbb{R}$ be such that $\|G_s^n\|_{\infty}\leq M<\infty$, for all $n\in \bb N$ and $s\in[0,T]$. For any $t \in [0,T]$, $\varepsilon >0$ and $\ell=n^{a-1-\delta}$ with $\delta >0$ such that $a-1-\delta \geq 0$, we have that
\begin{equation}\label{rl_bulk1}
\varlimsup_{n \to +\infty} \mathbb{E}_{\mu_{n}}\Bigg( \Bigg| \int_{0}^{t}\dfrac{1}{n}\sum_{x \in \Sigma^{\varepsilon}_{n}}G_{s}^{n}\left(\tfrac{x}{n}\right)\big\{ \eta_{sn^2}(x)-\overleftarrow{\eta}_{sn^2}^{\ell}(x)\big\}\eta_{sn^2}(x+1) \, ds \Bigg| \Bigg) = 0.
\end{equation}
\end{lemma}

\begin{proof}
Let $\nu^n_{\rho(\cdot)}$ be a Bernoulli product measure associated with the profile $\rho(\cdot)$ satisfying Lemma \ref{cond perfil}. Let $H\left(\mu_{n}| \nu^{n}_{\rho(\cdot)}\right)$ be the entropy of $\mu_{n}$ with respect to $\nu^n_{\rho(\cdot)}$, and $B>0$. By entropy's and Jensen's inequalities, the expected value in \eqref{rl_bulk1} can be bounded from above by $\tfrac{H\left(\mu_{n}| \nu^n_{\rho(\cdot)}\right)}{nB}$ plus
\begin{equation}\label{afterJensen}
\frac{1}{nB}\log  \mathbb{E}_{\nu^n_{\rho(\cdot)}}\Bigg( \exp \Bigg\{ nB\Bigg| \int_{0}^{t}\dfrac{1}{n}\sum_{x \in \Sigma_{n}^{\varepsilon}}G_{s}^{n}\left(\tfrac{x}{n}\right)\big\{\eta_{sn^2}(x)- \overleftarrow{\eta}_{sn^2}^{\ell}(x)\big\}\eta_{sn^2}(x+1) \, ds \Bigg|\Bigg\}\Bigg).
\end{equation}
Since $\rho(\cdot)$ satisfies \eqref{measure_prof}, it holds that
\begin{equation}\label{expliENTROPY}
\begin{split}
H\left(\mu_{n}|\nu^n_{\rho(\cdot)}\right) \, & \, \leq \log \left( \frac{1}{(\alpha \wedge (1-\beta))^n} \right)\sum_{\eta \in \Omega_{n}} \mu_{n}(\eta) \, \leq \, n C(\alpha,\beta).
\end{split}
\end{equation}
Thus, we only need to  treat the term in \eqref{afterJensen}. From Feynman-Kac's formula, \eqref{afterJensen} is bounded from above by
\begin{equation*}
 \int_{0}^{t} \sup_{f} \Bigg( \Bigg| \int_{\Omega_n} \dfrac{1}{n}\sum_{x \in \Sigma_{n}^{\varepsilon}}G_{s}^{n}\left(\tfrac{x}{n}\right)\big\{\eta(x)- \overleftarrow{\eta}^{\ell}(x)\big\}\eta(x+1)f(\eta) \, d \nu^n_{\rho(\cdot)} \Big| + \frac{n}{B} \langle L_n\sqrt{f}, \sqrt{f} \rangle_{\nu^n_{\rho(\cdot)}} \Bigg) \, ds,
\end{equation*}
where the supremum is carried over all densities $f$ with respect to $\nu^{n}_{\rho(\cdot)}$. To bound the first integral in the last display, we note that $\eta(x)-\overleftarrow{\eta}^{\ell}(x) = \tfrac{1}{\ell}\sum_{y\in \overleftarrow{\Lambda}^{\ell}_{x}}\eta(x)-\eta(y)$, and that $\eta(x)-\eta(y) = \sum_{z=y}^{x-1}\eta(z+1)-\eta(z)$. Therefore, by  summing and subtracting the term $\tfrac{1}{2}f(\eta^{z,z+1})$ and using the hypothesis on $G$, we can bound that integral from above by 
\begin{equation}\label{termLim}
\begin{split}
\frac{M}{2 \ell n} \sum_{x\in \Sigma_{n}^{\varepsilon}} \sum_{y\in\overleftarrow{ \Lambda}_x^\ell}\sum_{z=y}^{x-1}& \Bigg| \int_{\Omega_{n}}\big(\eta(z+1)-\eta(z)\big)\eta(x+1)\big( f(\eta)+f(\eta^{z,z+1}) \big) \, d\nu^n_{\rho(\cdot)} \Bigg|\\
+\frac{M}{2 \ell n} \sum_{x\in \Sigma_{n}^{\varepsilon}} \sum_{{y\in\overleftarrow{ \Lambda}_x^\ell}}\sum_{z=y}^{x-1}& \Bigg| \int_{\Omega_{n}}  \big(\eta(z+1)-\eta(z)\big)\eta(x+1)\big( f(\eta)-f(\eta^{z,z+1}) \big) \, d\nu^n_{\rho(\cdot)}\Bigg|.
\end{split}
\end{equation}
Let $\bar{\eta}$ denote the configuration $\eta$ removing its value at the sites $z$ and $z+1$. Thus, we can write the first integral in \eqref{termLim} as 
\begin{equation}\label{function1}
\begin{split}
&\Bigg| \sum_{\bar{\eta}\in \Omega_{n-2}} \Bigg( \bar{\eta}(x+1)\big(f(\bar{\eta},0,1)+f(\bar{\eta},1,0)\big) \left( 1- \rho\left(\tfrac{z}{n}\right)\right)\rho\left( \tfrac{z+1}{n} \right)\\
&-\bar{\eta}(x+1)\big( f(\bar{\eta},0,1)+f(\bar{\eta},1,0)\big) \,\rho\left( \tfrac{z}{n} \right) \left(1- \rho\left( \tfrac{z+1}{n} \right)\right)\Bigg)\nu^{n-2}_{\rho(\cdot)}(\bar{\eta}) \Bigg|,
\end{split}
\end{equation}
where the notation $f(\bar{\eta},1,0)$ means that we are computing $f(\eta)$ with $\eta(z)=1$ and $\eta(z+1)=0$. Using the fact that $\rho(\cdot)$ satisfies the 
		hypotheses of Lemma \ref{cond perfil}, \eqref{function1} is bounded from above by a constant (depending on $\rho(\cdot)$) times
\begin{equation*}
\frac{1}{n}\sum_{\bar{\eta}\in \Omega_{n-2}}\big(f(\bar{\eta},0,1)+f(\bar{\eta},1,0)\big)\nu_{\rho(\cdot)}^{n-2}(\bar{\eta}).
\end{equation*}
Since last term is bounded from above by
\begin{equation*}
\frac{2}{n}\sum_{z\in\{0,1\}}\sum_{\eta \in \Omega_{n}}f(\eta)\Big(\prod_{y=z,z+1} \rho\left( \tfrac{y}{n} \right)^{\eta(y)}  \left( 1- \rho\left( \tfrac{y}{n} \right) \right)^{1-\eta(y)}\Big)^{-1}\nu_{\rho(\cdot)}^n(\eta)
\end{equation*}
and $f$ is a density with respect to $\nu_{\rho(\cdot)}^{n}$, \eqref{function1} is of order $O(\frac{1}{n})$. Thus, the first expression in \eqref{termLim} is bounded from above by a constant, times $ \tfrac{\ell }{n}$. It remains to treat the second integral term in \eqref{termLim}. Note that for two nonnegative numbers $a$ and $b$, $a-b = [\sqrt{a}-\sqrt{b}][\sqrt{a}+\sqrt{b}]$. So, from Young's inequality we have that for any $A>0$ the second integral in \eqref{termLim} is bounded from above by
\begin{equation}\label{eq_5.7+}
\begin{split}
&\frac{M}{4n\ell A}\sum_{x\in \Sigma_{n}^{\varepsilon}}\sum_{{y\in\overleftarrow{ \Lambda}_x^\ell}}\sum_{z=y}^{x-1} \Bigg| \int_{\Omega_{n}}\big(\eta(z+1)-\eta(z)\big)^2\eta(x+1)^2\,\big(\sqrt{f(\eta)}+\sqrt{f(\eta^{z,z+1})}\big)^2 \, d\nu^n_{\rho(\cdot)}\Bigg|\\
+&\frac{AM}{4n \ell}\sum_{x\in \Sigma_{n}^{\varepsilon}}\sum_{{y\in\overleftarrow{ \Lambda}_x^\ell}}\sum_{z=y}^{x-1} \Bigg| \int_{\Omega_{n}}\big(\sqrt{f(\eta)}-\sqrt{f(\eta^{z,z+1})}\big)^2 \, d\nu^n_{\rho(\cdot)} \Bigg|\,.
\end{split}
\end{equation}
A simple computation, based on the fact that $f$ is a density, 
$\eta(\cdot)$ is bounded from above by one and  \eqref{D_S},
shows that the previous  display is bounded from above by 
\begin{equation}\label{bound1}
\frac{M \ell }{A} + \frac{M A}{4}D_{S}(\sqrt{f}, \nu^n_{\rho(\cdot)})\,.
\end{equation} 
Now, recall from \eqref{eq:prices_DF} that
\begin{equation*}
\begin{split}
\langle L_n\sqrt{f},\sqrt{f} \rangle_{\nu_{\rho(\cdot)}^n}  \, & \, \leq  -\dfrac{n^{a-2}}{4}D_S\Big(\sqrt{f},\nu_{\rho(\cdot)}^n \Big) + O(\tfrac{1}{n}).
\end{split}
\end{equation*}
Taking $A = \frac{n^{a-1}}{BM}$ in \eqref{bound1}, from last inequality and the previous computations, we have that the expectation in the statement of the lemma is bounded from above by a constant, times 
\begin{equation*}
\frac{1}{B} + T\left( \frac{\ell}{n} + \frac{B \ell }{n^{a-1}} \right).
\end{equation*}
Therefore, from our choice of $\ell$, taking $n \to +\infty$ and then $B \to +\infty$, the proof ends. 
\end{proof}

\begin{remark}
We stress that, in the proof above and the ones below, we present the replacement lemmas using $\nu_{\rho(\cdot)}^n$ and asking $\rho(\cdot)$ to satisfy the conditions stated in the first part of Lemma \ref{cond perfil}. Nevertheless, in the case  $\theta \geq 1$, it is enough to consider the  constant profile $\rho(\cdot)$, due to the bound obtained in the second part of Lemma \ref{cond perfil}. 
\end{remark}
\begin{remark}
We  observe that the restriction imposed above Remark 2.1 that the parameters $\alpha,\beta\in(0,1)$ comes from the estimate in \eqref{expliENTROPY}. Since, as mentioned above, in the case $\theta \geq 1$ we can take any constant profile, that restriction on the parameters is only needed in Dirichlet case, that is when $\theta<1$.
\end{remark}

\begin{remark}\label{Remark lemma 6.2}
A simple modification of the proof of Lemma \ref{teocontas} also shows that, for all $\varepsilon >0$ $$\varlimsup_{n \to +\infty} \mathbb{E}_{\mu_{n}}\Bigg( \Bigg| \int_{0}^{t}\dfrac{1}{n}\sum_{x \in \Sigma^{\varepsilon}_{n}}G_{s}^{n}\left(\tfrac{x}{n}\right)\big\{ \eta_{sn^2}(x)-\overrightarrow{\eta}_{sn^2}^{\ell}(x)\big\}\eta_{sn^2}(x-1) \, ds \Bigg| \Bigg) = 0.$$
\end{remark}
\begin{lemma}\label{teorcontas2}
Let $G^n_s: [0,1] \rightarrow \mathbb{R}$ be such that  $\|G_s^n\|_{\infty}\leq M<\infty$, for all $n\in \bb N$ and $s\in[0,T]$. For any $t \in [0,T]$ and $\ell=n^{a-1-\delta}$ with $\delta >0$ such that $a-1-\delta \geq 0$, we have that
\begin{equation}\label{rl_bulk2}
\lim_{\ve \to 0}\,\varlimsup_{n \to +\infty} \mathbb{E}_{\mu_{n}}\Bigg( \Bigg| \int_{0}^{t}\frac{1}{n}\sum_{x \in \Sigma_{n}^{\varepsilon}}G_{s}^{n}\left(\tfrac{x}{n}\right)\overleftarrow{\eta}_{sn^2}^\ell(x) \big\{ \eta_{sn^2}(x+1)- \overrightarrow{\eta}_{sn^2}^{\varepsilon n}(x+1) \big\} \, ds  \Bigg| \Bigg) = 0.
\end{equation}
\end{lemma}
\begin{proof}
As in the previous lemma, let $\nu^n_{\rho(\cdot)}$ be a Bernoulli product measure associated with the profile $\rho(\cdot)$ satisfying Lemma \ref{cond perfil}. Then, for $B>0$, the expectation in \eqref{rl_bulk2} is bounded from above by $\frac{C(\alpha,\beta)}{B}$, plus
\begin{equation*}
 \int_{0}^{t}\sup_{f}\Bigg( \Bigg| \int \dfrac{1}{n}\sum_{x\in \Sigma_{n}^{\varepsilon}}G_{s}^{n}\left(\tfrac{x}{n}\right)\overleftarrow{\eta}^\ell(x)\big\{ \eta(x+1)- \overrightarrow{\eta}^{\varepsilon n}(x+1)\big\}f(\eta) \, d\nu^n_{\rho(\cdot)} \Bigg| + \frac{n}{B}\<L_n\sqrt{f},\sqrt{f}  \>_{\nu^n_{\rho(\cdot)}} \Bigg) \, ds.
\end{equation*}
Recall the definition of $\overleftarrow{\Lambda}^{\ell}_{x}$ in \eqref{boxes}. Denote by $X_{1} = \{\eta \in \Omega_{n}: \overleftarrow{\eta}^{\ell}(x) \geq \frac{2}{\ell} \}$ the set of configurations that have at least two particles in $\overleftarrow{\Lambda}^{\ell}_{x}$. Thus, we can write the first integral inside the supremum above as the sum of the integral over the set $X_1$, plus the integral over its complementary $X_1^{c}$. By the hypothesis on $G$, the fact that $|\eta(x)|\leq 1$ for $x\in \Sigma_n$, and since $f$ is a density, the integral over $X_{1}^{c}$ is bounded from above by a constant, times $\tfrac{1}{\ell}$. Taking $n \to +\infty$, and by the hypothesis in $\ell$, the integral over $X_{1}^{c}$ vanishes. Moreover, to treat the remaining integral term, we just need to follow the same computations done in Lemma \ref{teocontas}. Hence, it is enough to estimate the next two terms
\begin{equation}\label{expressao11}
\begin{split}
& \frac{M}{2n^2\varepsilon}\sum_{x\in \Sigma_{n}^{\varepsilon}}\sum_{y\in\overrightarrow{ \Lambda}_{x+1}^{\varepsilon n}}\Bigg| \int_{X_1}\overleftarrow{\eta}^\ell(x)\big\{\eta(x+1)-\eta(y)\big\}\big(f(\eta)+f(\eta^{x+1,y})\big) \, d\nu^n_{\rho(\cdot)} \Bigg| \\
+&\frac{M}{2n^2\varepsilon}\sum_{x\in \Sigma_{n}^{\varepsilon}}\sum_{y\in\overrightarrow{ \Lambda}_{x+1}^{\varepsilon n}} \Bigg| \int_{X_1}\overleftarrow{\eta}^\ell(x)\big\{\eta(x+1)-\eta(y)\big\}\big(f(\eta)-f(\eta^{x+1,y})\big) \, d\nu^n_{\rho(\cdot)} \Bigg|.
\end{split}
\end{equation}
We begin by estimating the first term in the previous display. We use the notation $\bar{\eta}$ for the configuration $\eta$ removing its value at the sites $x+1$ and $y$. Since $x+1$ and $y$ do not intersect  $\overleftarrow{\Lambda}^{\ell}_{x}$, the term inside the absolute value in the first equation in \eqref{expressao11},  can be written as
\begin{equation*}
\begin{split}
&\Bigg| \sum_{\bar{\eta}\in \Omega_{n-2}} \mathbb{1}_{\bar\eta\in X_1} \overleftarrow{\bar\eta}^\ell(x)\Bigg\{\Big(f(\bar{\eta},0,1)+f(\bar{\eta},1,0)\Big) \Big( \left( 1- \rho\left(\tfrac{y}{n}\right)\right)\rho\left( \tfrac{x+1}{n} \right)
-\rho\left( \tfrac{y}{n} \right) \left(1- \rho\left( \tfrac{x+1}{n} \right)\right)\Big) \Bigg\}\nu^{n-2}_{\rho(\cdot)}(\bar{\eta}) \Bigg|. 
\end{split}
\end{equation*}
Using the fact that $\rho(\cdot)$ satisfies the conditions in the statement of Lemma \ref{cond perfil}, the first term in \eqref{expressao11} can be bounded from above by a constant, times
$$\frac{M}{2n^2\varepsilon}\sum_{x\in \Sigma_{n}^{\varepsilon}}\sum_{y\in\overrightarrow{ \Lambda}_{x+1}^{\varepsilon n}}\Big|\tfrac{x+1-y}{n}\Big|,$$
which is of order $O\left(\varepsilon \right)$. 
To bound the second expression in \eqref{expressao11} we need to be more careful. Recall that the idea behind this lemma is to replace a particle at the site $x+1$ by the empirical density in the box $\overrightarrow{\Lambda}^{\varepsilon n}_{x+1}$. To accomplish this we have to construct a path (with allowed jumps from the SSEP and the PMM dynamics), in such a way that we can send a particle from the site $x+1$ to the site $y$, for any $y\in \overrightarrow{\Lambda}_{x+1}^{\varepsilon n}$. This is explained in the next paragraph.

Recall that we are integrating over $X_1$, so that we have at least two particles in $\overleftarrow{\Lambda}_{x}^\ell$. Suppose, without loss of generality, that we have a particle at site $x_1 \in \overleftarrow{\Lambda}^{\ell}_{x}$, and another one at site $x_2 \in \overleftarrow{\Lambda}^{\ell}_{x}$, with $x_1 <x_2$. Using the SSEP jumps, we can take the particle from the site $x_1$ close to the particle at the site $x_2$, in such a way that the distance between them is less than or equal to $2$. Denoting by $\bullet$ an occupied site and by $\circ$ an empty site, this approximation is done by the SSEP jumps and at the end we get one of the following structures ( $\bullet \; \circ \; \bullet$ or $\bullet \; \bullet \; \circ$). When we reach a structure of the previous form, we say that a \emph{mobile cluster} has been created. Now, since we have a mobile cluster, there exists a sequence of nearest-neighbor jumps (with the PMM dynamics) which allow us to move the mobile cluster to any position on the box $\overrightarrow{\Lambda}_{x+1}^{\varepsilon n}$. Note that the SSEP jumps are used to approximate particles inside a box of size $\ell$, with the choice of $\ell$ as in the statement of this lemma. However, the PMM jumps can be used in the presence of the mobile cluster, to take a particle from a site $x+1$ to a site $y$ at a distance at most $\ve n$. After the creation of the mobile cluster with SSEP jumps, we move it to a vicinity of the site $x+1$ until the distance between them is less than or equal to $2$. Then, using the PMM jumps we take a particle to the site $y$ and we bring back the mobile cluster to the same position where it was created. When we reach this step, we use the SSEP jumps again to put the particles back to their initial positions, $x_1$ and $x_2$, respectively. To illustrate all the steps mentioned above the reader can see Figure \ref{figure-path}.

Note that, in this path, we use at most $4\ell$  jumps from the SSEP and $6(\ell+ \varepsilon n)$ jumps from the  PMM.   
From this, it follows that for any configuration $\eta\in X_1$, if $x_1$ and $x_2$ denote the position of the two closest particles to $x+1$, then there exist $N(x_1) \leq \ell+\ve n $ and a sequence of allowed moves  $\{x(i)\}_{i=0,\ldots,N(x_1)}$, which takes values in the set of points $ \{x_1,\ldots, y\}$, such that
$\eta^{(0)} = \eta, \eta^{(i+1)} =( \eta^{(i)})^{x(i),x(i)+1}$ and the final configuration is $\eta^{(N(x_1))} = \eta^{x+1,y}$. Note that the rates for each exchange is strictly positive. 
With this in mind, we can rewrite the exchange $f(\eta) - f(\eta^{x+1,y})$ as
\begin{equation}\label{path1_new}
f(\eta) - f(\eta^{x+1,y}) = \sum_{i=1}^{N(x_1)}f(\eta^{(i-1)})-f(\eta^{(i)})=\sum_{i\in I^{\text{exc}}}f(\eta^{(i-1)})-f(\eta^{(i)})+\sum_{i\in I^{\text{pmm}}}f(\eta^{(i-1)})-f(\eta^{(i)}),
\end{equation}
where $I^{\text{exc}}$ (resp. $I^{\text{pmm}}$) are the sets of  indexes that count the bonds used with SSEP jumps (resp. PMM jumps)  along the path. Take into account the fact that the SSEP jumps  are used only to create and to destroy the mobile cluster,
while all the rest of the path is done with PMM jumps. Now, substituting \eqref{path1_new} in the second term of  \eqref{expressao11} and using the triangular inequality,  we need to estimate the following expressions
\begin{equation}\label{path2}
\begin{split}
&\frac{M}{2n^2 \ve}\sum_{x\in \Sigma_{n}^{\varepsilon}}\sum_{y\in \overrightarrow{\Lambda}_{x+1}^{\ve n}}\sum_{i\in I^{\text{exc}}} \Bigg| \int_{X_1}\overleftarrow{\eta}^\ell(x)\big(\eta(x+1)-\eta(y)\big)\big(f(\eta^{(i-1)})-f(\eta^{(i)}\big) \, d\nu^n_{\rho(\cdot)} \Bigg| \\
+&\frac{M}{2n^2 \ve}\sum_{x\in \Sigma_{n}^{\varepsilon}}\sum_{y\in \overrightarrow{\Lambda}_{x+1}^{\varepsilon n}}\sum_{i\in I^{\text{pmm}}} \Bigg| \int_{X_1}\overleftarrow{\eta}^\ell(x)\big(\eta(x+1)-\eta(y)\big)\big(f(\eta^{(i-1)})-f(\eta^{(i)})\big) \, d\nu^n_{\rho(\cdot)}\Bigg|.
\end{split}
\end{equation}
Since $\eta^{(i)}=(\eta^{(i-1)})^{x(i-1),x(i-1)+1}$, the way to estimate the first term above is the same as it is done in \eqref{eq_5.7+}. For the sake of completeness, for each $A>0$, by Young's inequality, the definitions of $\Sigma_{n}^{\varepsilon}$ and $\overrightarrow{\Lambda}_{x+1}^{\ve n}$ (see \eqref{available replac} and \eqref{boxes_1}, respectively), the first sum of \eqref{path2} is bounded from above by
\begin{equation*}
\begin{split}
&\frac{M}{4A}\sum_{i\in I^{\text{exc}}} \int_{X_1} \big(\sqrt{f(\eta^{(i-1)})}+\sqrt{f(\eta^{(i)}}\big)^2 \, d\nu^n_{\rho(\cdot)}  +\frac{AM}{4}\sum_{i\in I^{\text{exc}}}  \int_{X_1}\big(\sqrt{f(\eta^{(i-1)})}-\sqrt{f(\eta^{(i)}}\big)^2 \, d\nu^n_{\rho(\cdot)}\,. 
\end{split}
\end{equation*}
  Now, remember that  the indexes in $ I^{\text{exc}}$  count the number of SSEP jumps   to move the two particles that are in the box $\overleftarrow{\Lambda}_{x}^{\ell}$ (which exist due to the fact that the integral is over  $X_1$)  close to the bond $\{x-1,x\}$ and the path  back to the initial position of these particles.
  Then, the first term in \eqref{path2} is bounded from above by 
$$\frac{2M \ell }{A} + \frac{AM}{2}D_{S}(\sqrt{f},\nu^n_{\rho(\cdot)}).$$  
In order to estimate the second term in \eqref{path2}, we repeat the argument above using the PMM jump rates. Recall that $p_{x,x+1}(\eta)$ is introduced below 
\eqref{D_S}.
 Thus, for all $\tilde{A}>0$, the second term in \eqref{path2} is bounded from above by
\begin{equation*}
\begin{split}
&\frac{M}{4\tilde{A}}\sum_{i\in I^{\text{pmm}}} \int_{X_1}\frac{1}{p_{x(i-1),x(i-1)+1}(\eta)} \big(\sqrt{f(\eta^{(i-1)})}+\sqrt{f(\eta^{(i)}}\big)^2 \, d\nu^n_{\rho(\cdot)}\\+&\frac{\tilde{A}M}{4}\sum_{i\in I^{\text{pmm}}}  \int_{X_1}p_{x(i-1),x(i-1)+1}(\eta)\,\big(\sqrt{f(\eta^{(i-1)})}-\sqrt{f(\eta^{(i)}}\big)^2 \, d\nu^n_{\rho(\cdot)}\,.
\end{split}
\end{equation*}
Observe that for $\eta\in X_1$ and for  $i\in I^{\text{pmm}}$,   $p_{x(i-1),x(i-1)+1}(\eta)$ is either equal to $1$ or $2$. Therefore, the second  term in \eqref{path2} is bounded from above by
$$\frac{2M \ve n}{\tilde{A}} + \frac{ \tilde{A}M}{2}D_{P}(\sqrt{f},\nu^n_{\rho(\cdot)}).$$ 
Taking $A=\frac{n^{a-1}}{2M B}$ and $\tilde{A}=\frac{n}{2M  B}$, from the previous computations, the expectation in the statement of the lemma  is bounded from above by a constant, times
\begin{equation}\label{termos}
\frac{1}{B} + T\left( \varepsilon + \frac{\ell B}{n^{a-1}} + \ve B\right).
\end{equation}
Taking $n \to +\infty$, the third term of \eqref{termos} vanishes by the choice of $\ell$. Taking $\varepsilon \to 0$, the second and fourth terms of \eqref{termos} vanish. To finish, we send $B \to +\infty$ and the remaining term vanishes.
\end{proof}

\begin{figure}[H]  
	\begin{center}
		\begin{tikzpicture}[scale=0.53]
		\draw [line width=1] (-9,10.5) -- (7,10.5) ; 
		\foreach \x in  {-9,-8,-7,-6,-5,-4,-3,-2,-1,0,1,2,3,4,5,6,7} 
		\draw[shift={(\x,10.5)},color=black, opacity=1] (0pt,4pt) -- (0pt,-4pt) node[below] {};
		\draw[] (-2.8,10.5) node[] {};
		
		\draw[] (-8,10.3) node[below] {\footnotesize{$x-\ell +1$}};
		\draw[] (-6,10.3) node[below] {\footnotesize{$x_1$}};
		\draw[] (-3,10.3) node[below] {\footnotesize{$x_2$}};
		\draw[] (-1,10.3) node[below] {\footnotesize{$x$}};
		\draw[] (0,10.3) node[below] {\footnotesize{$x+1$}};
		\draw[] (3,10.3)  node[below] {\footnotesize{$y$}};
		\draw[] (6,10.3) node[below] {\footnotesize{$x+ \varepsilon n$}};
		
		\draw[thick] (-8, 10.26) rectangle (-1, 11);
		\draw[thick] (-0, 10.26) rectangle (6, 11);
		
		\shade[shading=ball, ball color=black!50!] (-6,10.65) circle (.15);
		\shade[shading=ball, ball color=black!50!] (-3,10.65) circle (.15);
		\shade[shading=ball, ball color=black!50!] (0,10.65) circle (.15);
		\end{tikzpicture} 
	\end{center}
	\begin{center}
		\begin{tikzpicture}[scale=0.53]
		\draw [line width=1] (-9,10.5) -- (7,10.5) ; 
		\foreach \x in  {-9,-8,-7,-6,-5,-4,-3,-2,-1,0,1,2,3,4,5,6,7} 
		\draw[shift={(\x,10.5)},color=black, opacity=1] (0pt,4pt) -- (0pt,-4pt) node[below] {};
		\draw[] (-2.8,10.5) node[] {};
		
		\draw[] (-8,10.3) node[below] {\footnotesize{$x-\ell +1$}};
		\draw[] (-6,10.3) node[below] {\footnotesize{$x_1$}};
		\draw[] (-3,10.3) node[below] {\footnotesize{$x_2$}};
		\draw[] (-1,10.3) node[below] {\footnotesize{$x$}};
		\draw[] (0,10.3) node[below] {\footnotesize{$x+1$}};
		\draw[] (3,10.3)  node[below] {\footnotesize{$y$}};
		\draw[] (6,10.3) node[below] {\footnotesize{$x+ \varepsilon n$}};
		
		\draw[thick] (-8, 10.26) rectangle (-1, 11);
		\draw[thick] (-0, 10.26) rectangle (6, 11);
		
		\shade[shading=ball, ball color=black!50!] (-4,10.65) circle (.15);
		\shade[shading=ball, ball color=black!50!] (-3,10.65) circle (.15);
		\shade[shading=ball, ball color=black!50!] (0,10.65) circle (.15);
		\end{tikzpicture} 
	\end{center}
	\begin{center}
		\begin{tikzpicture}[scale=0.53]
		\draw [line width=1] (-9,10.5) -- (7,10.5) ; 
		\foreach \x in  {-9,-8,-7,-6,-5,-4,-3,-2,-1,0,1,2,3,4,5,6,7} 
		\draw[shift={(\x,10.5)},color=black, opacity=1] (0pt,4pt) -- (0pt,-4pt) node[below] {};
		\draw[] (-2.8,10.5) node[] {};
		
		\draw[] (-8,10.3) node[below] {\footnotesize{$x-\ell +1$}};
		\draw[] (-6,10.3) node[below] {\footnotesize{$x_1$}};
		\draw[] (-3,10.3) node[below] {\footnotesize{$x_2$}};
		\draw[] (-1,10.3) node[below] {\footnotesize{$x$}};
		\draw[] (0,10.3) node[below] {\footnotesize{$x+1$}};
		\draw[] (3,10.3)  node[below] {\footnotesize{$y$}};
		\draw[] (6,10.3) node[below] {\footnotesize{$x+ \varepsilon n$}};
		
		\draw[thick] (-8, 10.26) rectangle (-1, 11);
		\draw[thick] (-0, 10.26) rectangle (6, 11);
		
		\shade[shading=ball, ball color=black!50!] (-2,10.65) circle (.15);
		\shade[shading=ball, ball color=black!50!] (-1,10.65) circle (.15);
		\shade[shading=ball, ball color=black!50!] (0,10.65) circle (.15);
		\end{tikzpicture} 
	\end{center}
	\begin{center}
		\begin{tikzpicture}[scale=0.53]
		\draw [line width=1] (-9,10.5) -- (7,10.5) ; 
		\foreach \x in  {-9,-8,-7,-6,-5,-4,-3,-2,-1,0,1,2,3,4,5,6,7} 
		\draw[shift={(\x,10.5)},color=black, opacity=1] (0pt,4pt) -- (0pt,-4pt) node[below] {};
		\draw[] (-2.8,10.5) node[] {};
		
		\draw[] (-8,10.3) node[below] {\footnotesize{$x-\ell +1$}};
		\draw[] (-6,10.3) node[below] {\footnotesize{$x_1$}};
		\draw[] (-3,10.3) node[below] {\footnotesize{$x_2$}};
		\draw[] (-1,10.3) node[below] {\footnotesize{$x$}};
		\draw[] (0,10.3) node[below] {\footnotesize{$x+1$}};
		\draw[] (3,10.3)  node[below] {\footnotesize{$y$}};
		\draw[] (6,10.3) node[below] {\footnotesize{$x+ \varepsilon n$}};
		
		\draw[thick] (-8, 10.26) rectangle (-1, 11);
		\draw[thick] (-0, 10.26) rectangle (6, 11);
		
		\shade[shading=ball, ball color=black!50!] (1,10.65) circle (.15);
		\shade[shading=ball, ball color=black!50!] (2,10.65) circle (.15);
		\shade[shading=ball, ball color=black!50!] (3,10.65) circle (.15);
		\end{tikzpicture} 
	\end{center}
	\begin{center}
		\begin{tikzpicture}[scale=0.53]
		\draw [line width=1] (-9,10.5) -- (7,10.5) ; 
		\foreach \x in  {-9,-8,-7,-6,-5,-4,-3,-2,-1,0,1,2,3,4,5,6,7} 
		\draw[shift={(\x,10.5)},color=black, opacity=1] (0pt,4pt) -- (0pt,-4pt) node[below] {};
		\draw[] (-2.8,10.5) node[] {};
		
		\draw[] (-8,10.3) node[below] {\footnotesize{$x-\ell +1$}};
		\draw[] (-6,10.3) node[below] {\footnotesize{$x_1$}};
		\draw[] (-3,10.3) node[below] {\footnotesize{$x_2$}};
		\draw[] (-1,10.3) node[below] {\footnotesize{$x$}};
		\draw[] (0,10.3) node[below] {\footnotesize{$x+1$}};
		\draw[] (3,10.3)  node[below] {\footnotesize{$y$}};
		\draw[] (6,10.3) node[below] {\footnotesize{$x+ \varepsilon n$}};
		
		\draw[thick] (-8, 10.26) rectangle (-1, 11);
		\draw[thick] (-0, 10.26) rectangle (6, 11);
		
		\shade[shading=ball, ball color=black!50!] (-4,10.65) circle (.15);
		\shade[shading=ball, ball color=black!50!] (-3,10.65) circle (.15);
		\shade[shading=ball, ball color=black!50!] (3,10.65) circle (.15);
		
		\end{tikzpicture} 
	\end{center}
	\begin{center}
		\begin{tikzpicture}[scale=0.53]
		\draw [line width=1] (-9,10.5) -- (7,10.5) ; 
		\foreach \x in  {-9,-8,-7,-6,-5,-4,-3,-2,-1,0,1,2,3,4,5,6,7} 
		\draw[shift={(\x,10.5)},color=black, opacity=1] (0pt,4pt) -- (0pt,-4pt) node[below] {};
		\draw[] (-2.8,10.5) node[] {};
		
		\draw[] (-8,10.3) node[below] {\footnotesize{$x-\ell +1$}};
		\draw[] (-6,10.3) node[below] {\footnotesize{$x_1$}};
		\draw[] (-3,10.3) node[below] {\footnotesize{$x_2$}};
		\draw[] (-1,10.3) node[below] {\footnotesize{$x$}};
		\draw[] (0,10.3) node[below] {\footnotesize{$x+1$}};
		\draw[] (3,10.3)  node[below] {\footnotesize{$y$}};
		\draw[] (6,10.3) node[below] {\footnotesize{$x+ \varepsilon n$}};
		
		\draw[thick] (-8, 10.26) rectangle (-1, 11);
		\draw[thick] (-0, 10.26) rectangle (6, 11);
		
		\shade[shading=ball, ball color=black!50!] (-6,10.65) circle (.15);
		\shade[shading=ball, ball color=black!50!] (-3,10.65) circle (.15);
		\shade[shading=ball, ball color=black!50!] (3,10.65) circle (.15);
		\end{tikzpicture} 
	\end{center}
	\vspace*{0.2cm}  
	\caption{Path used to send a particle from site $x+1$ to $y$ inside the box of size $\varepsilon n$.}
	\label{figure-path}
\end{figure}
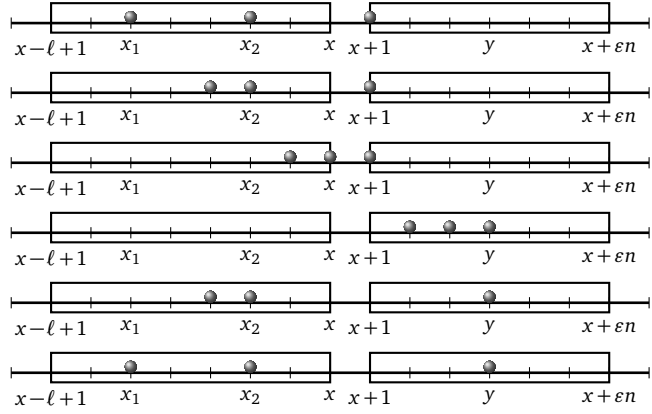
\begin{lemma}\label{engordandocaixa1}
Let $G^n_s: [0,1] \rightarrow \mathbb{R}$ be such that  $\|G_s^n\|_{\infty}\leq M<\infty$, for all $n\in \bb N$ and $s\in[0,T]$.  For any $t \in [0,T]$, $L=\varepsilon n$ and $\ell=n^{a-1-\delta}$ with $\delta >0$ such that $a-1-\delta \geq 0$, we have that
\begin{equation}\label{rl_bulk3}
\lim_{\varepsilon \to 0}\varlimsup_{n \to +\infty} \mathbb{E}_{\mu_{n}}\Bigg( \Bigg| \int_{0}^{t}\frac{1}{n}\sum_{x \in \Sigma_{n}^{\varepsilon}}G_{s}^{n}\left(\tfrac{x}{n}\right)\big\{ \overleftarrow{\eta}_{sn^2}^{\ell}(x) - \overleftarrow{\eta}_{sn^2}^{L}(x) \big\} \overrightarrow{\eta}^{\ve n}(x+1) \, ds  \Bigg| \Bigg) = 0.
\end{equation}
\end{lemma}
\begin{proof}
The proof follows exactly the argument of the proof of Lemma \ref{teorcontas2}. Again, letting $\nu^n_{\rho(\cdot)}$ be a Bernoulli product measure associated with the profile $\rho(\cdot)$ satisfying Lemma \ref{cond perfil}, the expectation in \eqref{rl_bulk3} can be bounded from above by $\frac{C(\alpha,\beta)}{B}$ plus
\begin{equation}\label{integr1}
\int_{0}^{t}\sup_{f}\Bigg( \Bigg| \int_{\Omega_n} \frac{1}{n}\sum_{x \in \Sigma_{n}^{\varepsilon}}G_{s}^{n}\left(\tfrac{x}{n}\right) \big\{ \overleftarrow{\eta}^{\ell}(x) - \overleftarrow{\eta}^{L}(x) \big\} \overrightarrow{\eta}^{\ve n}(x+1) f(\eta) \, d{\nu^n_{\rho(\cdot)}} \Bigg| + \frac{n}{B}\< L_n\sqrt{f}, \sqrt{f} \>_{\nu^n_{\rho(\cdot)}} \Bigg) \, ds,
\end{equation}
for any $B>0$. Take $L=\ell m$ with $m=\frac{\varepsilon n}{\ell}$ and note that 
\begin{equation*}
\begin{split}
\overleftarrow{\eta}^{\ell}(x) - \overleftarrow{\eta}^{L}(x)& = \frac{1}{m}\sum_{j=1}^{m-1}\Big(\overleftarrow{\eta}^\ell(x)-\overleftarrow{\eta}^\ell(x-j\ell)\Big).
\end{split}
\end{equation*}

From last identity, to bound the first integral inside the supremum in \eqref{integr1}, it is enough to estimate the term
\begin{equation*}\label{integr_new}
\frac{M}{mn}\sum_{x \in \Sigma_{n}^{\varepsilon}}\sum_{j=1}^{m-1}\Bigg| \int_{\Omega_n} \big\{ \overleftarrow{\eta}^{\ell}(x) - \overleftarrow{\eta}^{\ell}(x-j\ell) \big\} \overrightarrow{\eta}^{\ve n}(x+1) f(\eta) \, d{\nu^n_{\rho(\cdot)}} \Bigg|.
\end{equation*}
For $j=1, \ldots, m-1,$ let $X_2^j= \{\eta \in \Omega_{n}: \overleftarrow{\eta}^{\ell}(x) \geq \frac{2}{\ell} \}\cup  \{\eta \in \Omega_{n}: \overleftarrow{\eta}^{\ell}(x-j\ell) \geq \frac{2}{\ell} \}$. The integral in the previous display can be written as the integral over $X_2^j$ plus the integral over its complementary $(X_2^j)^c$. We observe that  the integral over $(X_2^j)^c$ is of order $O(\tfrac 1 \ell)$, and that we can write the integral over $X_2^j$ as
\begin{equation}\label{integr_new_new}
\frac{M}{mn}\sum_{x \in \Sigma_{n}^{\varepsilon}}\sum_{j=1}^{m-1}\Bigg| \int_{X_2^j}\frac{1}{\ell}\sum_{z\in\overleftarrow{\Lambda}_x^\ell}\Big(\eta(z)-\eta(z-j\ell)\Big)\overrightarrow{\eta}^{\ve n}(x+1) f(\eta) \, d{\nu^n_{\rho(\cdot)}} \Bigg|.
\end{equation}
Basically the idea above is to send a particle $z\in\overleftarrow{\Lambda}_{x}^\ell$ to a site inside a box of size $j\ell $, given that we have at least two particles in $\overleftarrow{\Lambda}_{x}^{\ell}$ or $\overleftarrow{\Lambda}_{x-j\ell}^{\ell}$, see Figure \ref{RL path L}. We stress that the path argument used here is the same used above to prove Lemma \ref{teorcontas2}.

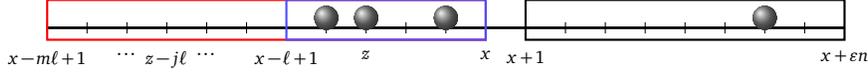
\begin{figure}[h!]
\begin{center}
\begin{tikzpicture}[scale =0.53]
\draw [line width=1] (-10,10.5) -- (10,10.5) ; 
\foreach \x in  {-10,-9,-8,-7,-6,-5,-4,-3,-2,-1,0,1,2,3,4,5,6,7,8,9,10} 
\draw[shift={(\x,10.5)},color=black, opacity=1] (0pt,4pt) -- (0pt,-4pt) node[below] {};
\draw[] (-2.8,10.5) node[] {};

\draw[] (-10,10.15) node[below] {\scriptsize{$x-m\ell +1$}};
\draw[] (-8,10.15) node[below] {\scriptsize{$\cdots$}};
\draw[] (-7,10.15) node[below] {\scriptsize{$z-j\ell$}};
\draw[] (-6,10.15) node[below] {\scriptsize{$\cdots$}};
\draw[] (-4,10.15) node[below] {\scriptsize{$x-\ell+1$}};
\draw[] (-2,10.15) node[below] {\scriptsize{$z$}};
\draw[] (1,10.15) node[below] {\scriptsize{$x$}};
\draw[] (2,10.15) node[below] {\scriptsize{$x+1$}};
\draw[] (10,10.15) node[below] {\scriptsize{$x+ \ve n$}};

\draw[thick, color=red] (-10, 10.2) rectangle (1, 11.2);
\draw[thick, color=blue!70!] (-4, 10.2) rectangle (1, 11.2);
\draw[thick] (2, 10.2) rectangle (10, 11.2);

\node[ball color=black!50!, shape=circle, minimum size=0.3cm] (U) at (-3,10.75) {};
\node[ball color=black!50!, shape=circle, minimum size=0.3cm] (U) at (-2,10.75) {};
\node[ball color=black!50!, shape=circle, minimum size=0.3cm] (U) at (0,10.75) {};
\node[ball color=black!50!, shape=circle] (K) at (8,10.75)  {};


\end{tikzpicture} 
\end{center}
\caption{Sending a particle from site $z$ to $z-j\ell$.}
\label{RL path L}
\end{figure}
Summing and subtracting $\tfrac{1}{2}f(\eta^{z-j\ell,z})$ in \eqref{integr_new_new}, we rewrite \eqref{integr_new_new} as: 

\begin{equation}\label{integr2}
\begin{split}
&\frac{M}{2mn\ell}\sum_{x \in \Sigma_{n}^{\varepsilon}}\sum_{j=1}^{m-1}\sum_{z\in\overleftarrow{\Lambda}_x^\ell} \Bigg|\int_{X_2^j}\big(\eta(z)-\eta(z-j\ell)\big)\overrightarrow{\eta}^{\ve n}(x+1)\big(f(\eta)+f(\eta^{z-j\ell,z})\big) \, d \nu^n_{\rho(\cdot)}\Bigg|\\
+&\frac{M}{2mn \ell}\sum_{x \in \Sigma_{n}^{\varepsilon}}\sum_{j=1}^{m-1} \sum_{z\in\overleftarrow{\Lambda}_x^\ell} \Bigg|\int_{ X_2^j} \big(\eta(z)-\eta(z-j\ell)\big)\overrightarrow{\eta}^{\ve n}(x+1)\big(f(\eta)-f(\eta^{z-j\ell,z})\big) \, d \nu^n_{\rho(\cdot)}\Bigg|.
\end{split}
\end{equation}
Note that, as in Lemma \ref{teorcontas2}, using the fact that $\rho(\cdot)$ is Lipschitz, the first term of \eqref{integr2} is bounded from above by a constant times $\frac{m \ell}{n}$. Since $m=\tfrac{\varepsilon n}{\ell}$, that term is of order $O(\ve)$. It remains to estimate the second term in \eqref{integr2}. The idea is to exchange a particle from the site $z$ to the site $z-j\ell$. This can be done since we are restricted to the set $X_{2}^{j}$, so that we know that there are at least two particles either in the box $\overleftarrow{ \Lambda}^\ell_x$ or in the box $\overleftarrow{ \Lambda}^\ell_{x-j\ell}$. With this in mind, we can again construct a path using the SSEP jumps to create a mobile cluster in the box where there are for sure two particles. Now, we use the PMM jumps to move the mobile cluster close to the particle at site $z$, and to send it to the site $z-j\ell$. Then, we put the mobile cluster back to its starting point using the PMM jumps, and we then put the two particles back to their initial position using the SSEP jumps. As in the previous lemma, for $A, \; \tilde{A} >0$, we can bound  the second term in \eqref{integr2} from above by a constant, times 
\begin{equation*}
\frac{\ell}{A} +  A D_{S}(\sqrt{f},\nu^n_{\rho(\cdot)}) + \frac{\ell m}{\tilde A}+\tilde AD_{P}(\sqrt{f},\nu^n_{\rho(\cdot)}).
\end{equation*}
By choosing $A=\tfrac{n^{a-1}}{B}$ and $\tilde A= \tfrac {n}{B}$,  we can bound \eqref{integr1} from above by a constant, times 
\begin{equation}\label{termos2}
\dfrac {1}{B}+T\Bigg(\varepsilon+ \dfrac{\ell B}{n^{a-1}}+\dfrac{\ell m B}{n}\Bigg).
\end{equation}
From the choice of $\ell$ and $m$, \eqref{termos2} can be bounded from above by $\tfrac {1}{B}+T(\varepsilon + n^{-\delta}B + \varepsilon B)$, which vanishes when we take  $n \to +\infty$, then $\varepsilon\to0$ and finally  $B \to +\infty$.
\end{proof}
\subsection{Replacement lemmas at the boundary}\label{replace3}
In this subsection we prove the replacement lemmas that are necessary for the boundary terms. Throughout this subsection let $\rho(\cdot)$ be a profile satisfying Lemma \ref{cond perfil}. 
\begin{lemma}\label{RL first} 
Fix $\theta < 1$. Let $\varphi: \Omega_n \rightarrow \Omega_n $ be a positive and bounded function which does not depend on the value of the configuration $\eta$ at the site $1$. For any $t\in [0,T]$, we have that
\begin{equation*}
\varlimsup_{n \to +\infty} \mathbb{E}_{\mu_{n}}\Bigg( \Bigg| \int_{0}^{t}\varphi(\eta_{sn^2})(\alpha - \eta_{sn^2}(1)) \, ds \Bigg| \Bigg) = 0.
\end{equation*}
The same is true replacing $\alpha$ by $\beta$, $1$ by $n-1$ and requiring $\varphi$ not to depend on $\eta$ at the site $n-1$.
\end{lemma}
\begin{proof}
As in the previous replacement lemmas, we have that the expectation in the statement of the theorem is bounded from above by $\frac{C(\alpha,\beta)}{B}$, plus
\begin{equation}\label{termos para controlar}
\begin{split}
T\sup_{f} \Bigg( \Bigg| \int_{\Omega_n} \varphi(\eta)(\alpha-\eta(1))f(\eta) \, d\nu^{n}_{\rho(\cdot)} \Bigg| + \dfrac{n}{B}\langle L_{n}\sqrt{f},\sqrt{f} \rangle_{\nu^{n}_{\rho(\cdot)}} \Bigg),
\end{split}
\end{equation}
where $B>0$ and the supremum is carried over all the densities $f$ with respect to $\nu^{n}_{\rho(\cdot)}$. Summing and subtracting $\tfrac{1}{2}f(\eta^1)$ in the first integral term inside the supremum in \eqref{termos para controlar}, we can bound this integral term from above by
\begin{equation}\label{termos3}
\dfrac{1}{2} \Bigg| \int_{\Omega_n} \varphi(\eta)(\alpha-\eta(1))\big(f(\eta)-f(\eta^1)\big) \, d\nu^{n}_{\rho(\cdot)} \Bigg| + \dfrac{1}{2} \Bigg| \int_{\Omega_n} \varphi(\eta)(\alpha-\eta(1))\big(f(\eta)+f(\eta^1)\big) \, d\nu^{n}_{\rho(\cdot)} \Bigg|.
\end{equation}
Since $\varphi$ is bounded, from Young's inequality and from similar computations made in Theorem \ref{replace-bulk}, the first term in \eqref{termos3} is bounded from above by a constant, times
\begin{equation}\label{RL bound boundary}
\frac{A}{4} + \frac{1}{4A}F^{\alpha}_{1}(\sqrt{f},\nu^{n}_{\rho(\cdot)}),
\end{equation}
where $A>0$ and $F^{\alpha}_{1}(\sqrt{f},\nu^{n}_{\rho(\cdot)})$ is defined in \eqref{I_terms}. To bound the remaining term in \eqref{termos3} we follow exactly the same idea used to bound the second expression in \eqref{termLim}. Then, after some computations we have that this term is bounded from above by a constant times $|\alpha - \rho(\tfrac{1}{n})|$. Now, from \eqref{measure_prof}, \eqref{RL bound boundary}, and with the choice $A=Bn^{\theta-1}m^{-1}$, we have that \eqref{termos para controlar} is bounded from above by a constant, times
$$\dfrac{Bn^{\theta-1}}{4m} + \Big|\rho(\tfrac{1}{n})-\alpha\Big|.$$
Taking $n \to +\infty$ and using the fact that $\rho(\cdot)$ satisfies Lemma \ref{cond perfil}, we have that these terms vanish since $\theta<1.$
\end{proof}
\begin{theorem} \label{RLbound1}
For any $t\in [0,T]$,  we have 
\begin{equation}\label{replacement_left_boundary}
\lim_{\varepsilon \to 0}\varlimsup_{n \to +\infty} \mathbb{E}_{\mu_{n}}\Bigg( \Bigg| \int_{0}^{t}\big\{ \eta_{sn^2}(1)\eta_{sn^2}(2)-\overrightarrow{\eta}_{sn^2}^{\varepsilon n}(1)\overrightarrow{\eta}_{sn^2}^{\varepsilon n }(\varepsilon n+1)\big\} \, ds \Bigg| \Bigg) = 0
\end{equation}
and
\begin{equation}\label{replacement_right_boundary}
\lim_{\varepsilon \to 0}\varlimsup_{n \to +\infty} \mathbb{E}_{\mu_{n}}\Bigg( \Bigg| \int_{0}^{t}\big\{ \eta_{sn^2}(n-1)\eta_{sn^2}(n-2)-\overleftarrow{\eta}_{sn^2}^{\varepsilon n}(n-1)\overleftarrow{\eta}_{sn^2}^{\varepsilon n}(n-1-\varepsilon n)\big\} \, ds \Bigg| \Bigg) = 0.
\end{equation}
\end{theorem} 
For simplicity of the presentation, we only prove \eqref{replacement_left_boundary}, that is, the left boundary part. We note that the result concerning the right boundary in \eqref{replacement_right_boundary} can be proved with an analogous argument, therefore we leave the details to the reader. We divide the proof of \eqref{replacement_left_boundary} in the following steps:
\begin{enumerate}
\item[1)] replace $\eta(1)\eta(2)$ by $\eta(1)\eta(\ell+1)$, for $\ell = n^{a-1-\delta}$; (Lemma \ref{lemmaFronteira1})
  \vspace{0.1cm}
  
\item[2)] replace $\eta(1)\eta(\ell +1)$ by $\overrightarrow{\eta}^{\ell}(1)\eta(\ell +1)$, for $\ell = n^{a-1-\delta}$; (Lemma \ref{lemmaFronteira2})
  \vspace{0.1cm}
   
\item[3)] replace $\overrightarrow{\eta}^{\ell}(1)\eta(\ell +1)$ by $\overrightarrow{\eta}^{\ell}(1)\overrightarrow{\eta}^{\varepsilon n}(\varepsilon n +1)$, for $\ell = n^{a-1-\delta}$; (Lemma \ref{lemmaFronteira3})
  \vspace{0.1cm}
  
\item[4)] replace $\overrightarrow{\eta}^{\ell}(1)\overrightarrow{\eta}^{\varepsilon n}(\varepsilon n +1)$ by $\overrightarrow{\eta}^{L}(1)\overrightarrow{\eta}^{\varepsilon n}(\varepsilon n +1)$, for $\ell = n^{a-1-\delta}$ and $L=\varepsilon n$. (Lemma \ref{lemmaFronteira4})
\end{enumerate}
\begin{lemma}\label{lemmaFronteira1}
For any $t \in [0,T]$,   $\ell=n^{a-1-\delta}$ with $\delta >0$ such that $a-1-\delta \geq 0$, we have 
\begin{equation*}\label{shif}
\varlimsup_{n \to +\infty} \mathbb{E}_{\mu_{n}}\Bigg( \Bigg| \int_{0}^{t}\eta_{sn^2}(1)\big\{ \eta_{sn^2}(2)-\eta_{sn^2}(\ell+1)\big\} \, ds \Bigg| \Bigg) = 0.
\end{equation*}
\end{lemma}
\begin{proof}
Following the same steps of the proof of Lemma \ref{teocontas}, the expectation in the statement of the lemma is bounded from above by $\frac{C(\alpha,\beta)}{B}$, plus
\begin{equation}\label{shift2}
T \sup_{f} \Bigg( \Bigg| \int_{\Omega_n} \eta(1)\big\{\eta(2)- \eta(\ell+1)\big\}f(\eta) \, d\nu^n_{\rho(\cdot)}\Bigg| + \dfrac{n}{B} \langle L_n\sqrt{f}, \sqrt{f} \rangle_{\nu^n_{\rho(\cdot)}} \Bigg),
\end{equation}
where $B>0$ and the supremum is carried over all the densities $f$ with respect to $\nu^{n}_{\rho(\cdot)}$. Write $\eta(2)-\eta(\ell+1) = \sum_{y=2}^{\ell} \eta(y)-\eta(y+1)$. Using the same strategy that we used to bound the term in \eqref{termLim}, for $A>0$, the first term inside the supremum in \eqref{shift2} is bounded from above by a constant, times
\begin{equation}\label{shift3}
\frac{\ell}{n} + \frac{\ell}{A} + AD_{S}(\sqrt{f},\nu^n_{\rho(\cdot)}).
\end{equation}
With the choice $A= \frac{n^{a-1}}{B}$, from \eqref{eq:prices_DF}, \eqref{shift2}, and \eqref{shift3}, we have that the expectation in the statement of the lemma is bounded from above by a constant times
\begin{equation*}
\frac{1}{B} + T\Bigg( \frac{\ell}{n} + \frac{\ell B}{n^{a-1}} \Bigg).
\end{equation*}
Taking $n \to +\infty$, and from the choice of $\ell$, we have that the right-hand side of last expression vanishes. By sending $B \to +\infty$ we finish the proof.
\end{proof}
\begin{corollary}\label{cor_5.12}
	For any $t \in [0,T]$,     we have 
	\begin{equation*}\label{shif_1}
	\varlimsup_{n \to +\infty} \mathbb{E}_{\mu_{n}}\Bigg( \Bigg| \int_{0}^{t}\big\{ \eta_{sn^2}(1)-\eta_{sn^2}(2)\big\} \, ds \Bigg| \Bigg) = 0
	\end{equation*} and
		\begin{equation*}\label{shif_2}
	\varlimsup_{n \to +\infty} \mathbb{E}_{\mu_{n}}\Bigg( \Bigg| \int_{0}^{t}\big\{ \eta_{sn^2}(n-1)-\eta_{sn^2}(n-2)\big\} \, ds \Bigg| \Bigg) = 0.
	\end{equation*}
\end{corollary}

\begin{proof}
	The reader can repeat the proof of Lemma \ref{lemmaFronteira1} taking $\ell=1$ and replacing $\eta(2)$ by $\eta(1)$.  
\end{proof}

\begin{lemma}\label{lemmaFronteira2}
For any $t \in [0,T]$, $\ell=n^{a-1-\delta}$ with $\delta >0$ such that $a-1-\delta \geq 0$, we have
\begin{equation}\label{shif4}
\varlimsup_{n \to +\infty} \mathbb{E}_{\mu_{n}}\Bigg( \Bigg| \int_{0}^{t}\big\{\eta_{sn^2}(1) - \overrightarrow{\eta}_{sn^2}^{\ell}(1)\big\}\eta_{sn^2}(\ell+1) \, ds \Bigg| \Bigg) = 0.
\end{equation}
\end{lemma}
\begin{proof}
Following the same steps of previous lemmas, we have that the expectation in \eqref{shif4} is bounded from above by $\frac{C(\alpha,\beta)}{B}$, plus
\begin{equation*}\label{shif5}
T\sup_{f} \Bigg( \Bigg| \int_{\Omega_n} \big\{\eta(1) - \overrightarrow{\eta}^{\ell}(1)\big\}\eta(1+\ell)f(\eta)\, d\nu^n_{\rho(\cdot)}\Bigg| + \frac{n}{B} \langle L_n\sqrt{f}, \sqrt{f} \rangle_{\nu^n_{\rho(\cdot)}} \Bigg),
\end{equation*}
where $B>0$ and the supremum is carried over all densities $f$ with respect to $\nu^{n}_{\rho(\cdot)}$. Now, following exactly the same computations done in the proof of Lemma \ref{teocontas}, the expectation in the statement of the lemma is bounded from above by a constant times
\begin{equation*}
\frac{1}{B} + T\Bigg( \frac{\ell}{n} + \frac{\ell B}{n^{a-1}} \Bigg).
\end{equation*}
Taking $n \to +\infty$ and then $B \to +\infty$, the expression above vanishes due to our choice of $\ell$.
\end{proof}
\begin{lemma}\label{lemmaFronteira3}
For any $t \in [0,T]$, $\delta >0$ and  $\ell=n^{a-1-\delta}$ such that $a-1-\delta \geq 0$, we have
\begin{equation}\label{shif6}
\lim_{\varepsilon \to 0}\varlimsup_{n \to +\infty} \mathbb{E}_{\mu_{n}}\Bigg( \Bigg| \int_{0}^{t}\overrightarrow{\eta}_{sn^2}^{\ell}(1)\big\{ \eta_{sn^2}(\ell+1)-\overrightarrow{\eta}_{sn^2}^{\varepsilon n}(\varepsilon n+1) \big\} \, ds \Bigg| \Bigg) = 0.
\end{equation}
\end{lemma}
\begin{proof}
Following the same steps of the proof of Lemma \ref{teocontas}, we have that the expectation in \eqref{shif6} is bounded from above by $ \frac{C(\alpha,\beta)}{B}$, plus
\begin{equation*}\label{shift7}
T \sup_{f} \Bigg( \Bigg| \int_{\Omega_n} \overrightarrow{\eta}^{\ell}(1)\big\{ \eta(\ell+1)-\overrightarrow{\eta}^{\varepsilon n}(\varepsilon n+1)\big\}f(\eta) \, d\nu^n_{\rho(\cdot)}\Bigg| + \frac{n}{B} \langle L_n\sqrt{f}, \sqrt{f} \rangle_{\nu^n_{\rho(\cdot)}} \Bigg),
\end{equation*}
where $B>0$ and the supremum is carried over all densities $f$ with respect to $\nu^{n}_{\rho(\cdot)}$. Let $X_3 = \{ \eta\in \Omega_n : \overrightarrow{\eta}^{\ell}(1) \geq \tfrac{2}{\ell} \}$. Write the first integral inside the supremum as the integral over the set $X_3$ plus the integral over its complementary $X_3^c$. Note that $$\eta(\ell+1)-\overrightarrow{\eta}^{\varepsilon n}(\varepsilon n+1)=\frac{1}{\varepsilon n}\sum_{y= \varepsilon n +1}^{2\varepsilon n}\eta(\ell+1)-\eta(y).$$
Now, following the same computations done in  the proof of Lemma \ref{teorcontas2}, we have that the expectation in the statement of the lemma is bounded from above by a constant times
\begin{equation*}
\dfrac{1}{B} + T\Bigg( \dfrac{\ell}{n} + \dfrac{1}{n} + \varepsilon + \dfrac{\ell B}{n^{a-1}} + \varepsilon B \Bigg).
\end{equation*}
Taking $n \to +\infty$, then $\varepsilon \to 0$, and finally $B \to +\infty$, the result follows due to our choice of $\ell$.
\end{proof}
\begin{lemma}\label{lemmaFronteira4}
For any $t \in [0,T]$, $L=\varepsilon n$ and $\ell=n^{a-1-\delta}$ with $\delta >0$ such that $a-1-\delta \geq 0$, we have
\begin{equation}\label{shif8}
\lim_{\varepsilon \to 0}\varlimsup_{n \to +\infty} \mathbb{E}_{\mu_{n}}\Bigg( \Bigg| \int_{0}^{t}\big\{  \overrightarrow{\eta}_{sn^2}^{\ell}(1) - \overrightarrow{\eta}_{sn^2}^{L}(1)\big\} \overrightarrow{\eta}_{sn^2}^{\varepsilon n}(\varepsilon n +1) \, ds \Bigg| \Bigg) = 0.
\end{equation}
\end{lemma}
\begin{proof}
Following the same steps of Lemma \ref{engordandocaixa1}, we have that the expectation in \eqref{shif8} is bounded from above by $\frac{C(\alpha,\beta)}{B} $ plus
\begin{equation*}\label{shift9}
T \sup_{f} \Bigg( \Bigg| \int_{\Omega_n} \big\{ \overrightarrow{\eta}^{\ell}(1) - \overrightarrow{\eta}^{L}(1) \big\} \overrightarrow{\eta}^{\varepsilon n}(\varepsilon n+1)f(\eta) \, d\nu^n_{\rho(\cdot)}\Bigg| + \frac{n}{B} \langle L_n\sqrt{f}, \sqrt{f} \rangle_{\nu^n_{\rho(\cdot)}} \Bigg),
\end{equation*}
where $B>0$ and the supremum is carried over all the densities $f$ with respect to $\nu^{n}_{\rho(\cdot)}$. Take $L=\ell m$ with $m=\tfrac{\varepsilon n}{\ell}$. As in Lemma \ref{engordandocaixa1}, let $X_4^j= \{\eta \in \Omega_{n}: \overrightarrow{\eta}^{\ell}(1) \geq \frac{2}{\ell} \}\cup  \{\eta \in \Omega_{n}: \overrightarrow{\eta}^{\ell}(1+j\ell) \geq \frac{2}{\ell} \}$. Now, following exactly the same computations done in the proof of that lemma,  we have that the expectation in \eqref{shif8} is bounded from above by a constant times
\begin{equation*}
\frac{1}{B} + T\Bigg( \varepsilon + \frac{\ell B}{n^{a-1}} + B\ve \Bigg).
\end{equation*}
Taking $n \to +\infty$, then $\varepsilon \to 0$, and $B \to +\infty$, the result follows due to our choice of $\ell$ and $m$.
\end{proof}

\begin{lemma} \label{RLbound2}
For any $t\in [0,T]$  we have 
\begin{equation*}\label{replacement_left_boundary2}
\lim_{\varepsilon \to 0}\varlimsup_{n \to +\infty} \mathbb{E}_{\mu_{n}}\Bigg( \Bigg| \int_{0}^{t} \big\{ \eta_{sn^2}(1)-\overrightarrow{\eta}_{sn^2}^{\varepsilon n}(1)\big\} \, ds \Bigg| \Bigg) = 0
\end{equation*}
and
\begin{equation*}\label{replacement_right_boundary2}
\lim_{\varepsilon \to 0}\varlimsup_{n \to +\infty} \mathbb{E}_{\mu_{n}}\Bigg( \Bigg| \int_{0}^{t}\big\{ \eta_{sn^2}(n-1)-\overleftarrow{\eta}_{sn^2}^{\varepsilon n}(n-1)\big\} \, ds \Bigg| \Bigg) = 0.
\end{equation*}
\end{lemma}
\begin{proof}
This proof is similar to the proof presented in Lemma \ref{teocontas} and it has two steps. The first one is to replace $\eta(1)$ by $ \overrightarrow{\eta}_{sn^2}^{\ell}(1)$ and the second one is to replace $\overrightarrow{\eta}_{sn^2}^\ell(1)$ by $ \overrightarrow{\eta}_{sn^2}^{\varepsilon n}(1)$. We leave the details to the reader.
\end{proof}

	\subsection{Fixing the profile at the boundary for the case $\theta<1$}\label{replace4}
In this subsection we intend to prove item $3.$ in Definition \ref{Def. Dirichlet}, that is, $\rho_t(0)=\alpha$ and $\rho_t(1)=\beta$ for all $t\in (0,T]$. We note that it is a simple observation to show that these facts are a consequence of combining both Lemma \ref{RL first} with $\varphi\equiv 1$ and Lemma \ref{RLbound2}.  We refer the interested reader to Appendix A.4  of \cite{patricia1}.


\section{Energy estimates}
\label{sec: energy} 

The idea of this section is to prove that any limit point $\mathbb{Q}$ of the sequence $\{ \mathbb{Q}_n \}_{n\in \mathbb{N}}$ is concentrated on trajectories $\rho_t(u)du$, in which $\rho^2$ belongs to $L^2(0,T; \mathcal{H}^1)$, see Definition \ref{Def. Sobolev space}.  Since our model is an exclusion process it is standard to check that  $\mathbb Q$ is supported on trajectories of measures that are absolutely continuous with respect to the Lebesgue measure, that is  
$\pi_t(du)=\rho_t(u)\,du$, for all $t\in[0,T]$ where $\rho:[0, T]\times [0,1]\to[0,1]$. Knowing this, we define the linear functional $\<\!\<\rho^2, \cdot\>\!\>$ on $C_{0}^{0,1}([0,T]\times (0,1))$ by
$$ \<\!\<\rho^2,G\>\!\>:=\int_0^T\int_0^1(\rho_s(u))^2\,G_s(u)\,du\, ds =\int_0^T\<(\rho_s)^2, G_s\> \, ds.$$
The next  proposition  shows that $\<\!\<\rho^2,\cdot\>\!\>$ is $\mathbb{Q}$- almost surely continuous. Therefore, the linear functional can be extended to $L^{2}([0,T]\times(0,1))$. Furthermore, by the Riesz's Representation Theorem we can find $\xi \in L^{2}([0,T]\times(0,1))$ such that
$$\<\!\<\rho^2,G\>\!\> = - \int_{0}^{T}\int_{0}^{1}G_s(u)\xi_s(u)\,duds,$$
for all $G \in C_{0}^{0,1}([0,T]\times (0,1))$, which implies $\rho^2 \in L^2(0,T; \mathcal{H}^1)$.

\begin{proposition}\label{prop:energy_estimate}
There exist positive constants $K_0$ and $c$ such that
\begin{equation*}
\mathbb{E}^{\bb Q}\Bigg( \sup _{G\in C^{0,1}_{c}([0,T]\times (0,1))}\Big( \<\!\<\rho^2, \p_u G\>\!\>-c \<\!\<G, G\>\!\>\Big)\Bigg) \,\leq \, K_0 < \infty,
\end{equation*}
where $\bb Q$ is a limit point of $\bb Q_{n}$.
\end{proposition}

\begin{proof} 
By density and by the Monotone Convergence Theorem, it is enough to prove that for a countable dense subset $\lbrace  G^{m}\rbrace_{m \in \mathbb{N}}$ of $C_{c}^{0,1}([0,T]\times (0,1))$ it holds that 
\begin{equation}\label{a_provar}
\mathbb{E}^{\bb Q}\Bigg( \max _{k \leq m} \big(  \<\!\<\rho^2, \p_u G^k\>\!\>-c \,\<\!\<G^k, G^k\>\!\> \big)\Bigg) \, \leq \, K_{0},
\end{equation}
for any $m$ and for some $K_{0}$ independent of $m$. Using arguments similar to ones used in \eqref{approximation} and \eqref{4.12}, it is enough to prove that
\begin{equation}\label{a_provar_1}
\varlimsup_{\ve\to 0}\mathbb{E}^{\bb Q}\Bigg( \max _{k \leq m} \Bigg( \int_0^T\Big\{\int_\ve^{1-\ve}\< \rho_s, \overrightarrow\iota^u_\ve\>\< \rho_s, \overleftarrow\iota^u_\ve\>\;\p_u G^k_s(u) \,\,du \,-c \|G^k_s \|_{2}^{2}\,\Big\}\,ds \Bigg)\Bigg) \, \leq \, K_{0},
\end{equation}
where $\Vert\cdot\Vert_2$ is the norm of $L^2[0,1]$.
Note that the function that associates to a trajectory $\pi_\cdot \in \mc D([0,T], \mc M_+)$ the number 
$$\max _{k \leq m} \Bigg( \int_0^T\Big\{\int_\ve^{1-\ve}\< \pi_s, \overrightarrow\iota^u_\ve\>\< \pi_s, \overleftarrow\iota^u_\ve\>\;\p_u G^k_s(u) \,\,du -c \|G^k_s \|_{2}^{2}\,\Big\}\,ds\Bigg)$$ is lower semi-continuous and bounded with respect to the Skorokhod topology of $ \mc D([0,T], \mc M_+)$. For that reason,  \eqref{a_provar_1} is bounded from above by 
\begin{eqnarray}\label{633}
\varlimsup_{\ve\to 0}\varliminf _{n\rightarrow +\infty} \bb E _{\mu _{n}}\Bigg( \max _{k \leq m} \Bigg( \int_0^T\Big\{\int_\ve^{1-\ve}\< \pi_s^n, \overrightarrow\iota^u_\ve\>\< \pi_s^n, \overleftarrow\iota^u_\ve\>\;\p_u G^k_s(u) \,\,du -c \|G^k_s \|_{2}^{2}\,\Big\}\,ds\Bigg) \Bigg).
\end{eqnarray}
Now, we use the facts that the error from changing the integral in the space variable by its Riemann sum is of order $O(\tfrac 1n)$ and 
$$\< \pi^n_s, \overrightarrow\iota^{x/n}_\ve\>\< \pi^n_s, \overleftarrow\iota^{x/n}_\ve\>=
\overrightarrow{\eta}^{\ve n}_{sn^2}(x+1)\overleftarrow{\eta}^{\ve n}_{sn^2}(x)+O\left(\tfrac{1}{\ve n}\right),$$
 to rewrite \eqref{633} as
\begin{eqnarray*}
	\varlimsup_{\ve\to 0}\varliminf_{n\rightarrow +\infty} \bb E _{\mu _{n}}\Bigg( \max _{k \leq m} \Bigg( \int_0^T\Big\{\dfrac{1}{n}\sum_{x\in\Sigma_n^\ve}\overrightarrow{\eta}^{\ve n}_{sn^2}(x+1)\overleftarrow{\eta}^{\ve n}_{sn^2}(x)\;\p_u G^k_s(\pfrac{x}{n})  -c\|G^k_s \|_{2}^{2}\,\Big\}\,ds \Bigg) \Bigg).
\end{eqnarray*}
It follows from Theorem \ref{replace-bulk} that in order to prove this proposition we just need to show that
\begin{eqnarray*}
	\varlimsup_{\ve\to 0}\varliminf_{n\rightarrow +\infty} \bb E _{\mu _{n}}\Bigg( \max _{k \leq m} \Bigg( \int_0^T\Big\{\dfrac{1}{n}\sum_{x=1}^{n-2}\eta_{sn^2}(x+1)\eta_{sn^2}(x)\;\p_u G^k_s(\pfrac{x}{n})  -c \|G^k_s \|_{2}^{2}\,\Big\}\,ds\Bigg) \Bigg) \,\leq K_0.
\end{eqnarray*}
The limits in the  sum  above changed  from $\Sigma_n^\ve=\{1+\varepsilon n, \ldots, n-1-\varepsilon n\}$ to $\{1, \dots, n-2\}$ because the error of this change is of order $O(\ve)$. We use entropy's and Jensen's  inequality to change the initial measure from $\mu_{n}$ to $\nu_{\rho(\cdot)}$, where $\rho(\cdot)$ satisfies the hypotheses of Lemma \ref{cond perfil}. Moreover, the fact that $\exp\left\{\max_{k\leq m} a_{k}\right\}\leq \sum_{k=1}^{m}\exp\{a_{k}\}$ and \eqref{expliENTROPY} imply that the expectation in the previous display is bounded from above by 
\begin{equation*}
C(\alpha,\beta) + \dfrac{1}{n}\log \mathbb{E}_{\nu_{\rho(\cdot)}^n}\Bigg( \sum_{k =1}^{m} \exp \Bigg\{ \int_0^T\Big\{\sum_{x=1}^{n-2}\eta_{sn^2}(x+1)\eta_{sn^2}(x)\;\p_u G^k_s(\pfrac{x}{n})  -cn\|G^k_s\|_{2}^{2}\Big\}\,ds\Bigg\} \Bigg),
\end{equation*}
where $C(\alpha, \beta)$ is a constant which depends on $\alpha$ and $\beta$. By the linearity of the expectation and  the property  $\varlimsup_{n \to +\infty} n^{-1}\log(a_n + b_n) = \max\left( \varlimsup_{n \to +\infty} n^{-1}\log(a_n), \, \varlimsup_{n \to +\infty} n^{-1}\log(b_n) \right),$
to study the second term in the previous display it is enough to bound the term 
\begin{equation}\label{EE2}
\dfrac{1}{n}\log \mathbb{E}_{\nu_{\rho(\cdot)}^n}\Bigg( \exp \Bigg\{ \int _{0}^{T}\Big\{\sum_{x=1}^{n-2}\eta_{sn^2}(x+1)\eta_{sn^2}(x)\;\p_u G_s(\pfrac{x}{n})  - cn\|G_s \|_{2}^{2} \Big\}\,ds \Bigg\} \Bigg)
\end{equation}
by a constant independent of $G\in C_{c}^{0,1}([0,T]\times (0,1))$.
Therefore, by the Feynman-Kac's formula, the expression \eqref{EE2} is bounded from above by 
\begin{equation} 
\label{EE3}
\int _{0}^{T}\sup _{f}\Bigg(\dfrac{1}{n}\int_{\Omega_{n}}\sum_{x=1}^{n-2}\eta(x+1)\eta(x)\;\p_u G_s(\pfrac{x}{n}) f(\eta) \, d {\nu_{\rho(\cdot)}^n} - c\Vert G_s \Vert_{2}^{2}  + n \langle L_{n}\sqrt{f},\sqrt{f}\rangle_{{\nu_{\rho(\cdot)}^n}} \Bigg) \, ds,
\end{equation}
where the supremum is carried over all the densities $f$ with respect to $\nu_{\rho(\cdot)}^n$. Note that by a Taylor expansion on $G$, it is easy to see that we can replace its space derivative by the discrete gradient $\nabla^{+}_{n} G_{s}\left(\tfrac{x}{n}\right)$ plus an error of order $O\left(\frac 1n\right)$. Then, from a summation by parts, we obtain that the first term above is equal to
\begin{equation*}
\begin{split}
& \int_{\Omega_{n}} \sum_{x=2}^{n-2} G_{s}\left(\tfrac{x}{n}\right)\big\{\eta(x-1)\eta(x)-\eta(x)\eta(x+1)\big\}f(\eta) \, d{\nu_{\rho(\cdot)}^n} \\
+& \int_{\Omega_{n}} \big\{ G_{s}\left(\tfrac{n-1}{n}\right)\eta(n-2)\eta(n-1)-G_{s}\left(\tfrac{1}{n}\right)\eta(1)\eta(2)\big\}f(\eta) \, d{\nu_{\rho(\cdot)}^n}.
\end{split}
\end{equation*}
The last term above is negligible (order $O(\frac{1}{n})$), because $G\in C_{c}^{0,1}([0,T]\times (0,1))$, and  in the first one we add and subtract $\eta(x-1)\eta(x+1)$ to the expression inside braces in the sum and we get the two similar terms:
\begin{equation}\label{2terms}
\begin{split}
& \int_{\Omega_{n}} \sum_{x=2}^{n-2} G_{s}\left(\tfrac{x}{n}\right)\big\{\eta(x-1)\eta(x)-\eta(x-1)\eta(x+1)\big\}f(\eta) \, d{\nu_{\rho(\cdot)}^n} \\
+& \int_{\Omega_{n}} \sum_{x=2}^{n-2} G_{s}\left(\tfrac{x}{n}\right)\big\{\eta(x-1)\eta(x+1)-\eta(x)\eta(x+1)\big\}f(\eta) \, d{\nu_{\rho(\cdot)}^n}.
\end{split}
\end{equation}
We handle only with the first term above, because the second one is treated similarly.
By writing the first term above as one half of it plus one half of it, and in one of the halves we swap the occupation variables $\eta(x)$ and $\eta(x+1)$, last display becomes equal to
\begin{equation}
\begin{split}
\label{EE5}
&\dfrac{1}{2}\sum_{x=2}^{n-2} G_{s}\left(\tfrac{x}{n}\right)\int_{\Omega_{n}} \eta(x-1)\,\big\{\eta(x)-\eta(x+1)\big\}\big(f(\eta)-f(\eta^{x,x+1})\big) \, d{\nu_{\rho(\cdot)}^n}\\
+&\dfrac{1}{2}\sum_{x=2}^{n-2} G_{s}\left(\tfrac{x}{n}\right)\int_{\Omega_{n}} \eta(x-1)\,\big\{\eta(x)-\eta(x+1)\big\}\big(f(\eta)+f(\eta^{x,x+1})\big) \, d{\nu_{\rho(\cdot)}^n}.
\end{split}
\end{equation}
To treat the second term above we use similar computations to those performed in \eqref{termLim} and \eqref{function1}, and we can show that the integral in the last expression is bounded from above by 
$\dfrac{C}{n}$, where $C=C(\alpha,\beta)>0$. Therefore, the Young's inequality implies that the second term in \eqref{EE5} is bounded from above by
\begin{equation*}
\dfrac{C^2}{4}+\dfrac{1}{n}\sum_{x=2}^{n-1} \big(G_{s}\left(\tfrac{x}{n}\right)\big)^2.
\end{equation*}
Since the rate $p_{x,x+1}(\eta)=c_{x,x+1}(\eta)\big\{a_{x,x+1}(\eta)+a_{x+1,x}(\eta)\big\}$ does not degenerate in the first integral of \eqref{EE5}, we can
use Young's inequality in the first term of  \eqref{EE5} to obtain: 
\begin{equation*}
\begin{split}
&\dfrac{1}{4A}\sum_{x=2}^{n-2}\big(G_{s}\left(\tfrac{x}{n}\right)\big)^{2}\int_{\Omega_{n}}  \frac{\eta(x-1)}{\eta(x-1)+\eta(x+2)}\big(\sqrt{f(\eta)}+ \sqrt{ f(\eta^{x,x+1})}\big)^{2} \, d{\nu_{\rho(\cdot)}^n} \\&+\dfrac{A}{4}\sum_{x=2}^{n-2}  \int_{\Omega_{n}} p_{x,x+1}(\eta)\,\big(\sqrt{f(\eta)}- \sqrt{ f(\eta^{x,x+1})}\big)^{2}  \, d{\nu_{\rho(\cdot)}^n} \\
& \, \leq \, \dfrac{\tilde C}{4A}\sum_{x\in \Sigma_{n}} \left(G_{s}\left(\tfrac{x}{n}\right)\right)^{2}+ \dfrac{A}{4} D_{P}\left(\sqrt{f},\nu_{\rho(\cdot)}^n\right),
\end{split} 
\end{equation*}
for all  $A>0$. The constant $\tilde C=\tilde c(\alpha, \beta)$ comes from the bound in the Radon- Nikodym derivative of  interchange of variables in the  first integral above. Note that we can obtain the same bound for the second term of \eqref{2terms}. Thus, from the previous computations and  \eqref{eq:prices_DF}, the supremum \eqref{EE3} can be bounded from above by 
$$ \int_0^T\Bigg(\left(\dfrac{\tilde C}{2A}+\dfrac{1}{n}\right)\sum_{x\in \Sigma_{n}} \left(G_{s}\left(\tfrac{x}{n}\right)\right)^{2}+ \dfrac{A}{2} D_{P}\left(\sqrt{f},\nu_{\rho(\cdot)}^n\right)-c\,\Vert G_s\Vert_2^2- \dfrac{n}{4} D_{P}\left(\sqrt{f},\nu_{\rho(\cdot)}^n\right)+O(1)\Bigg)\,ds.$$
 Choosing $A=n/2$, last expression is equal to
\begin{equation*}\label{eq:en_est_bound}
 \int_0^T \Big\{ \dfrac{\tilde C+1}{n} \sum_{x\in \Sigma_{n}} \left(G_{s}\left(\tfrac{x}{n}\right)\right)^{2} \,ds  \; -\;  c \|G_s \|_{2}^{2}\,\Big\}\,ds
\end{equation*}
plus an error of order $O(1)$, which does not depend on $G$.  Since $\tfrac{1}{n} \sum_{x\in \Sigma_{n}} (G_{s}(\tfrac{x}{n}))^{2}$ converges to $\|G_s\|_2^2$, as $n\to+\infty$, then it is enough to choose $c>\tilde C+1$ to conclude that the last expression is $O(1)$. 
\end{proof}

\begin{lemma}\label{media_fronteira}
If $\rho^2\in L^2(0,T; \mathcal{H}^1)$, then:
\begin{equation}\label{aaa}
\vert\rho_{t}(u)-\rho_{t}(v)\vert\leq \Vert \partial_u(\rho_{t})^2\Vert_2\tfrac{(v-u)^{1/2}}{a}+a,
\end{equation}
for all $0\leq u\leq v\leq 1$, for almost  $t\in[0,T]$  and for all $a>0$. In particular, for all $\ve>0$
\begin{equation*}
\vert\rho_{t}(0)-\rho_{t}(\ve)\vert\leq \ve^{1/4}\Big( \Vert \partial_u(\rho_{t})^2\Vert _2+1\Big).
\end{equation*}The same inequality holds for $\rho_{t}(1)-\rho_{t}(1-\ve)$. Moreover, for each $\ve>0$, recalling the definition of $\<\rho_s,\overrightarrow\iota_\ve^u\>$ in \eqref{RC61}, we get
\begin{equation*}
\vert\rho_t(u)- \<\rho_t,\overrightarrow\iota_\ve^u\>\vert\leq \ve^{1/4}\Big(\tfrac{2}{3}\Vert \partial_u (\rho_t)^2\Vert_2+1\Big), 
\end{equation*}
for all $u\in[0,1]$ and for almost  $t\in[0,T]$.
Besides that the same inequality above holds for $ \overleftarrow\iota^u_\ve$ in the place of $ \overrightarrow\iota^u_\ve$.
\end{lemma}

\begin{proof}
  For $u\leq v$, we have
 \begin{equation}\label{aa}
\rho_{t}(u)-\rho_{t}(v)=\frac{\rho_{t}^2(u)-\rho_{t}^2(v)}{\rho_{t}(u)+\rho_{t}(v)+a} +
a\frac{\rho_{t}(u)-\rho_{t}(v)}{\rho_{t}(u)+\rho_{t}(v)+a},
 \end{equation}
 for all $a>0$.
 The first term on the right-hand side of the previous display is bounded from above by
 \begin{equation}\label{6.9}
 \frac{1}{a}\vert\rho_{t}^2(u)-\rho_{t}^2(v)\vert\,= \frac{1}{a}\Big\vert\int_{u}^{v}\partial_w(\rho_{t})^2(w)\,dw\Big\vert.
 \end{equation}
 The last equality is true for almost $t\in[0,T]$, because $\rho^2\in L^2(0,T; \mathcal{H}^1)$.
 By Cauchy-Schwarz inequality, \eqref{6.9} is bounded from above by  $\Vert \partial_u(\rho_{t})^2\Vert_2\tfrac{(v-u)^{1/2}}{a}$. Now, to treat the second term on the right-hand side of \eqref{aa}, we observe that
$$ a\frac{\rho_{t}(u)-\rho_{t}(v)}{\rho_{t}(u)+\rho_{t}(v)+a} \, =\,\,
a\,-\, a\frac{2\rho_{t}(v)+a}{\rho_{t}(u)+\rho_{t}(v)+a}  \;\leq\;a,$$
because the second term on the left-hand side of last equality is negative. 
From this, we get \eqref{aaa}.
	By \eqref{RC61}, the difference $\rho_t(u)- \<\rho_t,\overrightarrow\iota_\ve^u\>$ can be written as 
\begin{equation*}
\frac{1}{\ve}\int_{u}^{u+\ve}[\rho_{t}(u)-\rho_{t}(v)] \,dv.
\end{equation*}
Now, from  \eqref{aaa}, we get
$$\vert\rho_t(u)- \<\rho_t,\overrightarrow\iota_\ve^u\>\vert\,\leq\,\Vert \partial_u(\rho_{t})^2\Vert_2\tfrac{2}{3}\tfrac{\ve^{1/2}}{a}+a,$$
for all $u\in [0,1]$, for almost $t\in [0,T]$ and for all $a>0$.
Therefore, it is enough to choose $a=\ve^{1/4}$ to conclude the proof.
\end{proof}

\section{Uniqueness of weak solutions}
\label{uniqueness}
In this section we prove Lemma \ref{lem:uniquess}, that is, the uniqueness of weak solutions of the hydrodynamic equations defined in Section \ref{s2}. We start covering the Dirichlet case, in which we use the Oleinik's trick, and we finish the section presenting the uniqueness for the Robin case.  We remark that both methods  presented below, cover the Neumann case. We decided to include a brief description at the end of the proof for the Dirichlet case stating what would be the  differences for the Neumann case. 

Before presenting the proofs suppose that $\rho_{1}(t,u)$ and $\rho_{2}(t,u)$ are weak solutions of the PME starting from the same initial condition $g(\cdot)$ and with  suitable boundary conditions for each problem. We stress that throughout this section we will denote $w_t(u) = \rho_1(t,u)-\rho_2(t,u)$ and  $v_t(u) =\rho_1(t,u)+\rho_2(t,u)$, for $(t,u)\in [0,T]\times[0,1]$.

\subsection{The Dirichlet and Neumann cases}
\label{DirichletAndNeumannCasesSubsec}
Suppose that  $\rho_{1}(t,u)$ and  $\rho_{2}(t,u)$ are weak solutions of \eqref{eq:Dirichlet} starting from the same initial condition $g(\cdot)$. Doing an integration by parts in \eqref{eq:Dirichlet}, we have that 
\begin{equation}\label{fracaa}
\begin{split}
  \< w_T,G_T \> &+ \int_{0}^{T}  {\< \partial_{u}\rho_{1}^{2}(t,\cdot)-\partial_{u}\rho_{2}^{2}(t,\cdot),\partial_{u}G_t \>} \, dt - \int_{0}^{T} \< w_t, \partial_t G_t \> \,dt =0,
\end{split}
\end{equation}
for all $G\in C_{0}^{1,2}([0,T]\times[0,1])$. Observe that the left-hand side of this identity is well defined even if we assume only that $G \in L^2(0,T;\mathcal{H}^{1}_{0})$ and $\partial_t G \in L^2(0,T;L^2[0,1])$. In fact, by mollifying such $G$ we can approximate it by smooth functions $G_k \in C_{0}^{1,2}([0,T]\times[0,1])$ and, using a limit argument, conclude that \eqref{fracaa} holds for $G$, since it holds for $G_k$. We leave the details to the reader and we refer to \cite{vazquez_book} for more details.

Now we consider the function $\zeta \in L^2(0,T;\mathcal{H}^{1}_{0})$ such that $\partial_t \zeta \in L^2(0,T;L^2[0,1])$ given by 
\begin{equation*}
\zeta(t,u)=
\begin{cases}
\int_{t}^{T} w_s(u)v_s(u) \, ds\,, & \textrm{ if } 0<t<T\,,\\
0\,, & \textrm{ if } t \geq T\,,\\
\end{cases}
\end{equation*}
where  $T>0$. Note that $\zeta(t,0)=\zeta(t,1)=0$ for all $t\in[0,T]$, comes from the fact that $\rho_1(t,u)$ and $\rho_2(t,u)$ satisfy item $(3)$ of Definition \ref{Def. Dirichlet}.

%


Observe that 
\begin{equation}\label{zeta}
\begin{split}
\partial_{t}\zeta(t,u) =& \,\,-w_t(u)v_t(u)\in L^{2}([0,T]\times [0,1]), \\ \
\partial_{u}\zeta(t,u) =& \,\,\int_{t}^{T}\big(\partial_{u}\rho_{1}^{2}(s,u)-\partial_{u}\rho_{2}^{2}(s,u)\big) \, ds\in L^{2}([0,T]\times [0,1]). 
\end{split}
\end{equation}
Replacing $G$ by $\zeta$ in \eqref{fracaa}, we have
\begin{equation*}
\begin{split}
\int_{0}^{T} \<  \partial_{u}\rho_{1}^{2}(t,\cdot)-\partial_{u}\rho_{2}^{2}(t,\cdot), \partial_{u}\zeta_t \> \, dt - \int_{0}^{T}   \< w_t, \partial_{t}\zeta_t \> \, dt = 0.
\end{split}
\end{equation*}
Using \eqref{zeta} it follows that
\begin{equation*}
\begin{split}
\int_{0}^{1}\int_{0}^{T}& \,\, \Bigg\{ w^2_t(u)v_t(u) +  \big(\partial_{u}\rho_{1}^{2}(t,u)-\partial_{u}\rho_{2}^{2}(t,u)\big)\left(\int_{t}^{T}(\partial_{u}\rho_{1}^{2}(s,u)-\partial_{u}\rho_{2}^{2}(s,u)\,) \, ds \right) \Bigg\} \, dt \, du =0,
\end{split}
\end{equation*}
that is
\begin{equation*}
\begin{split}
\int_{0}^{T} \,\, \< w_t,w_t v_t \> \, dt 
+ \,\, \frac{1}{2}\int_{0}^{1}\left( \int_{0}^{T}(\partial_{u}\rho_{1}^{2}(t,u)-\partial_{u}\rho_{2}^{2}(t,u))\, dt \right)^{2} \, du =0.
\end{split}
\end{equation*} 
From last identity, we conclude that $\rho_{1}(t,u)=\rho_{2}(t,u)$ a.e. in $[0,T]\times [0,1]$.

Now, we remark that the proof above also shows uniqueness in the Neumann case. The only difference with respect to the proof above is that we do not need to require the profile $\rho(\cdot)$ to have a fixed value at the boundary. We give now a sketch of the proof in this case.  Suppose  that $\rho_{1}(t,u)$ and $\rho_{2}(t,u)$ are now weak solutions of \eqref{eq:Robin integral} with $\kappa =0$, starting from the same initial condition $g(\cdot)$. Doing an integration by parts in \eqref{eq:Robin integral} with $\kappa =0$ we have that,
\begin{equation*}
\begin{split}
\< w_T,G_T \> + \int_{0}^{T}  \< \partial_{u}\rho_{1}^{2}(t,\cdot)-\partial_{u}\rho_{2}^{2}(t,\cdot), \partial_{u}G_t \> \, dt  - \int_{0}^{T} \< w_t, \partial_{t}G_t \>\, dt =0,
\end{split}
\end{equation*}
for all $G\in C^{1,2}([0,T]\times[0,1])$. Note that the last equation is exactly the same as in \eqref{fracaa}. Now, by the same arguments used in the Dirichlet case, we can reach the same conclusion for the Neumann case.

\subsection{The Robin case}
We adapt Filo's proof to our model (see \cite{filo}, Theorem 3), and we present it in details below. Although the proof there holds for any spatial dimension, we consider only the one-dimensional case. Before starting the proof, we need some technical results. The following result is concerning a parabolic value problem with Robin conditions:

\begin{lemma}
	Suppose that $a=a(t,u)$ is a positive $C^{2,2}([0,T] \times [0,1])$ function, $b=b(t,u)$ is a positive $C^2([0,T])$ function, for $u=0$ and $u=1$, $h=h(u) \in C^2_0([0,1])$, and $\lambda \ge 0$. Then, for $t \in (0,T]$, the problem with Robin conditions 
	\begin{equation}\label{parabolicEigenvalueRobinProb}
	\begin{cases}
	&\partial_s \varphi + a \Delta \, \varphi  = \lambda \varphi, \quad {\rm for} \quad (s,u) \in [0,t) \times (0,1), \\[2pt]
	&\partial_u \varphi(s,0) = b(s,0) \, \varphi(s,0), \quad {\rm for } \quad s \in [0,t), \\[2pt]
	&\partial_u \varphi(s,1) = - b(s,1)\,  \varphi(s,1), \quad {\rm for } \quad s \in [0,t), \\[2pt]
	&\varphi(t,u)  =  h(u), \quad {\rm for } \quad u \in (0,1),
	\end{cases}
	\end{equation}
	has a unique solution $\varphi_0$ in $C^{1,2}([0,t] \times [0,1])$. Moreover, if $0 \le h \le 1$, then
	\begin{equation}
	0 \le \varphi_0(s,u) \le e^{-\lambda(t-s)}, \quad {\rm for} \quad (s,u)\in[0,t]\times[0,1].
	\label{ExponentialBoundForSolutionOfRobinLinear}
	\end{equation} 
	\label{existenceOfSolutionToRobinBoundarySmoothProblem} 
\end{lemma} 

\begin{proof}
	First, observe that by setting $\tau=t-s$ and $\zeta(\tau,u) = e^{-\lambda (t-\tau)} \varphi(t-\tau,u)$, \eqref{parabolicEigenvalueRobinProb} is equivalent to
	\begin{equation}\label{parabolicNonEigenvalueRobinProb}
	\begin{cases}
    &\partial_{\tau} \zeta - a \Delta \, \zeta = 0, \quad {\rm for} \quad (\tau,u) \in (0,t] \times (0,1), \\[2pt]
	&\partial_u \zeta(\tau,0) = b(t-\tau,0) \, \zeta(\tau,0), \quad {\rm for } \quad \tau \in (0,t], \\[2pt]
	&\partial_u \zeta(\tau,1) = - b(t-\tau,1)\,  \zeta(\tau,1), \quad {\rm for } \quad \tau \in (0,t], \\[2pt]
	&\zeta(0,u)  =  e^{-\lambda t}h(u), \quad {\rm for } \quad u \in (0,1),
	\end{cases}
	\end{equation}
	which has a unique $C^{1,2}([0,t] \times [0,1])$ solution $\zeta_0(\tau,u)$ according to \cite{LSU} (see Theorem 5.3) or \cite{lieberman} (see Theorem 4). Now, we need to show that $0 \le \zeta_0 \le e^{-\lambda t}$ in $[0,t] \times [0,1]$, under the assumption that $0 \le h \le 1$. Suppose that 
	$$\max_{[0,t] \times [0,1]} \zeta_0 > e^{-\lambda t}.$$ 
	From the maximum principle for parabolic equations, 
	$$  \max_{[0,t] \times [0,1]} \zeta_0 = \max_{ \Sigma_t \cup (\{0\} \times [0,1] ) } \zeta_0,$$ 
	where $\Sigma_t = ([0,t] \times \{0\} ) \cup ([0,t] \times \{1\} )$. Since $\zeta_0(0, u) = e^{-\lambda t}h(u) \le e^{-\lambda t}$, for $0\le u \le 1$, there exists some $(\tau_1,u_1) \in \Sigma_t$
	that realizes the maximum of $\zeta_0$. Suppose, without loss of generality, that $u_1=0$. Observe that $\tau_1 > 0$, due to the fact that $\zeta_0$ is continuous in $[0,t] \times [0,1]$ and $\zeta_0(0,0)=e^{-\lambda t} h(0)=0$. Since $\zeta_0(\tau_1,u_1) > e^{-\lambda t}$ and $b$ is positive, it follows that
	$$ \partial_u \zeta_0(\tau_1,0)  =  b(t-\tau_1,0) \, \zeta_0(\tau_1,0) > 0.$$
	Hence, for $u > 0$ sufficiently close to $0$, we have
	$$ \zeta_0(\tau_1,u) > \zeta_0(\tau_1,0),$$
	contradicting the fact that $(\tau_1,0)$ is a point of maximum of $\zeta_0$. Therefore, $\zeta_0 \le e^{-\lambda t}$.
	By an analogous argument, we can prove that $\zeta_0 \ge 0$, concluding that $0 \le \zeta_0 \le e^{-\lambda t}$.
	
	Now, let $\varphi_0(s,u) = e^{\lambda s}\zeta_0(t-s,u)$. As we have already mentioned, since $\zeta_0$ is the solution of \eqref{parabolicNonEigenvalueRobinProb}, then $\varphi_0$ is the solution of \eqref{parabolicEigenvalueRobinProb}. Furthermore, since $0 \le \zeta_0 \le e^{-\lambda t}$, we have that $0 \le \varphi_0(s,u) \le e^{-\lambda(t-s)}$, which proves the lemma.
\end{proof}

\begin{lemma} Let $\varphi_0$ be the solution of the parabolic problem \eqref{parabolicEigenvalueRobinProb}.
	There exists a positive constant $C=C(b,h)$ such that
	$$ \int_0^t \int_0^1 a(s,u) (\Delta \varphi_0(s,u))^2 \, du ds \le C(b,h).$$
	\label{upperBoundForTheSecondDerivativeOfSmoothSol} 
\end{lemma}
\begin{proof} Multiplying the first line of \eqref{parabolicEigenvalueRobinProb} by $\Delta \varphi_0(s,u)$, and integrating it in space and time, we obtain
	$$ \int_0^t \int_0^1 \partial_s  \varphi_0 \, \Delta \varphi_0  \, du \, ds \, + \int_0^t \int_0^1 a (\Delta \, \varphi_0)^2  \, du \, ds  \, -  \int_0^t \int_0^1 \lambda \varphi_0 \, \Delta \varphi_0  \, du ds = 0 . $$
	Integrating last equation by parts, we have
	\begin{align*}
	& \int_0^t  \partial_s  \varphi_0(s,1) \, \partial_{u} \varphi_0(s,1) \, ds - \int_0^t  \partial_s  \varphi_0(s,0) \, \partial_{u} \varphi_0(s,0) \, ds \\[5pt]
	&- \frac{1}{2} \int_0^t \int_0^1 \partial_s  |\partial_{u} \varphi_0 |^2 \, du ds + \int_0^t \int_0^1 a (\Delta \, \varphi_0)^2  \, du \, ds \\[5pt]
	&  -  \int_0^t  \lambda \varphi_0(s,1) \, \partial_{u} \varphi_0(s,1)  \, ds +  \int_0^t  \lambda \varphi_0(s,0) \, \partial_{u} \varphi_0(s,0)  \, ds +  \int_0^t \int_0^1 \lambda  \, |\partial_{u} \varphi_0|^2  \, du ds = 0\,.
	\end{align*}
	Integrating the third term in the last equation and using the boundary conditions, it follows that
	\begin{align*}
	& \int_0^t \int_0^1 \big(a (\Delta \, \varphi_0)^2 + \lambda  \, |\partial_{u} \varphi_0|^2\big)  \, du \, ds  +  \int_0^t  \lambda b(s,1) (\varphi_0(s,1))^2 \,  ds +  \int_0^t  \lambda b(s,0)(\varphi_0(s,0))^2   \, ds \\[5pt] &
	- \int_0^t  \partial_s  \varphi_0(s,1) \, b(s,1) \varphi_0(s,1) \, ds  - \int_0^t  \partial_s  \varphi_0(s,0) \, b(s,0) \varphi_0(s,0) \, ds \\[5pt]
	&- \frac{1}{2}  \int_0^1  |\partial_{u} \varphi_0 |^2(t,u) -  |\partial_{u} \varphi_0 |^2(0,u) \, du  = 0\,.
	\end{align*}
	Now, doing an integration by parts on the fourth and fifth terms in the above display, and using the initial condition, we obtain:
	\begin{align*}
	& \int_0^t \int_0^1\big( a (\Delta \, \varphi_0)^2 + \lambda  \, |\partial_{u} \varphi_0|^2 \big) \, du \, ds  +  \int_0^t  \lambda b(s,1) (\varphi_0(s,1))^2 \,  ds +  \int_0^t  \lambda b(s,0)(\varphi_0(s,0))^2   \, ds \\[5pt] 
	&- \frac{1}{2} b(t,1) (\varphi_0(t,1))^2  + \frac{1}{2} b(0,1) (\varphi_0(0,1))^2 + \frac{1}{2} \int_0^t  \partial_s  b(s,1) (\varphi_0(s,1))^2 \, ds   \\[5pt]
	&- \frac{1}{2} b(t,0) (\varphi_0(t,0))^2  + \frac{1}{2} b(0,0) (\varphi_0(0,0))^2 + \frac{1}{2} \int_0^t  \partial_s  b(s,0) (\varphi_0(s,0))^2 \, ds   \\[5pt]
	&- \frac{1}{2}  \int_0^1  |h'(u) |^2 \, du  + \frac{1}{2}  \int_0^1  |\partial_{u} \varphi_0 |^2(0,u) \, du  = 0\,.
	\end{align*}
	Therefore,
	\begin{align*}
	\int_0^t \int_0^1 a (\Delta \, \varphi_0)^2  \, du \, ds &\le  \frac{1}{2}  \int_0^1  |h'(u) |^2 \, du \\[5pt] 
	&+ \frac{1}{2} b(t,1) (\varphi_0(t,1))^2  - \frac{1}{2} b(0,1) (\varphi_0(0,1))^2 - \frac{1}{2} \int_0^t  \partial_s  b(s,1) (\varphi_0(s,1))^2 \, ds   \\[5pt]
	&+ \frac{1}{2} b(t,0) (\varphi_0(t,0))^2  - \frac{1}{2} b(0,0) (\varphi_0(0,0))^2 - \frac{1}{2} \int_0^t  \partial_s  b(s,0) (\varphi_0(s,0))^2 \, ds\, . 
	\end{align*}
	Since $\varphi_0$ is bounded, according to Lemma \ref{existenceOfSolutionToRobinBoundarySmoothProblem}, the right-hand side of last inequality is bounded from above by some constant $C$, that depends only on $h$ and $b$.
\end{proof}

Before presenting the uniqueness of weak solutions of the hydrodynamic equation with Robin boundary conditions, we need two more technical results:
\begin{lemma}
	Let $b$ be a nonnegative and bounded measurable function in $[0,T]$ and $1 \le p < +\infty$. There exists a sequence $\{b_k\}_{k\geq 0}$ of positive functions in $C^{\infty}([0,T])$, such that $b_k$ converges to $b$ in $L^p([0,T])$ and
	$$ \left\| \frac{b}{b_k} - 1 \right\|_{L^p(A)} \to 0 ,$$
	where $A =\{ t \in (0,T] \, : \, b(t) > 0 \}$.
	\label{L2EstimateForTheDifferenceBetweenFractionAnd1}
\end{lemma}

\begin{proof}
	Let $\varepsilon_k = 1/k > 0$. Consider a sequence of positive numbers $\{\delta_j\}_{j\geq 0}$, such that $\delta_j \to 0$. Since $b > 0$ in $A$, we have
	$$ \frac{b(t)}{b(t) + \delta_j} - 1 \to 0 \quad {\rm for \; any } \; t \in A \; \; {\rm as } \; j \to +\infty, \quad {\rm and } \quad \left| \frac{b(t)}{b(t) + \delta_j} - 1 \right| < 2.$$
	From the dominated convergence theorem, $b/(b+\delta_j) -1$ converges to $0$ in $L^p(A)$. Hence, for a large $j_0$, we have
	\begin{equation}
	\left\| \frac{b}{b + \delta_{j_0}} - 1 \right\|_{L^p(A)} < \frac{\varepsilon_k}{2} .
	\label{L2EstimateForFraction1}
	\end{equation}
	Let $\{c_m\}_{m\geq 0}$ be a sequence in $C^{\infty}([0,T])$, such that $c_m \to b + \delta_{j_0}$ in $L^p([0,T])$. Since $b + \delta_{j_0} \ge \delta_{j_0}$, we can assume that $c_m \ge \delta_{j_0}$. Then
	$$ \left\| \frac{b}{c_m} - \frac{b}{b + \delta_{j_0}} \right\|_{L^p(0,T)} = \left\| \frac{b( b + \delta_{j_0} - c_m)}{c_m (b + \delta_{j_0})} \right\|_{L^p(0,T)} \le \frac{\|b\|_{L^{\infty}([0,T])} \|  b + \delta_{j_0} - c_m \|_{L^{p}([0,T])}}{\delta_{j_0}^2}. $$
	Hence, using that $c_m \to b + \delta_{j_0}$ in $L^p([0,T])$, for a large $m_0$, we have that
	\begin{equation}
	\left\| \frac{b}{c_{m_0}} - \frac{b}{b + \delta_{j_0}} \right\|_{L^p(0,T)} < \frac{\varepsilon_k}{2}.
	\label{L2EstimateForFraction2222}
	\end{equation}
	Defining $b_k = c_{m_0}$, \eqref{L2EstimateForFraction1} and \eqref{L2EstimateForFraction2222}  imply that
	$$ \left\| \frac{b}{b_k} - 1 \right\|_{L^p(A)} < \varepsilon_k, $$
	proving the result. 
\end{proof}

\begin{remark}
	Using the same argument above, we can prove the following result that is used in \cite{filo}:
	if $a$ is a nonnegative bounded measurable function in $[0,T] \times [0,1]$, then there exists
	a sequence $\{a_k\}_{k\geq 0}$ of positive $C^{\infty}$ functions in time and space, such that 
	$$ \frac{1}{k}  \le a_k \le \| a \|_{L^{\infty}} + \frac{1}{k} \quad {\rm and} \quad \left\| \frac{a - a_k}{\sqrt{a_k}} \right\|_{L^2([0,T] \times [0,1])} \to 0 .$$
	\label{L2EstimateForDiffBetweenSquareRootAnd}
\end{remark}

\noindent {\sl Proof of Lemma \ref{lem:uniquess} for the Robin case (\cite{filo}):}
Although the proof that we will present is true for $\kappa \geq 0$, we will only consider the case $\kappa > 0$. But the interested reader can check that for $k=0$, the proof also holds.   Suppose that $\rho_1(t,u)$ and $\rho_2(t,u)$ are weak solutions of \eqref{eq:Robin}. Since $\rho_1(t,u)$ and $\rho_2(t,u)$ satisfy \eqref{eq:Robin integral}, we conclude that
\begin{equation*}
\begin{split}
\< w_t,G_t \>  -\int_{0}^{t} \< w_s, \partial_{s}G_s \> \,ds &- \int_{0}^{t} \< w_s, v_s\Delta G_s \> \, ds +\int^{t}_{0}   w_s(1)v_s(1)\partial_u G_s(1)-w_s(0)v_s(0)\partial_u G_s(0)  \, ds\\
& + \kappa\int_0^t w_s(0)G_s(0) + w_s(1)G_s(1)ds=0.
\end{split}
\end{equation*}
Therefore, this equation can be rewritten as
\begin{equation}
\begin{split}
\< w_t, G_t \> \;& =\,\int_{0}^{t} \< w_s, \partial_{s}G_s + v_s \Delta G_s \> \, ds -\int^{t}_{0}  w_s(1)\,\big( \kappa G_s(1) + v_s(1) \partial_u G_s(1)\big)  \, ds \\
&+\, \int^{t}_{0}   w_s(0)\,\big( v_s(0) \partial_u G_s(0) - \kappa G_s(0)\big)  \, ds \,.
\end{split}
\label{eq:RobinEquationForDifferenceOfOriginal}
\end{equation}
To estimate the integrals above we need to use a suitable test function, which is the solution of the parabolic equation \eqref{parabolicEigenvalueRobinProb}. Unfortunately, the function $v$ above does not have regularity enough. To avoid this difficulty, 
observe that $0 \le v(t,u) \le 2$, since $0 \le \rho_1$ and $\rho_2 \le 1$. 
Then, according to Lemma \ref{L2EstimateForTheDifferenceBetweenFractionAnd1}, taking $b$ equal to $v$, for $\varepsilon > 0$ and $p=1$, there exists a positive function $b_{\varepsilon} \in C^2([0,T] \times \{ 0, 1\})$ such that 
\begin{equation}
\left\| \frac{v(t,u_i)}{b_{\varepsilon}(t,u_i)} - 1 \right\|_{L^1(A_i)} < \varepsilon \quad {\rm for } \quad i\in\{0,1\},
\label{eq:L1ApprochOfaByb}
\end{equation} 
where $u_0=0$, $u_1=1$ and $A_i=\{ t \in (0,T] \; : \; v(t,u_i) > 0 \}$.
Moreover, from Remark \ref{L2EstimateForDiffBetweenSquareRootAnd} with $a=v$, there exists a sequence of functions $\{a_n\}_{n\geq 0}$ in $C^{\infty}$ in time and space, such that
\begin{equation} 
\frac{1}{n} \le a_n \le 2 + \frac{1}{n} \quad {\rm and} \quad \frac{a_n - v}{\sqrt{a_n}} \to 0 \quad {\rm in } \quad L^2([0,T] \times [0,1]) \quad {\rm as} \quad n \to +\infty .
\label{eq:L2approachOfaByan}
\end{equation}
For fixed $\lambda = 0$ and $h\in C^2_0([0,1])$, consider the parabolic problem \eqref{parabolicEigenvalueRobinProb} with $a$ and $b$ replaced by $a_n$ and $\kappa /b_{\varepsilon}$, respectively. Observe that $\kappa /b_{\varepsilon}$ is a positive $C^2$ function.
Then, from Lemma \ref{existenceOfSolutionToRobinBoundarySmoothProblem} there exists a unique solution $\varphi_n(s,u)$ to this problem associated to $a_n$ and $\kappa /b_{\varepsilon}$.

Now, for $G(s,u)=\varphi_n(s,u)$, we estimate each integral of the right-hand side of  \eqref{eq:RobinEquationForDifferenceOfOriginal}. For the first integral, using the fact that $\varphi_n$ is a solution of \eqref{parabolicEigenvalueRobinProb} (with $\lambda=0$), and the Cauchy-Schwarz inequality, we obtain
\begin{align*}
& \int_{0}^{t} \< w_s, \partial_{s}\varphi_n(s,\cdot) + v_s \Delta \varphi_n(s,\cdot) \> \, ds \, \\[5pt] 
&= \,   \int_{0}^{t} \< w_s, \partial_{s}\varphi_n(s,\cdot) + a_n(s,u) \Delta \varphi_n(s,\cdot)\> \, ds  
+\;\int_{0}^{t} \< w_s, ( v_s - a_n(s,\cdot)) \Delta \varphi_n(s,\cdot) \> \, ds    \\[5pt]
&\;\le  \;  \int_{0}^{t} \left\| w_s \, \frac{(v - a_n)}{\sqrt{a_n}} \right\|_{L^2([0,1])} \| \sqrt{a_n} \Delta \varphi_n \|_{L^2([0,1])} \,ds \,. \\[5pt]
\end{align*}
Hence, from Cauchy-Schwarz inequality, \eqref{ExponentialBoundForSolutionOfRobinLinear}, Lemma \ref{upperBoundForTheSecondDerivativeOfSmoothSol}, and $|w_s|=|\rho_1-\rho_2| \le 2$, we have 
\begin{equation}
\int_{0}^{t} \< w_s, \partial_{s}\varphi_n(s,\cdot) + v_s \Delta \varphi_n(s,\cdot) \> \,ds\, \;  \le \;   2  \left\|  \frac{(v - a_n)}{\sqrt{a_n}} \right\|_{L^2([0,T] \times [0,1]) }\, \sqrt{C(\kappa / b_{\varepsilon} , h)}\,.
\label{eq:EstimateForFirstIntegral}
\end{equation}
For the boundary integrals of \eqref{eq:RobinEquationForDifferenceOfOriginal} we use the Robin condition satisfied by $\varphi_n$. For the right-hand side of the boundary ($u_1=1$), we have 
$$ \partial_u \varphi_n(s,1) = - \frac{\kappa}{b_{\varepsilon}(s,1)}\,  \varphi_n(s,1).$$
Then, for $G(s,u)=\varphi_n(s,u)$, the second integral on the right-hand side of \eqref{eq:RobinEquationForDifferenceOfOriginal} becomes
\begin{align*}
\int^{t}_{0}   w_s(1) ( \kappa \varphi_n(s,1) + v_s(1) \partial_u \varphi_n(s,1))  \, ds \,&=\, \int^{t}_{0}   w_s(1) \left( \kappa \varphi_n(s,1) - v_s(1) \frac{\kappa}{b_{\varepsilon}(s,1)}  \varphi_n(s,1)\right)  \, ds\,.
\end{align*}
Note that if $s_0 \not\in A_1^t:=\{ s \in [0,t] \; : \; v_s(1) > 0 \}$, then $\rho_1(s_0,1)=\rho_2(s_0,1)=0$ and, therefore, $w(s_0, 1)=0$.
Hence, from the fact that $|w| \le 2$, and \eqref{ExponentialBoundForSolutionOfRobinLinear} together with the choice $\lambda=0$, we get
\begin{align*}
\left| \int^{t}_{0}   w_s(1) ( \kappa \varphi_n(s,1) + v_s(1) \partial_u \varphi_n(s,1))  \, ds \right| &= \left| \int_{A_1^t}   w_s(1) \left( \kappa \varphi_n(s,1) - v_s(1) \frac{\kappa}{b_{\varepsilon}(s,1)}  \varphi_n(s,1)\right)  \, ds \right| \\[5pt]
&\le 2\kappa \left\| 1 - \frac{v_s(1)}{b_{\varepsilon}(s,1)} \right\|_{L^1(A_1^t)}\,.
\end{align*}
Then, using \eqref{eq:L1ApprochOfaByb} and that $A_1^t \subset A_1$, we have
\begin{equation}
\left| \int^{t}_{0}   w_s(1) ( \kappa \varphi_n(s,1) + v_s(1) \partial_u \varphi_n(s,1))  \, ds \right| \le 2\kappa \varepsilon.
\label{eq:EstimateForSecondIntegral}
\end{equation}
By an analogous argument, we also have
\begin{equation}
\left| \int^{t}_{0}   w_s(0) ( v_s(0) \partial_u \varphi_n(s,0) - \kappa \varphi_n(s,0))  \, ds \right| \le 2\kappa \varepsilon.
\label{eq:EstimateForThirdIntegral}
\end{equation}
Therefore, from the fact that $\varphi_n(t,u) = h(u)$, and from \eqref{eq:RobinEquationForDifferenceOfOriginal}, \eqref{eq:EstimateForFirstIntegral},  \eqref{eq:EstimateForSecondIntegral}, and \eqref{eq:EstimateForThirdIntegral}, we conclude that
\begin{equation*}
\< w_t, h \> \; \le \; 2 \, \left\|  \frac{(v - a_n)}{\sqrt{a_n}} \right\|_{L^2([0,T] \times [0,1]) }  \sqrt{C(\kappa / b_{\varepsilon} , h)} \;\, +\, \; 4\kappa \varepsilon \,.
\end{equation*} 
Taking $n \to +\infty$ and using \eqref{eq:L2approachOfaByan}, it follows that
$$\< w_t, h \> \; \le \;  4\kappa \varepsilon. $$
Since $\varepsilon > 0$ is arbitrary,
$$\< w_t, h \> \; \le \; 0, $$
for any $h \in C^2_0([0,1])$. Now consider a sequence $h_n \in C^2_0([0,1])$ such that $h_n(\cdot) \to \mathbb{1}_{\{u \in [0,1] \; : \; w_t(u) > 0 \} }(t,\cdot)$ in $L^2([0,1])$.
Then, from the last inequality, we obtain
$$ \int_0^1 w^{+}(t,u) \,du \; \le \; 0,  $$
where $w^+=\max\{w,0\}$.
Therefore, for any $t \in [0,T]$,  $\rho_1(t,u) \le \rho_2(t,u)$ for almost every $u\in [0,1]$.
That is, $\rho_1 \le \rho_2$ for almost every $(t,u) \in [0,T] \times [0,1]$. In the same way, $\rho_2 \le \rho_1$ a.e., completing the proof.
\hfill $\square$

\section*{Acknowledgements}
A.N. was supported through a grant ``L'OR\' EAL - ABC - UNESCO Para Mulheres na Ci\^encia''. R.P. thanks  FCT/Portugal for support through the project Lisbon Mathematics PhD (LisMath). This project has received funding from the European Research Council (ERC) under  the European Union's Horizon 2020 research and innovative programme (grant agreement   No 715734).

\end{document}